\documentclass[12pt,twoside]{article}
\usepackage{amsmath,amsfonts,amssymb,amsthm,array,mathdots}
\usepackage{a4,enumitem,indentfirst}
\usepackage{stmaryrd} 

\usepackage{url} 
\usepackage{epsfig, color} 

\graphicspath{{figures/}}

\newtheorem{theo}{Theorem}
\newtheorem{prop}{Proposition}[section]
\newtheorem{conj}[prop]{Conjecture}
\newtheorem{coro}[prop]{Corollary}

\newtheorem{fact}[prop]{Fact}
\newtheorem{lemma}[prop]{Lemma}

\theoremstyle{definition}
\newtheorem{example}[prop]{Example}

\theoremstyle{remark}
\newtheorem{rem}[prop]{Remark}

\newcommand\Frenet[1]{\mathfrak{F}_{#1}}

\newcommand{\Word}{{\mathbf W}}
\newcommand{\Ideal}{{\mathbf I}}
\newcommand{\Filter}{{\mathbf U}}

\newcommand{\conv}{\textrm{convex}}
\newcommand{\nconv}{\textrm{non-convex}}
\newcommand{\tok}{\preceq}

\newcommand{\bfx}{{\mathbf x}}

\newcommand{\Ac}{\operatorname{Ac}}
\newcommand{\Sing}{\operatorname{Sing}}
\newcommand{\sing}{\operatorname{sing}}

\newcommand\SO{\operatorname{SO}}
\newcommand\SL{\operatorname{SL}}
\newcommand\GL{\operatorname{GL}}
\newcommand\Lo{\operatorname{Lo}}
\newcommand\Up{\operatorname{Up}}
\newcommand\slalgebra{\operatorname{\mathfrak{sl}}}
\newcommand\so{\operatorname{\mathfrak{so}}}
\newcommand\lo{\operatorname{\mathfrak{lo}}}
\newcommand\up{\operatorname{\mathfrak{up}}}
\newcommand\Spin{\operatorname{Spin}}
\newcommand\inv{\operatorname{inv}}
\newcommand\Inv{\operatorname{Inv}}
\newcommand\Diag{\operatorname{Diag}}
\newcommand\diag{\operatorname{diag}}
\newcommand\Quat{\operatorname{CG}}
\newcommand\jacobi{\lambda}

\newcommand{\resultant}{\operatorname{resultant}}
\newcommand{\discriminant}{\operatorname{discriminant}}
\newcommand{\longacute}{\operatorname{acute}}
\newcommand{\longgrave}{\operatorname{grave}}

\newcommand{\longhat}{\operatorname{hat}}
\newcommand{\chop}{\operatorname{chop}}
\newcommand{\adv}{\operatorname{adv}}
\newcommand{\iti}{\operatorname{iti}}
\newcommand{\pathiti}{\operatorname{path}}
\newcommand{\Pathiti}{\operatorname{Path}}
\newcommand{\swminor}{\operatorname{swminor}}

\newcommand{\ba}{{\mathbf{a}}}

\newcommand{\nmesmo}{\llbracket  n \rrbracket}
\newcommand{\nmaisum}{\llbracket n+1 \rrbracket}

\newcommand{\transpose}{\top}

\newcommand{\Xdim}{\operatorname{Xdim}}
\newcommand{\qord}{\operatorname{ord}}

\newcommand{\mult}{\operatorname{mult}}
\newcommand{\bump}{\operatorname{bump}}

\newcommand{\NN}{{\mathbb{N}}}
\newcommand{\ZZ}{{\mathbb{Z}}}

\newcommand{\RR}{{\mathbb{R}}}

\newcommand{\Ss}{{\mathbb{S}}}
\newcommand{\BB}{{\mathbb{B}}}
\newcommand{\DD}{{\mathbb{D}}}

\newcommand{\HH}{{\mathbb{H}}}

\newcommand{\cH}{{\cal H}}
\newcommand{\cL}{{\cal L}}

\newcommand{\cU}{{\cal U}}

\newcommand{\cA}{{\cal A}}
\newcommand{\cB}{{\cal B}}

\newcommand{\cD}{{\cal D}}
\newcommand{\cP}{{\cal P}}

\newcommand{\cW}{{\cal W}}

\newcommand{\bi}{{\mathbf{i}}}
\newcommand{\bj}{{\mathbf{j}}}
\newcommand{\bL}{{\mathbf{L}}}
\newcommand{\bQ}{{\mathbf{Q}}}
\newcommand{\fF}{{\mathfrak F}}

\newcommand{\fa}{{\mathfrak a}}

\newcommand{\fh}{{\mathfrak h}}
\newcommand{\fu}{{\mathfrak u}}
\newcommand{\fl}{{\mathfrak l}}
\newcommand{\fn}{{\mathfrak n}}

\newcommand{\Pos}{\operatorname{Pos}}
\newcommand{\Neg}{\operatorname{Neg}}
\newcommand{\Bru}{\operatorname{Bru}}
\newcommand{\cLjojo}{\cL_n^{\diamond}}
\newcommand{\Brujojo}{\operatorname{Bru}_{\acute\eta}^{\diamond}}
\newcommand{\Bruadv}{\operatorname{Bru}_{\acute\eta}^{0}}
\newcommand{\Bruchop}{\operatorname{Bru}_{\acute\eta}^{1}}

\begin{document}

\title{Combinatorialization of spaces of nondegenerate spherical curves}
\author{Victor Goulart \\ \texttt{goulart@mat.puc-rio.br}
\and Nicolau C. Saldanha \\ \texttt{saldanha@puc-rio.br}
\and \small{Mathematics Department} \\  \small{PUC-Rio, Brazil}}
\date{\today}
\maketitle

\begin{abstract}

\scriptsize{A parametric curve $\gamma$ of 
class $C^n$ on the $n$-sphere is said to be 
nondegenerate (or locally convex) when 
$\det\left(\gamma(t),\gamma'(t),\cdots,\gamma^{(n)}(t)\right)>0$
for all values of the parameter $t$.
We orthogonalize this ordered basis to obtain
the \emph{Frenet frame} $\Frenet{\gamma}$ of
$\gamma$ assuming values in the orthogonal 
group $\SO_{n+1}$ 
(or its universal double cover, $\Spin_{n+1}$), 
which we decompose into Schubert or Bruhat cells. 
To each nondegenerate curve $\gamma$ we assign 
its itinerary: a word $w$ in the 
alphabet $S_{n+1}\smallsetminus\{e\}$ that 
encodes the succession of non open Schubert cells 
pierced by the complete flag of $\RR^{n+1}$ 
spanned by the columns of $\Frenet{\gamma}$. 
Without loss of generality, we can focus on  
nondegenerate curves with initial and final flags
both fixed at the (non oriented) standard complete flag. 
For such curves, given a word $w$, the subspace of curves 
following the itinerary $w$ is a contractible 
globally collared topological submanifold of finite codimension. 
By a construction reminiscent of Poincar\'e duality, 
we define abstract cell complexes mapped 
into the original space of curves by 
weak homotopy equivalences. 
The gluing instructions come from a partial order in the 
set of words. 
The main aim of this construction is to attempt to determine 
the homotopy type of spaces of nondegenerate curves for $n>2$.   
The reader may want to contrast 
the present paper's combinatorial approach 
with the geometry-flavoured methods of 
previous works.}

\medskip 

\end{abstract}

\section{Introduction}
\label{section:intro}

For a fixed positive integer $n\ge 2$, 
consider the $n$-sphere $\Ss^{n}\subset\RR^{n+1}$ 
with respect to the usual Euclidean metric 
of $\RR^{n+1}$. 
A parametric curve $\gamma:[0,1]\to\Ss^{n}$ 
of class $C^{n}$ is said to be \emph{nondegenerate} 
\cite{Khesin-Ovsienko, Khesin-Shapiro1, Khesin-Shapiro2, Little, Shapiro-Shapiro, Shapiro} or \emph{locally convex} 
\cite{Alves-Saldanha, Saldanha1, Saldanha2, Saldanha3, Saldanha-Shapiro} 
if and only if its derivatives up to 
$n^{\text{th}}$ order $\gamma(t),\gamma'(t),\cdots,\gamma^{(n)}(t)$ 
span a complete flag 
$\langle\gamma(t),\gamma'(t),\cdots,\gamma^{(n)}(t)\rangle$ 
of $\RR^{n+1}$ for all $t\in[0,1]$. 
Without loss of generality, we shall consider 
only \emph{positive} nondegenerate curves, 
\textit{i.e.}, those satisfying 
$\det\left(\gamma(t),\gamma'(t),\cdots,\gamma^{(n)}(t)\right)>0$.


Consider the \emph{Frenet frame} of a 
nondegenerate curve $\gamma:[0,1]\to\Ss^{n}$ 
at time $t\in[0,1]$ to be the orthogonal matrix 
$\Frenet{\gamma}(t)\in\SO_{n+1}$ obtained 
by Gram-Schmidt orthonormalization of the ordered basis 
$\left(\gamma(t),\gamma'(t),\cdots,\gamma^{(n)}(t)\right)$. 
We denote by the same symbol $\Frenet{\gamma}$ the 
Frenet curve just defined and its lift to the universal double cover 
of the special orthogonal group, the group $\Spin_{n+1}$. 
We denote by $\cL_n(\cdot)=\cL_n(1,\cdot)$ the resulting 
space of nondegenerate curves $\gamma$ satisfying 
$\Frenet{\gamma}(0)=1$ with the subspace topology 
induced by the standard $C^{n}$-norm of 
$C^{n}([0,1],\RR^{n+1})$. It is easily seen that 
$\cL_n(\cdot)$ is contractible; fixing the final frame 
$\Frenet{\gamma}(1)$ at a definite element  
$z\in\Spin_{n+1}$ produces subspaces $\cL_n(z)=\cL_n(1,z)$. 
Our main problem is to determine 
the homotopy type of such subspaces. The present paper 
provides a recipe for the construction of an abstract cell 
complex weak homotopy equivalent to $\cL_{n}(z)$. 


Consider the subgroup $\Diag^{+}_{n+1}\subset\SO_{n+1}$ 
of diagonal matrices and its lift $\Quat_{n+1}\subset\Spin_{n+1}$. 
The group $\Quat_{n+1}$ is often called the \emph{Clifford group} 
(see \cite{Lawson-Michelson}) and
$\Quat_3=\{\pm 1,\pm\mathbf{i}, \pm \mathbf{j}, \pm\mathbf{k}\}\subset\mathbb{S}^{3}\subset\HH$ is the classical quaternion group.
Define our working space of 
nondegenerate curves as
\begin{equation}
\label{equation:cL}
\cL_{n}=\bigsqcup_{q\in\Quat_{n+1}}\cL_{n}(1,q)\,,
\end{equation}
the space of nondegenerate curves whose flags 
$\langle\gamma(t),\gamma'(t),\cdots,\gamma^{(n)}(t)\rangle$ 
begin and end at the standard complete flag 
$\langle e_{1},e_{2},\cdots,e_{n+1}\rangle$ of $\RR^{n+1}$. 
It is shown in \cite{Saldanha-Shapiro} that for each 
$z\in\Spin_{n+1}$ we can determine explicitly an 
element $q\in\Quat_{n+1}$ such that 
the spaces $\cL_{n}(z)$ and $\cL_{n}(q)$ 
are homeomorphic. Therefore, in order to understand 
all the spaces $\cL_{n}(z)$, one may restrict attention 
to the disjoint union of $2^{n+1}$ spaces 
in Equation \ref{equation:cL}. 
It is worth pointing out that some but not all 
of the remaining spaces exhibit the same 
homotopy type (see historical remarks below 
and final remarks on Section \ref{sect:finalremarks}).

For a number of technical reasons, we shift to a 
slightly relaxed notion of nondegeneracy that replaces 
each space $\cL_{n}(z)$ with a homotopy equivalent 
smooth Hilbert manifold 
(see Appendix \ref{appendix:Hilbert}). 
We deliberately blur the distinction 
between the original space and its Hilbert manifold 
version by using the same notation for both throughout. 
Since we are mainly interested 
in describing homotopy types, this slight 
abuse is mostly harmless.

Let $S_{n+1}$ be the group of permutations of 
$\llbracket n+1\rrbracket=\{1,2,\ldots,n+1\}$. 
For $\sigma\in S_{n+1}$, let $\inv(\sigma)$ 
be the number of inversions of $\sigma$. 
We denote by $\Word_{n}$ the set of finite words in the 
alphabet $S_{n+1}\smallsetminus\{e\}$. 
For $w =(\sigma_1,\ldots,\sigma_\ell)\in\Word_n$, set
$\dim(w)=\dim(\sigma_{1})+\cdots+\dim(\sigma_{\ell})$, 
where $\dim(\sigma)=\inv(\sigma)-1$. 
For $\gamma\in\cL_n$, we define below its 
\emph{itinerary} $w=\iti(\gamma)\in\Word_n$. 
Let $\cL_n[w]$ be the set of curves with itinerary $w$: 
this yields a stratification 
\begin{equation}
\label{equation:stratification}\cL_{n}=\bigsqcup_{w\in\Word_{n}}\cL_{n}[w]
\end{equation}
into topological submanifolds $\cL_{n}[w]$.

\begin{theo}
\label{theo:stratification}
For each $w\in\Word_{n}$, there is a unique 
$q=q_w\in\Quat_{n+1}$ such that $\cL_{n}[w]$ 
is a contractible globally collared submanifold of 
$\cL_{n}(q)$ of codimension $\dim(w)$.
\end{theo}

The map $w\mapsto q_w$ is given a simple formula  
in Lemma \ref{lemma:cLnw}. 

In order to define the itinerary $\iti(\gamma)$, 
recall the Bruhat decomposition 
$$\GL_{n+1}=
\bigsqcup_{\sigma\in S_{n+1}} \Up_{n+1}P_{\sigma}\Up_{n+1}$$
of the real general linear group $\GL_{n+1}=\GL(n+1,\RR)$ 
into double cosets
of the subgroup $\Up_{n+1}$ of invertible 
real upper triangular matrices, indexed by the 
group $S_{n+1}$. 
The permutation matrix $P_{\sigma}\in\GL_{n+1}$ is defined by 
$e_{j}^{\transpose}P_{\sigma}=e_{j^{\sigma}}^{\transpose}$. 
We denote the image of the integer 
$j\in\nmaisum$ under the permutation 
$\sigma$ 
by $j^{\sigma}$. 

The above partition yields a cell decomposition 
of the real complete flag variety 
$\GL_{n+1}/\Up_{n+1}=\{M\cdot\Up_{n+1}\,\vert\, M\in\GL_{n+1}\}$ 
into Schubert cells $\mathcal{C}_{\sigma}$, 
whose dimension is $\inv(\sigma)$. 
The \emph{(strong) Bruhat order} on $S_{n+1}$ 
\cite{Bjorner-Brenti, Humphreys, Verma}  
is defined by $\sigma\leq\rho$ if and only if 
$\overline{\mathcal{C}_{\sigma}}\subseteq\overline{\mathcal{C}_{\rho}}$ 
(with respect to either Zariski or the usual topology). 
The Lie group $\Spin_{n+1}$ is the 
universal cover of the flag variety $\GL_{n+1}/\Up_{n+1}$. 
The \emph{unsigned Bruhat cells} $\Bru_\sigma\subset\Spin_{n+1}$ 
are the $2^{n+1}$-fold covers of the respective Schubert cells
$\mathcal C_\sigma\subset\GL_{n+1}/\Up_{n+1}$; 
we have $\overline{\Bru_{\sigma}}\subseteq
\overline{\Bru_{\rho}}$ if and only if $\sigma\leq\rho$. 
Although these are well known facts 
\cite{Bernstein-Gelfand-Gelfand, Chevalley, Demazure, Konstant, Verma}, 
the aspects of the subject that are going to 
bear on the proofs of Theorems \ref{theo:stratification} and
\ref{theo:CW} shall be derived from scratch 
in the following sections, in a language specially adapted 
to our purposes. 

We shall often regard $S_{n+1}$ as a Coxeter-Weyl 
group of type $\operatorname{A}_{n}$ with transpositions 
$a_{1}=(12),a_{2}=(23),\cdots,a_{n}=(n\, \,n+1)$ 
as generators. The \emph{Coxeter element} 
$\eta$ is the unique 
element with $\inv(\eta)=m=n(n+1)/2$. 
The Bruhat cell $\Bru_{\eta}$ 
is an open dense subset of $\Spin_{n+1}$. 
We show in Lemma \ref{lemma:itinerary} that 
for each $\gamma\in\cL_n$ there are only 
finitely many instants 
\mbox{$0=t_{0}<t_{1}<\cdots<t_{\ell}<t_{\ell+1}=1$} 
at which 
$\Frenet{\gamma}(t)\notin\Bru_{\eta}$. 
Let $\sigma_i$ be such that $\Frenet{\gamma}(t_i)\in\Bru_{\eta\sigma_i}$.
Set $\iti(\gamma)=(\sigma_1,\ldots,\sigma_\ell)\in\Word_n$. 


By a construction similar to Poincar\'e duality, 
we obtain from the stratification in Equation 
\ref{equation:stratification} 
a CW complex $\cD_n$ and a weak homotopy equivalence 
$c:\cD_n\to\cL_n$.

\begin{theo}
\label{theo:CW}

There exists a CW complex $\cD_{n}$ with one cell 
$i_{w}:\DD^{\dim(w)}\to\cD_{n}$ for each word 
$w\in\Word_{n}$ and a compatible family of continuous maps 
$c_w:\DD^{\dim(w)}\to\cL_{n}$. Compatibility means  
that there exists a continuous map $c:\cD_{n}\to\cL_{n}$ 
defined by $c\circ i_w=c_w$ (for each $w\in\Word_n$). 
The map $c:\cD_{n}\to\cL_{n}$ is a weak homotopy equivalence. 
\end{theo}

The sense in which this complex is a dual to the stratification 
is explained in Section \ref{sect:valid}, particularly in the concept 
of valid complexes.
The maps $c_w$ are constructed rather explicitly in Sections 
\ref{sect:transversal} and \ref{sect:valid}. 
The map $c_w$ turns out to intersect $\cL_n[w]$ only at $c_w(0)$, 
in a topologically transversal manner.  
One important concept in this construction is a partial order 
$\tok$ in the set of words $\Word_n$. 
We define $w_{1}\tok w_{2}$ if and only if 
$\overline{\cL_{n}[w_{2}]}\subseteq\overline{\cL_{n}[w_{1}]}$. 
It turns out that the image $\DD^{\dim(w)}\smallsetminus\{0\}$ 
under $c_w$ is contained in strata $\cL_n[w']$ for $w'\tok w$.

This combinatorial approach has already 
allowed further progress on the subject, which we 
expect to cover in a forthcoming paper 
\cite{Alves-Goulart-Saldanha-Shapiro}. A sample may 
be found in the conjectures stated in our final remarks 
(Section \ref{sect:finalremarks}).


The space $\cL_{2}(I)=\cL_{2}(-1)\sqcup\cL_{2}(1)$ 
of closed nondegenerate curves 
was originally studied by J. Little in the 
seventies \cite{Little}, and shown to have three 
connected components: $\cL_{2}(+1)$, containing 
curves with an odd number of transversal 
self-intersections; $\cL_{2,\conv}(-1)$, 
the subspace of simple curves; and 
$\cL_{2,\nconv}(-1)$, 
containing curves with positive even number 
of self-intersections (we will make sense of this 
notation in due course). The works of B. Khesin, 
B. Shapiro and M. Shapiro in the nineties 
$\cite{Khesin-Shapiro2, Shapiro-Shapiro, Shapiro}$ 
extended this result for $n$ and $z\in\Spin_{n+1}$ 
arbitrary, showing that $\cL_{n}(z)$ has one or 
two connected components: one if and only if it does not 
contain convex curves (defined in Appendix 
\ref{appendix:convex}) and two otherwise, one of 
them being the contractible subspace of the said 
convex curves. The former results are reobtained 
in the present paper within our combinatorial framework. 
In \cite{Saldanha-Shapiro} it is shown that each space 
$\cL_n(z)$ is homeomorphic to one of the spaces 
$\cL_n(q)$, $q\in\Quat_{n+1}$; 
also, several among the latter are homeomorphic.
In \cite{Saldanha3} the spaces $\cL_{2}(z)$ were 
completely classified into three homotopy types explicitly described. 
The analogous problem of describing the homotopy 
types of the spaces $\cL_{n}(z)$ for $n>2$ is currently open. 
We hope to contribute to this problem, in particular 
solving it for $n=3$, in a subsequent paper 
\cite{Alves-Goulart-Saldanha-Shapiro} using the 
combinatorial framework proposed in the present 
work (see conjectures in our final remarks, Section 
\ref{sect:finalremarks}). Some partial results for 
$n=3$ were obtained in \cite{Alves-Saldanha} by 
different methods, more akin to those of \cite{Saldanha3}.

One common motivation for the study of such problems 
is the study of linear ordinary differential operators 
\cite{Polya}.
This point of view was the original motivation of B. Khesin 
and B. Shapiro for considering the problem in the 
early nineties 
\cite{Khesin-Ovsienko, Khesin-Shapiro1, Khesin-Shapiro2, Shapiro-Shapiro3}.
The second named author was first led to consider 
this problem while studying the topology and geometry of critical 
sets of nonlinear differential operators with periodic 
coefficients, in a series of works with D. Burghelea 
and C. Tomei 
\cite{Burghelea-Saldanha-Tomei1,Burghelea-Saldanha-Tomei2,Burghelea-Saldanha-Tomei3,Saldanha-Tomei}

In Section \ref{sect:symmetric} we review some basics of the symmetric group: useful representations of a permutation and the poset structures given by the strong and weak Bruhat orders. We also introduce the \emph{multiplicities} of a permutation. 
In Section \ref{sect:sign}, 
we study the hyperoctahedral group $B_{n+1}$ and the double cover 
$\tilde B^{+}_{n+1}$ of $B^{+}_{n+1}=B_{n+1}\cap\SO_{n+1}$. 
We are particularly interested in the maps $\longacute:S_{n+1}\to\tilde B^{+}_{n+1}$ and $\longhat:S_{n+1}\to\Quat_{n+1}$, which will 
come up constantly in our discussion. 
In Section \ref{sect:triangle} we introduce 
triangular systems of coordinates in large open subsets 
$\cU_z$ of the group $\Spin_{n+1}$ and study 
locally convex curves in the nilpotent lower triangular group 
$\Lo^{1}_{n+1}$.
In Section \ref{sect:totallypositive} we recall the important concept 
of total positivity. More generally, we define the subsets 
$\Pos_\sigma, \Neg_\sigma\subset\Lo^{1}_{n+1}$ 
for $\sigma\in S_{n+1}$. 
We review some classical results and prove some useful facts.
In Section \ref{sect:acctriangle}, we consider the related concept of 
accessibility in the triangular group and prove the contractibility 
of certain sets which will come up later. 
In Section \ref{sect:bruhatcell} we prove several useful results 
about the well known stratification of the group $\Spin_{n+1}$ 
in Bruhat cells. We also recall the useful notion of projective transformations.
In Section \ref{sect:chopadvance} we review the chopping operation defined in \cite{Saldanha-Shapiro} and introduce the dual notion of advancing. 
We also study the \emph{singular set} of a locally convex curve $\Gamma$: 
the set $\{t_1<t_2<\cdots<t_\ell\}\subset (0,1)$ of values of $t$ for 
which $\Gamma(t)\notin\Bru_\eta$.
In Section \ref{sect:ac}, we revisit the concept of accessibility, 
now in the spin group. We prove the contractibility of several 
subsets of $\Spin_{n+1}$.
Section \ref{sect:paths} contains the proof of 
Theorem \ref{theo:stratification}, divided in a series of lemmas. 
In Section \ref{sect:transversal} we construct the restriction 
to an open ball around the origin of the cell $c_w$ 
in the complex $\cD_n$.
In Section \ref{sect:multitineraries} we define the 
\emph{multiplicity vector} for a word and study a few examples 
of the restrictions constructed in Section \ref{sect:transversal}.
In Section \ref{sect:poset} we prove a few basic necessary facts 
about the partial order $\tok$ defined in $\Word_n$.
Section \ref{sect:lowersets} contains a few remarks about 
lower and upper subsets of the poset $(\Word_n,\tok)$. 
In Section \ref{sect:valid} we prove Theorem \ref{theo:CW}. 
In Sections \ref{sect:Dn1} and \ref{sect:Dn2} we explicitly construct 
the $1$ and $2$-skeletons of the complex $\cD_n$. 
As a corollary, we obtain the well known classification of 
connected components of $\cL_n$. 
The $2$-skeleton is of course a key ingredient towards proving 
the simple connectivity of connected components of $\cL_n$ 
(see Conjecture \ref{conj:simplyconnected}).
Section \ref{sect:finalremarks} contain our final remarks, 
with an emphasis on results which we expect to prove in a 
forthcoming paper \cite{Alves-Goulart-Saldanha-Shapiro} 
and which are stated here as conjectures.
Appendix \ref{appendix:convex} reviews the notion of convexity 
of spherical curves and studies its relation with the notion of itinerary. 
In Appendix  \ref{appendix:Hilbert} we define precisely the 
smooth Hilbert manifold homeomorphic to $\cL_n$ to which 
Theorems \ref{theo:stratification} and $\ref{theo:CW}$ apply.

This paper is an extended version of the Ph.D. thesis 
\cite{Goulart} of the first author, supervised by the second.
The first named author is grateful to his co-advisor 
Boris Khesin for the warm hospitality during his time 
as a visiting graduate student at the 
Department of Mathematics of the University of Toronto. 
The second author thanks the kind hospitality of Stockholm 
University during his visits. 
Both authors would like to thank Em\'ilia Alves, Boris Khesin, 
Ricardo Leite, Carlos Gustavo Moreira, 
Paul Schweitzer, Boris Shapiro, Carlos Tomei, 
David Torres, Cong Zhou and Pedro Z\"{u}lkhe for helpful conversations 
and gratefully acknowledge the financial support of 
PUC-Rio, CAPES, CNPq and FAPERJ (Brazil).

\section{The symmetric group}
\label{sect:symmetric}

Consider the symmetric group $S_{n+1}$
(acting on $\nmaisum = \{1, 2, \ldots, n+1 \}$)
as the Coxeter-Weyl group $A_n$, i.e.,
use the $n$ generators
$a = a_1 = (12)$, $b = a_2 = (23)$, \dots, $a_n = (n,n+1)$.
For $\sigma \in S_{n+1}$ and $k \in \nmaisum$
we use the notation $k^\sigma$ (rather than $\sigma(k)$)
so that $(k^{\sigma_1})^{\sigma_2} = k^{(\sigma_1\sigma_2)}$.
A permutation can be denoted in many ways:
one common notation is as a list of values
$[1^\sigma\,2^\sigma\cdots n^\sigma\,(n+1)^\sigma]$, 
the so called \emph{complete notation};
another is as a product of the generators above;
for instance, $[ab] = [312] \in S_3$. Notice that 
we enclose between brackets a expression 
of a permutation in terms of Coxeter generators.
We adopt this device to avoid confusion 
between the single permutation with that expression and 
a string of generators with no product intended. Strings of 
permutations are to appear prominently in this paper, 
encoding the so called \emph{itinerary} of a nondegenerate
curve. 

For $\sigma \in S_{n+1}$, let $P_{\sigma}$ be 
the permutation matrix defined by 
$e_k^\transpose P_{\sigma} = e_{k^\sigma}^\transpose$;
for instance, for $n = 2$ we have:
$$ P_a = \begin{pmatrix} 0 & 1 & 0 \\ 1 & 0 & 0 \\ 0 & 0 & 1 \end{pmatrix},
\quad
P_b = \begin{pmatrix} 1 & 0 & 0 \\ 0 & 0 & 1 \\ 0 & 1 & 0 \end{pmatrix}, $$
$$ P_{[ab]} = P_a P_b =
\begin{pmatrix} 0 & 0 & 1 \\ 1 & 0 & 0 \\ 0 & 1 & 0 \end{pmatrix},
\quad
P_{[ba]} = P_b P_a =
\begin{pmatrix} 0 & 1 & 0 \\ 0 & 0 & 1 \\ 1 & 0 & 0 \end{pmatrix}. $$
For $\sigma \in S_{n+1}$, let $\inv(\sigma)$ be the length of $\sigma$
with the generators $a_i$, $1 \le i \le n$
(we reserve the symbol $\ell$ for lengths of the itineraries mentioned above).
Equivalently,
$\inv(\sigma) = |\Inv(\sigma)|$ is the number of inversions of $\sigma$;
the set of inversions is
\[ \Inv(\sigma) = \{ (i,j) \in \nmaisum^2 \;|\;
1 \le i < j \le n+1, \; i^\sigma > j^\sigma \}. \]
A set $I\subseteq\{(i,j)\in \nmaisum^2\;\vert\;i<j\}$ 
is the set of inversions of a permutation $\sigma\in S_{n+1}$ 
if and only if 
\[\forall i<j<k,\quad \left((i,j),(j,k)\in I \to (i,k)\in I\right) \land
\left((i,j),(j,k)\notin I \to (i,k)\notin I\right). \]
Also, if $\rho = \sigma\eta$ then 
$\Inv(\sigma)\sqcup \Inv(\rho) = \Inv(\eta)$.

Let $\Up_{n+1}^{1}, \Lo_{n+1}^{1} \subset \GL_{n+1}$
be the nilpotent triangular groups 
of real upper and lower triangular matrices
with all diagonal entries equal to $1$.
For $\sigma \in S_{n+1}$ consider the subgroups
\begin{equation}
\label{equation:Upsigma}
\begin{aligned}
\Up_\sigma &= \Up_{n+1}^1 \cap (P_\sigma \Lo_{n+1}^{1} P_\sigma^{-1}) \\
& = \{ U \in \Up_{n+1}^1 \;|\; \forall i, j \in \nmaisum,
((i < j, U_{ij} \ne 0) \to ((i,j) \in \Inv(\sigma))\}, \\
\Lo_\sigma &= \Lo_{n+1}^1 \cap (P_\sigma \Up_{n+1}^{1} P_\sigma^{-1}) 
= (\Up_\sigma)^\transpose = P_\sigma \Up_{\sigma^{-1}} P_\sigma^{-1}, 
\end{aligned}
\end{equation}
affine subspaces of dimension $\inv(\sigma)$.
If $\rho = \sigma\eta$ then 
any $L \in \Lo_{n+1}^{1}$ can be written uniquely as
$L = L_1L_2$, $L_1 \in \Lo_\sigma$, $L_2 \in \Lo_\rho$.

A \emph{reduced word} for $\sigma$ is an identity
$$ \sigma = a_{i_1}a_{i_2}\cdots a_{i_k}, \quad k = \inv(\sigma), $$
or, more formally, it is a finite sequence of indices
$(i_1,i_2,\ldots,i_k) \in \nmesmo^k$ 
satisfying the identity above.
Two reduced words for the same permutation $\sigma$
are connected by a finite sequence of local moves of two kinds:
\begin{gather}
\label{equation:reducedword1}
(\cdots,i,j,\cdots) \leftrightarrow (\cdots,j,i,\cdots), \quad
\quad |i-j|\ne 1; \\
\label{equation:reducedword2}
(\cdots,i,i+1,i,\cdots) \leftrightarrow (\cdots,i+1,i,i+1,\cdots); \quad
\end{gather}
corresponding to the identities $a_ia_j = a_ja_i$ for $|i-j| \ne 1$
and $a_i a_{i+1} a_i =  a_{i+1} a_i a_{i+1}$, respectively.
Recall that there exists a unique $\eta \in S_{n+1}$
with $\inv(\eta) = n(n+1)/2$, the Coxeter element
(elsewhere usually denoted by $w_0$: we spare the letter 
$w$ for itineraries, as mentioned above);
write
\[ P_{\eta} =
\begin{pmatrix} & & 1 \\ & \iddots & \\ 1 & & \end{pmatrix}. \]

Recall that the (strong) Bruhat order in the symmetric group $S_{n+1}$
can be defined as follows: $\sigma_0 \le \sigma_1$ if
some reduced word for $\sigma_0$ 
is a substring of some reduced word for $\sigma_1$ 
(a substring here need not have consecutive letters).

Write $\sigma_0 \triangleleft \sigma_1$ if 
$\sigma_0$ is an immediate predecessor of $\sigma_1$ 
(in the Bruhat order).
Recall that $\sigma_0 \triangleleft \sigma_1$
if $\inv(\sigma_1) = \inv(\sigma_0) + 1$ and
$\sigma_1 = \sigma_0 (j_0j_1) = (i_0i_1) \sigma_0$;
here $i_0 < i_1$, $j_0 < j_1$,
$i_0^{\sigma_0} = j_0$, $i_1^{\sigma_0} = j_1$,
$i_0^{\sigma_1} = j_1$, $i_1^{\sigma_1} = j_0$.
We have $\sigma_0 < \sigma_k$ 
(with $k = \inv(\sigma_k) - \inv(\sigma_0)$)
if and only if there exist $\sigma_1, \ldots, \sigma_{k-1}$ with
$\sigma_0 \triangleleft \sigma_1 \triangleleft \cdots
\triangleleft \sigma_{k-1} \triangleleft \sigma_k.$

If $\sigma_1$ is written as $[1^{\sigma_1}\cdots(n+1)^{\sigma_1}]$,
it is easy to find its immediate predecessors:
look for integers $j_1 > j_0$ appearing in the list 
$[1^{\sigma_1}\cdots(n+1)^{\sigma_1}]$, 
$j_1$ to the left of $j_0$,
such that the integers which appear in the list between $j_1$ and $j_0$
are either larger than $j_1$ or smaller than $j_0$;
the permutation $\sigma_0 \triangleleft \sigma_1$ is then obtained
by switching the entries $j_1$ and $j_0$.
In the matrix $P_{\sigma_1}$, we must look for positive entries
$(i_0,j_1)$, $(i_1,j_0)$ such that the interior of the rectangle
with these vertices includes no positive entry. Then 
$P_{\sigma_0}$ is obtained by flipping these entries 
to the diagonal position while leaving the complement 
of the rectangle unchanged.

The strong Bruhat order must not be confused with the left ang right
weak Bruhat orders.
Define the weak left Bruhat order
by taking the transitive closure of:
$\sigma_1 \triangleleft_{L} \sigma_0$ if $\sigma_1 \triangleleft \sigma_0$
and $\sigma_0 = a_i \sigma_1$ (for some $i$).
Equivalently, $\sigma_1 \le_L \sigma_0$ if
$\Inv(\sigma_1^{-1}) \subseteq \Inv(\sigma_0^{-1})$.
Similarly, 
$\sigma_1 \triangleleft_{R} \sigma_0$ if $\sigma_1 \triangleleft \sigma_0$
and $\sigma_0 = \sigma_1 a_i$ (for some $i$);
the transitive closure $\sigma_1 \le_R \sigma_0$
is characterized by $\Inv(\sigma_1) \subseteq \Inv(\sigma_0)$.
Notice that either $\sigma_1 \triangleleft_L \sigma_0$
or $\sigma_1 \triangleleft_R \sigma_0$
imply  $\sigma_1 \triangleleft \sigma_0$;
on the other hand,
$\sigma_1 = [2143] = [a_1a_3] \triangleleft \sigma_0 = [4123] = [a_1a_2a_3]$,
but $\sigma_1 \not\le_L \sigma_0$ and $\sigma_1 \not\le_R \sigma_0$.
For more on Coxeter groups and Bruhat orders, see \cite{Bjorner-Brenti, Humphreys}.


\begin{lemma}
\label{lemma:aij}
Consider $\sigma \in S_{n+1}$ and $i, j \in \nmesmo$
such that $|i-j| > 1$.
Then $\sigma \triangleleft \sigma a_i$
if and only if
$\sigma a_j \triangleleft \sigma a_j a_i = \sigma a_i a_j$.
Similarly, $\sigma \triangleleft \sigma a_j$ if and only if
$\sigma a_i \triangleleft \sigma a_i a_j$.
\end{lemma}

\begin{proof}
The condition 
$\sigma \triangleleft \sigma a_i$
is equivalent to $i^\sigma < (i+1)^\sigma$.
But $i^\sigma = i^{(\sigma a_j)}$
and $(i+1)^\sigma = (i+1)^{(\sigma a_j)}$,
proving the first equivalence.
The second one is similar.
\end{proof}

Define
$$ \sigma_0 \vee e = \sigma_0, \qquad \sigma_0 \vee a_i = \begin{cases}
\sigma_0, & \text{if } \sigma_0 a_i \triangleleft \sigma_0; \\
\sigma_0 a_i, & \text{if } \sigma_0 \triangleleft \sigma_0 a_i. \end{cases} $$
A simple computation verifies that
\begin{gather*}
|i-j| \ne 1 \quad\implies\quad
(\sigma_0 \vee a_i) \vee a_j = (\sigma_0 \vee a_j) \vee a_i; \\
((\sigma_0 \vee a_i) \vee a_{i+1}) \vee a_i =
((\sigma_0 \vee a_{i+1}) \vee a_i) \vee a_{i+1}. 
\end{gather*}
We may therefore recursively define
$$ \sigma_1 \triangleleft \sigma_1 a_i \quad\implies\quad
\sigma_0 \vee (\sigma_1 a_i) = (\sigma_0 \vee \sigma_1) \vee a_i; $$
the previous remarks, together with the connectivity of reduced words
under the moves in Equations \ref{equation:reducedword1} and \ref{equation:reducedword2},
show that this is well defined.
Equivalently, $\sigma_0 \vee \sigma_1$ is the smallest $\sigma$
(in the strong Bruhat order) satisfying both
$\sigma_0 \le_R \sigma$ and $\sigma_1 \le_L \sigma$.
Notice that $S_{n+1}$ is not a lattice with the strong Bruhat order;
the $\vee$ operation above uses more than one partial order.
In general, we may have $\sigma_0 \vee \sigma_1 \ne \sigma_1 \vee \sigma_0$
and $\sigma_0 \vee \sigma_0 \ne \sigma_0$.
We do have associativity:
$(\sigma_0 \vee \sigma_1) \vee \sigma_2 =
\sigma_0 \vee (\sigma_1 \vee \sigma_2)$.

\begin{example}
\label{example:vee}
Take $n = 3$, $\sigma_0 = [2413]$,
$\sigma_1 = [2431] = a_3a_2a_3a_1$.
We have $\sigma_0\vee\sigma_1 =
(((\sigma_0 \vee a_3) \vee a_2) \vee a_3) \vee a_1 =
((\sigma_0 \vee a_2) \vee a_3) \vee a_1 =
(\sigma_0 a_2 \vee a_3) \vee a_1 =
\sigma_0 a_2 a_3 \vee a_1 = \sigma_0 a_2 a_3 a_1 = \eta$.
\end{example}

Another useful representation of a permutation is in terms of its
\emph{multiplicities}, which we now define.
For $\sigma \in S_{n+1}$ and $1 \le k \le n$, let 
$$ \mult_k(\sigma) = \sum_{j\in\llbracket k\rrbracket} (j^\sigma - j), \quad
\mult(\sigma) = \left( \mult_1(\sigma), \mult_2(\sigma), \ldots,
\mult_n(\sigma) \right). $$
With the convention $\mult_0(\sigma) = \mult_{n+1}(\sigma) = 0$,
we have $k^\sigma = k + \mult_k(\sigma) - \mult_{k-1}(\sigma)$,
so that the \emph{multiplicity vector} $\mult(\sigma)$ easily 
determines $\sigma$. The reason for calling $\mult_k(\sigma)$ 
a multiplicity will become clear in section \ref{sect:multitineraries}.

If $d, \tilde d \in \NN^n$ we write $d \le \tilde d$
if, for all $k$, $d_k \le \tilde d_k$.
If $\sigma_0 \le \sigma_1$ (in the Bruhat order)
then $\mult(\sigma_0) \le \mult(\sigma_1)$
and $\inv(\sigma_0) \le \inv(\sigma_1)$.

\begin{example}
For $n = 5$, let $\sigma_0 = [432156]$ and $\sigma_1 = [612345]$.
We have $\mult(\sigma_0) = (3,4,3,0,0) \le \mult(\sigma_1) = (5,4,3,2,1)$
but $\inv(\sigma_0) = 6 > \inv(\sigma_1) = 5$.
For $n = 6$, let $\sigma_2 = [4321567]$ and $\sigma_3 = [7123456]$.
We have $\inv(\sigma_2) = \inv(\sigma_3) = 6$ 
and $\mult(\sigma_2) = (3,4,3,0,0,0) < \mult(\sigma_3) = (6,5,4,3,2,1)$.
\end{example}

\begin{lemma}
\label{lemma:lessdot}
Let $\sigma_0 \triangleleft \sigma_1$
with $\sigma_1 = (i_0i_1) \sigma_0 = \sigma_0 (j_0j_1)$.
Then
$$ \mult_k(\sigma_1) = 
\mult_k(\sigma_0) + (j_1 - j_0)\; [i_0 \le k < i_1]. $$
\end{lemma}

Here we use Iverson notation:
$$  [i_0 \le k < i_1] =
\begin{cases} 1, & i_0 \le k < i_1, \\ 0, & \textrm{otherwise.} \end{cases} $$

\begin{proof}
This is an easy computation.
\end{proof}

The notion of multiplicity is closely related 
to a beautiful 1-1 correspondence, discovered 
by S. Elnitsky \cite{Elnitsky}, between commutation 
classes of reduced words for a permutation 
$\sigma\in S_{n+1}$ and the rhombic tilings of 
a certain (possibly degenerate) $2(n+1)$-gon 
associated to $\sigma$. 
An equivalent (if somewhat deformed) version 
of this construction is obtained by considering 
tesselations by parallelograms of the plane region 
$\mathcal{P}_{\sigma}$ between the graphs 
of $k\mapsto (2\mult_{k}(\sigma_{0})-\mult_{k}(\eta))$ 
and $k\mapsto (-\mult_{k}(\eta))$. These decompositions 
can be performed directly on the graph of 
$k\mapsto\mult_{k}(\sigma)$, as shown in figure 
\ref{figure:multiplicity} below.

\begin{figure}[h!]
\centering
\includegraphics[width=0.3\textwidth]{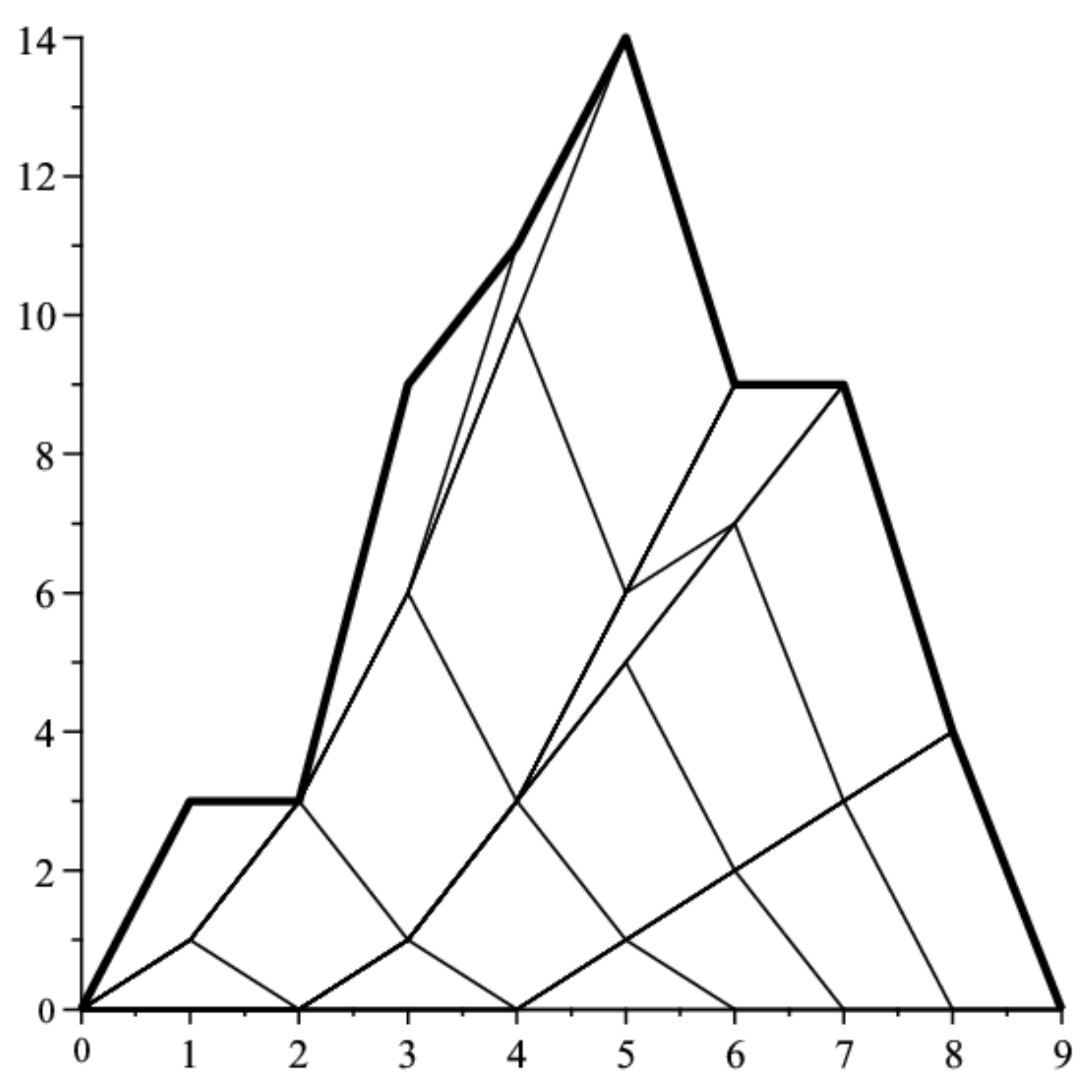}
\includegraphics[width=0.3\textwidth]{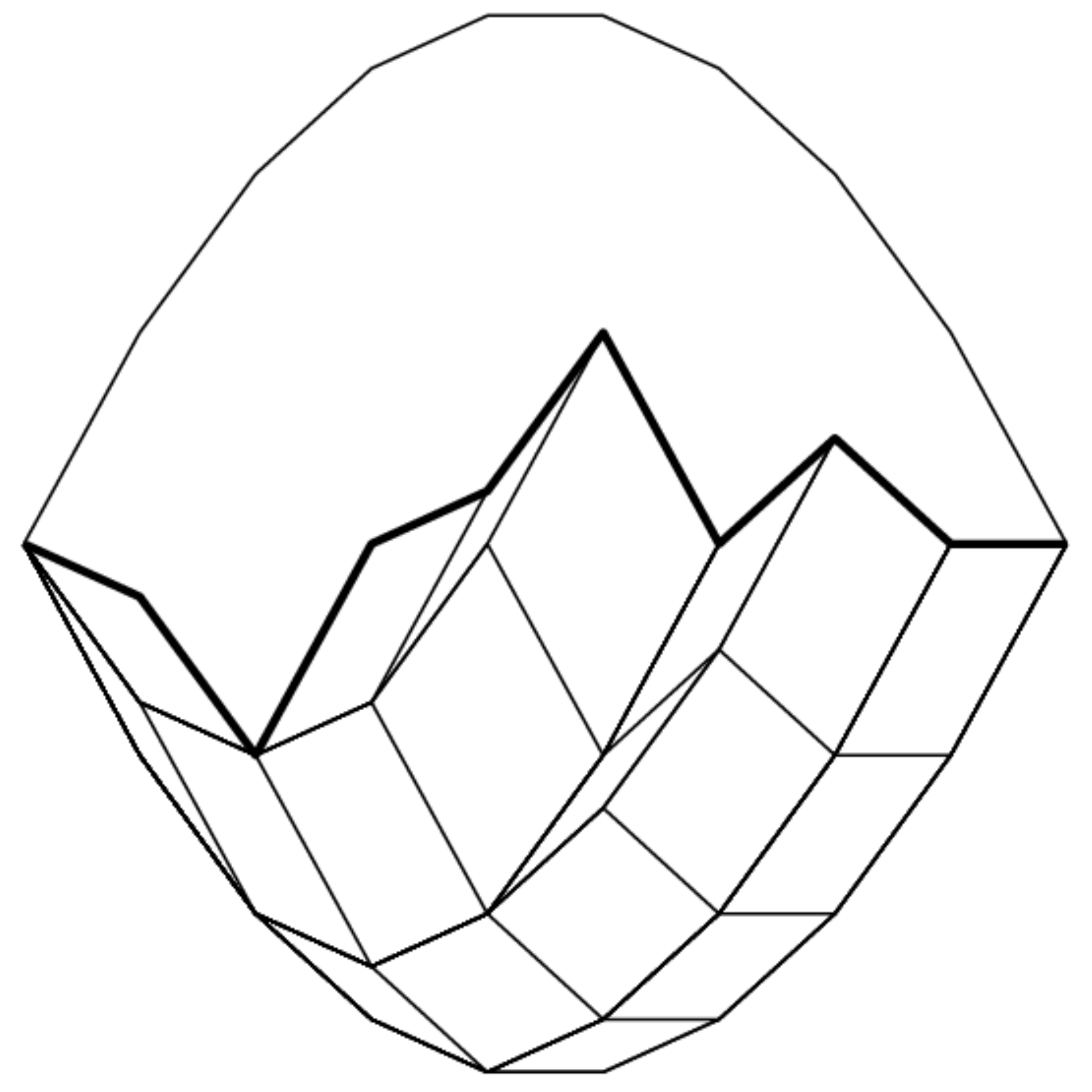}
\caption{Conjugate tilings of the graph of 
$k\mapsto \mult_{k}{\sigma}$ and of the 
Elnitsky's polygon $\mathcal{P}_{\sigma}$ for 
$\sigma=[429681735]\in S_{9}$ corresponding 
to the commutation class of the reduced word 
$\sigma_{0}=a_{1}a_{3}a_{4}a_{5}a_{4}a_{3}a_{2}a_{1}a_{6}a_{5}a_{7}a_{6}a_{5}a_{4}a_{3}a_{8}a_{7}a_{6}a_{5}$.}
\label{figure:multiplicity}
\end{figure}

\section{Signed permutations}
\label{sect:sign}

Let $B_{n+1}$ be the hyperoctahedral group of signed permutation matrices,
i.e., orthogonal matrices $P$ such that there exists a permutation $\sigma$
with $e_k^\transpose P = \pm e_{k^{\sigma}}$ for all $k$.
The group $B_{n+1}$ is a Coxeter group (whence the notation)
but we shall not use this presentation.
Let $B_{n+1}^{+} = B_{n+1} \cap \SO_{n+1}$.
Let $\Diag^{+}_{n+1} \subset B^{+}_{n+1}$
be the normal subgroup of diagonal matrices, 
isomorphic to $\{ \pm 1 \}^n$. 
We have $B^{+}_{n+1}/\Diag^{+}_{n+1}=S_{n+1}$, 
the quotient map being denoted by $P\mapsto\sigma_{P}$.

Consider the universal double covering 
$1\to\{\pm 1\}\hookrightarrow\Spin_{n+1}\stackrel{\Pi}{\longrightarrow}\SO_{n+1}\to 1$ of the special orthogonal group and let 
$\Quat_{n+1} = \Pi^{-1}[\Diag^{+}_{n+1}] \subset \Spin_{n+1}$ 
and 
$\widetilde B^{+}_{n+1} = \Pi^{-1}[B^{+}_{n+1}]$.
The group $\Quat_{n+1}$ has $2^{(n+1)}$ elements
and is a normal subgroup of $\widetilde B^{+}_{n+1}$;
the quotient is again the symmetric group $S_{n+1}$; 
In other words, we have the exact sequences
$$ 1 \to \Quat_{n+1} \to \widetilde B^{+}_{n+1} \to S_{n+1} \to 1; \quad
1 \to \{\pm 1\} \to \Quat_{n+1} \to \Diag^{+}_{n+1} \to 1. $$

Recall that $\Spin_{3}\approx\Ss^{3}$. The \emph{Clifford group} 
(see \cite{Lawson-Michelson})
$\Quat_{n+1}$ is a natural generalization of the classical quaternion group $\operatorname{Quat}_{3}=\{\pm 1,\pm\mathbf{i}, \pm \mathbf{j}, \pm\mathbf{k}\}\subset\mathbb{S}^{3}\subset\HH$.

Let $\fa_i \in \so_{n+1}=\mathfrak{spin}_{n+1}$ be the matrix with only two nonzero entries: 
\begin{equation}
\label{equation:fa}
(\fa_i)_{i+1,i} = 1, \qquad (\fa_i)_{i,i+1} = -1. 
\end{equation}
Consider $\alpha_{i}:\RR\to\Spin_{n+1}$ given by $\alpha_{i}(t)=\exp(t\fa_{i})$ and let $\acute a_i = \alpha_{i}\left(\frac{\pi}{2}\right) \in \widetilde B_{n+1}^{+}$ and $\hat{a_{i}}=\acute{a}_{i}^{2}\in\Quat_{n+1}$, so that $(\hat a_i)^2 = -1$.
The matrix $P = \Pi(\acute a_i) \in B_{n+1}^{+}$ has nonzero entries
$$ P_{i+1,i} = 1, \qquad P_{i,i+1} = -1, \qquad
P_{j,j} = 1, \quad j \notin \{i,i+1\}: $$
$P$ is a rotation of $\frac{\pi}{2}$ in the plane
spanned by $e_i$ and $e_{i+1}$.
Notice that $\sigma_{\Pi(\acute a_i)} = a_i$.


\begin{lemma}
\label{lemma:stepacute}
The following identities hold:
$$ |i-j| \ne 1 \quad \implies \quad
\acute a_j\acute a_i = \acute a_i\acute a_j, \;
\hat a_j \acute a_i = \acute a_i \hat a_j, \;
\hat a_j \hat a_i = \hat a_i \hat a_j; $$
$$ \acute a_i\acute a_{i+1}\acute a_i = \acute a_{i+1}\acute a_i\acute a_{i+1};
\qquad
(\acute a_i)^{-1}\acute a_{i+1}(\acute a_i)^{-1}
= \acute a_{i+1}(\acute a_i)^{-1}\acute a_{i+1}; $$
$$ |i-j| = 1 \quad \implies \quad
\hat a_j \acute a_i =  (\acute a_i)^{-1} \hat a_j, \;
\hat a_j \hat a_i = - \hat a_i \hat a_j. $$
\end{lemma}

\begin{proof}
These are simple computations.
\end{proof}

Each element $q \in \Quat_{n+1}$ can be written
uniquely as
$$ q = \pm \hat a_1^{\varepsilon_1} \hat a_2^{\varepsilon_2} \cdots \hat a_n^{\varepsilon_n}, \qquad
\varepsilon_i \in \{0,1\}. $$
In particular, the elements $\hat a_i$, $1 \le i \le n$,
generate $\Quat_{n+1}$.
Furthermore, if $z \in \widetilde B_{n+1}^{+}$ and
$\sigma_{\Pi(z)}= a_{i_1} \cdots a_{i_k} \in S_{n+1}$,
take $z_1 = \acute a_{i_1} \cdots \acute a_{i_k} \in \widetilde B_{n+1}^{+}$:
we have $\sigma_{\Pi(z)} = \sigma_{\Pi(z_1)}$ and therefore
$z = q z_1$ with $q \in \Quat_{n+1}$.
In particular, the elements $\acute a_i$, $1 \le i \le n$,
generate $\widetilde B_{n+1}^{+}$.
We make this construction more systematic.

\begin{lemma}
\label{lemma:goodacute}
If $\sigma \in S_{n+1}$ is expressed by two reduced words
$\sigma = a_{i_1} \cdots a_{i_k} = a_{j_1} \cdots a_{j_k}$
then
$\acute a_{i_1} \cdots \acute a_{i_k} = \acute a_{j_1} \cdots \acute a_{j_k}$.
\end{lemma}

\begin{proof}
Both moves (as in Equations \ref{equation:reducedword1} and \ref{equation:reducedword2})
are taken care of by Lemma \ref{lemma:stepacute}.
\end{proof}

Let $\grave a_i = (\acute a_i)^{-1}$.
For $\sigma \in S_{n+1}$, take a reduced word
$\sigma = a_{i_1} \cdots a_{i_k}$: set 
$$ \longacute(\sigma) = \acute\sigma = \acute a_{i_1} \cdots \acute a_{i_k};
\qquad
\longgrave(\sigma) = \grave\sigma = \grave a_{i_1} \cdots \grave a_{i_k}. $$
Lemma \ref{lemma:goodacute} shows that the maps 
$\longacute,\longgrave: S_{n+1} \to \widetilde B_{n+1}^{+}$ are well defined.
Notice that these maps are not homomorphisms.
Similarly, non-reduced words do not work
in the above formulas for $\acute\sigma$ and $\grave\sigma$;
$a_1$ has order $2$ but $\acute a_1$ has order $8$.
Also, define 
$$ \longhat(\sigma) = \hat\sigma =
\acute\sigma (\grave\sigma)^{-1} =
\acute a_{i_1} \cdots \acute a_{i_k} \acute a_{i_k} \cdots \acute a_{i_1}, $$
so that $\hat\sigma \in \Quat_{n+1}$ for all $\sigma \in S_{n+1}$.
Notice that these notations are consistent
with the previously introduced special cases $\acute a_i$ and $\hat a_i$.

Let $\inv_i(\sigma) = |\Inv_i(\sigma)|$ where
$$ \Inv_i(\sigma) = \{ j \;|\; i < j, i^\sigma > j^\sigma \}
= \{ j \;|\; (i,j) \in \Inv(\sigma) \}; $$
notice that $\Inv(\sigma) = \bigsqcup_i (\{i\} \times \Inv_i(\sigma))$
and therefore $\inv(\sigma) = \sum_i \inv_i(\sigma)$.

\begin{lemma}
\label{lemma:invpi}
For any $\sigma \in S_{n+1}$ and for any $i \in \nmaisum$
we have 
\[\inv_i(\sigma) - \inv_{i^\sigma}(\sigma^{-1})= i^\sigma - i =\mult_{i+1}(\sigma)-\mult_i(\sigma).\]
\end{lemma}

\begin{proof}
The permutation $\sigma$ restricts to a bijection between the two sets:
\begin{align*}
\{i+1, \ldots, n+1\} \smallsetminus \Inv_i(\sigma) 
&= \{ j \;|\; i < j, i^\sigma < j^\sigma \}, \\
\{i^\sigma+1, \ldots, n+1\} \smallsetminus \Inv_{i^\sigma}(\sigma^{-1}) 
&= \{ j' \;|\; i^{\sigma} < j', i < (j')^{\sigma^{-1}} \},
\end{align*}
with cardinalities
$n+1-i-\inv_i(\sigma)$ and $n+1-i^\sigma-\inv_{i^\sigma}(\sigma^{-1})$.
\end{proof}

\begin{lemma}
\label{lemma:pihat}
Consider $\sigma \in S_{n+1}$ and set $\acute{P}=\Pi(\acute{\sigma})\in B^{+}_{n+1}$. We have
\[e^{\transpose}_{i}\acute{P}=(-1)^{\inv_{i}(\sigma)}e^{\transpose}_{i^{\sigma}}\,, \qquad \acute{P}e_{j}=(-1)^{\inv_{j^{\sigma^{-1}}}(\sigma)}e_{j^{\sigma^{-1}}}\]
and therefore $\acute{P}_{ij}=e^{\transpose}_{i}\acute{P}e_{j}=(-1)^{\inv_{i}(\sigma)}\delta_{ji^{\sigma}}.$


The nonzero entries of $\hat P = \Pi(\hat\sigma) \in \Diag_{n+1}^{+}$ are
$$ (\hat P)_{ii} = (-1)^{\inv_{i}(\sigma) + \inv_{i^\sigma}(\sigma^{-1})}= (-1)^{i + i^\sigma}
= (-1)^{\mult_{i}(\sigma)+\mult_{i+1}(\sigma)}. $$

We also have $ \hat\sigma = \pm \hat a_1^{\mult_1(\sigma)}\cdots
\hat a_i^{\mult_i(\sigma)}\cdots\hat a_n^{\mult_n(\sigma)}$.
\end{lemma}



\begin{proof}
The first expression for the diagonal entries of 
$\hat{P}= \Pi(\acute\sigma)\Pi(\longacute(\sigma^{-1}))$ 
follows directly from the first two formulae, 
which we now prove by induction on $\inv(\sigma)$. 
The base cases $\inv(\sigma) \le 1$ are easy. 
Assume $\sigma_1 \triangleleft \sigma = a_k \sigma_1$, 
so that $\acute{P}=\Pi(\acute{a}_{k})\acute{P}_{1}$, 
where $\acute{P}_{1}=\Pi(\acute{\sigma}_{1})$. 
By the induction hypotheses we have 
\[e^{\transpose}_{i}\acute{P}=
e^{\transpose}_{i}\Pi(\acute{a}_{k})\acute{P}_{1}= 
(-1)^{\inv_{i}(a_{k})}e^{\transpose}_{i^{a_{k}}}\acute{P}_{1} =
(-1)^{\inv_{i}(a_{k})+\inv_{i^{a_{k}}}(\sigma_{1})}e^{\transpose}_{i^{\sigma}}.\]
To see that $\inv_{i}(\sigma)=\inv_{i}(a_{k})+\inv_{i^{a_{k}}}$ for all values of $i\in\nmaisum$, consider separately the cases $i<k$, $i=k$, $i=k+1$ and $i>k+1$. The second formula is similar. The alternate expressions for $\hat{P}_{ii}$ are obtained via Lemma \ref{lemma:invpi}.

Finally, setting $q = \hat a_1^{\mult_1(\sigma)}\cdots
\hat a_i^{\mult_i(\sigma)}\cdots\hat a_n^{\mult_n(\sigma)}$ 
and $E = \Pi(q)$ we have 
$ E_{i,i} = (-1)^{\mult_{i}(\sigma)+\mult_{i+1}(\sigma)} $
and therefore $E = \hat P$, hence $\hat\sigma = \pm q$.
\end{proof}

If $\sigma_1 \triangleleft \sigma_0 = a_i \sigma_1$ 
then, by definition,
$$ \acute\sigma_0 = \acute a_i \acute\sigma_1, \qquad
\hat\sigma_0 = \acute a_i \hat\sigma_1 \acute a_i. $$
We show how to obtain a different recursive formula for $\hat\sigma_0$.

\begin{lemma}
\label{lemma:parityquat}
Let $q \in \Quat_{n+1}$ and $E = \Pi(q) \in \Diag_{n+1}$; write
$$ q = \pm \hat a_1^{\varepsilon_1} \cdots \hat a_n^{\varepsilon_n}, \quad \varepsilon_j \in \ZZ. $$
With the convention $\varepsilon_0 = \varepsilon_{n+1} = 0$, we have:
\begin{enumerate}
\item{If $\varepsilon_{i-1} + \varepsilon_{i+1}$ is odd then
$q \acute a_i = (\acute a_i)^{-1} q$,
$q \hat a_i = - \hat a_i q$, $E_{i+1,i+1} = - E_{i,i}$. }
\item{If $\varepsilon_{i-1} + \varepsilon_{i+1}$ is even then
$q \acute a_i = \acute a_i q$,
$q \hat a_i =  \hat a_i q$, $E_{i+1,i+1} = E_{i,i}$. }
\end{enumerate}
\end{lemma}

\begin{proof}
From Lemma \ref{lemma:stepacute}, 
$$ \hat a_j \acute a_i = \begin{cases}
\acute a_i \hat a_j, & |i-j| \ne 1, \\
(\acute a_i)^{-1} \hat a_j, & |i-j| = 1; \end{cases} \qquad
\hat a_j (\acute a_i)^{-1} = \begin{cases}
(\acute a_i)^{-1} \hat a_j, & |i-j| \ne 1, \\
\acute a_i \hat a_j, & |i-j| = 1; \end{cases} $$
these imply the formulas for $q \acute a_i$.
We then have, for $\varepsilon_{i-1} + \varepsilon_{i+1}$ even,
$$ q \hat a_i = q \acute a_i \acute a_i =  \acute a_i q \acute a_i =
\acute a_i \acute a_i q = \hat a_i q $$
and, for $\varepsilon_{i-1} + \varepsilon_{i+1}$ odd,
$$ q \hat a_i = q \acute a_i \acute a_i =  (\acute a_i)^{-1} q \acute a_i =
(\acute a_i)^{-1} (\acute a_i)^{-1} q = -\hat a_i q. $$
Finally, notice that $\varepsilon_{i-1} + \varepsilon_{i+1}$ even implies
$E \Pi(\acute a_i) = \Pi(\acute a_i) E$ and therefore
$E_{i+1,i+1} = E_{i,i}$;
conversely, $\varepsilon_{i-1} + \varepsilon_{i+1}$ odd implies
$E \Pi(\acute a_i) = \Pi((\acute a_i)^{-1}) E = (\Pi(\acute a_i))^{-1} E$
and therefore $E_{i+1,i+1} = - E_{i,i}$.
\end{proof}

\begin{lemma}
\label{lemma:hatstep}
Let $\sigma_1 \triangleleft \sigma_0 = a_i \sigma_1 = \sigma_1 (j_0 j_1)$, 
$\delta = j_1 - j_0$.
\begin{enumerate}
\item{If $\delta$ is odd then $\hat \sigma_1 \hat a_i   = \hat a_i \hat \sigma_1$
and $\hat \sigma_0 = \hat a_i \hat \sigma_1 = \hat \sigma_1 \hat a_i$.}
\item{If $\delta$ is even then $\hat \sigma_1 \hat a_i = - \hat a_i \hat \sigma_1$
and $\hat \sigma_0 = \hat \sigma_1 = \hat a_i \hat \sigma_1 \hat a_i$.}
\end{enumerate}
\end{lemma}

\begin{proof}
We know (by definition) that
$ \hat\sigma_0 = \acute a_i \hat\sigma_1 \acute a_i $.
As in Lemma \ref{lemma:parityquat}, write
$\hat\sigma_1 = \pm \hat a_1^{\varepsilon_1} \cdots \hat a_n^{\varepsilon_n}$.
We know from Lemma \ref{lemma:pihat}
that we can take $\varepsilon_j = \mult_j(\sigma_{1})$.
Thus
$$ \varepsilon_{i+1} - \varepsilon_{i-1} = (i+1)^{\sigma_{1}} + i^{\sigma_{1}} - (i+1) - i =
j_1 + j_0 -2i-1 \equiv \delta + 1 \pmod 2. $$
If $\delta$ is odd  then
$\hat\sigma_1 \acute a_i = \acute a_i \hat\sigma_1$
and therefore $\hat \sigma_1 \hat a_i = \hat a_i \hat \sigma_1$ and
$\hat \sigma_0 = \hat a_i \hat \sigma_1 = \hat \sigma_1 \hat a_i$. 
If $\delta$ is even then
$\hat\sigma_1 \acute a_i = (\acute a_i)^{-1} \hat\sigma_1$
and therefore $\hat \sigma_1 \hat a_i = - \hat a_i \hat \sigma_1$ and
$\hat \sigma_0 = \hat \sigma_1 = \hat a_i \hat \sigma_1 \hat a_i$. 
\end{proof}

\begin{example}
Using this result it is easy to compute $\hat\sigma_0$ given $\sigma_0$.
Take, say, $\sigma_0 = [7245136] = [a_1a_2a_3a_4a_3a_2a_1a_5a_4a_3a_6]$.
Take
\begin{gather*}
\sigma_0 = a_1 \sigma_1 \triangleright \sigma_1 = [2745136], \qquad
\sigma_1 = a_2 \sigma_2 \triangleright \sigma_2 = [2475136], \\
\sigma_2 = a_3 \sigma_3 \triangleright \sigma_3 = [2457136], \qquad
\sigma_3 = a_4 \sigma_4 \triangleright \sigma_4 = [2451736], \\
\sigma_4 = a_3 \sigma_5 \triangleright \sigma_5 = [2415736], \qquad
\sigma_5 = a_2 \sigma_6 \triangleright \sigma_6 = [2145736], \\
\sigma_6 = a_1 \sigma_7 \triangleright \sigma_7 = [1245736], \qquad
\sigma_7 = a_5 \sigma_8 \triangleright \sigma_8 = [1245376], \\
\sigma_8 = a_4 \sigma_9 \triangleright \sigma_9 = [1243576], \qquad
\sigma_9 = a_3 \sigma_{10} \triangleright \sigma_{10} = [1234576] = a_6.
\end{gather*}
We therefore have
\begin{align*}
\hat \sigma_0 &= \hat a_1 \hat\sigma_1 = 
\hat a_1 \hat a_2 \hat\sigma_2 =
\hat a_1 \hat a_2 \hat\sigma_3 =
\hat a_1 \hat a_2 \hat\sigma_4 =
\hat a_1 \hat a_2 \hat\sigma_5 = \\
&= \hat a_1 \hat a_2 \hat a_2 \hat\sigma_6 =
\hat a_1 \hat a_2 \hat a_2 \hat a_1 \hat\sigma_7 =
\hat\sigma_7 = \hat\sigma_8 = \hat\sigma_9 = \hat a_3 \hat\sigma_{10} = \hat a_3\hat a_6. 
\end{align*}
\end{example}

\begin{example}
\label{example:hateta}
We have that
$\eta = a_1a_2a_1a_3a_2a_1\cdots a_na_{n-1}\cdots a_2a_1$
is a reduced word so that $ \acute\eta = \acute a_1\acute a_2\acute a_1\acute a_3\acute a_2\acute a_1
\cdots \acute a_n\acute a_{n-1}\cdots \acute a_2\acute a_1$
and $\hat\eta = (\acute\eta)^2$.
From Lemma \ref{lemma:pihat}, we have $\Pi(\hat\eta) = (-1)^n\;I$ and
\[ \Pi(\acute\eta) =
\begin{pmatrix} & & & \iddots \\ & & 1 & \\
 & -1 & & \\ 1 & & & \end{pmatrix}, \qquad \hat\eta = \begin{cases}
1, & n \equiv 0, 6 \pmod 8, \\
-1, & n \equiv 2, 4 \pmod 8, \\ 
\hat a_1 \hat a_3 \cdots \hat a_n, & n \equiv 1, 7 \pmod 8, \\
-\hat a_1 \hat a_3 \cdots \hat a_n, & n \equiv 3, 5 \pmod 8.
\end{cases} \]
Notice that, for $n$ odd, we have $\hat a_1 \hat a_3 \cdots \hat a_n=[2143\cdots]$ is in the center of $\Quat_{n+1}$.

\end{example}

\begin{lemma}
\label{lemma:smalln}
$$ n \le 3 \quad\implies\quad \forall \sigma \in S_{n+1}, \;
(\hat\sigma = 1 \to \sigma = e) \land
(\hat\sigma = \hat\eta \to \sigma = \eta). $$
$$ n \ge 4 \quad\implies\quad (\exists \sigma \in S_{n+1}, \;
(\hat\sigma = 1 \land \sigma \ne e)) \land
(\exists \sigma \in S_{n+1}, \;
(\hat\sigma = \hat\eta \land \sigma \ne \eta)). $$
\end{lemma}

\begin{proof}
For $n = 2$ we have $\hat{\sigma}=\pm 1$ if and only if $\sigma\in\{e,\eta\}$, with $\hat{e}=1$ and $\hat\eta = -1$.
For $n = 3$ we have $\hat{\sigma}=\pm 1$ if and only if $\sigma\in\{e,[1432],[3214],[3412]\}$ with $\hat e = 1$,
$\longhat([1432]) = \longhat([3214]) = \longhat([3412]) = -1$. 
Also, $\hat{\sigma}=\pm\hat{\eta}$ if and 
only if $\sigma\in\{[2143],[4123],[2341],\eta\}$, with
$\longhat([2143]) = \longhat([4123]) = 
\longhat([2341]) = \hat a_1 \hat a_3$,
$\hat\eta = -\hat a_1\hat a_3$.
For $n \ge 4$, take
$\sigma_0 = [52341\cdots] = [a_1a_2a_3a_4a_3a_2a_1]$;
we have $\hat\sigma_0 = 1$,
$\longhat(\eta\sigma_0) = \longhat(\sigma_0\eta) = \hat\eta$.
\end{proof}


\section{Triangular coordinates}
\label{sect:triangle}

Let $\Lo_{n+1}^{1}$ be the nilpotent group of real lower triangular matrices
with all diagonal entries equal to $1$;
let $\Up_{n+1}^{+}$ be the group of real upper triangular matrices
with all diagonal entries strictly positive.

Recall the $LU$ decomposition:
a matrix $A \in \GL_{n+1}^{+}$ can be (uniquely) written as $A = LU$,
$L \in \Lo_{n+1}^{1}$ and $U \in \Up_{n+1}^{+}$
provided each of its northwest minors has positive determinant.
This condition holds in a contractible open neighborhood 
of the identity matrix $I$;
for $A$ in this set, $L$ and $U$ are smoothly and uniquely defined.
We shall be more interested in $\cU_I \subset \SO_{n+1}$,
the intersection of this neighborhood with $\SO_{n+1}$,
which is also a contractible open subset.
Let $\bL: \cU_I \to \Lo_{n+1}^{1}$ take $Q \in \cU_I$
to the unique $L = \bL(Q) \in \Lo_{n+1}^{1}$ such that
there exists $U \in \Up_{n+1}^{+}$ with $Q = LU$:
the map $\bL$ is a diffeomorphism.
Indeed, its inverse $\bQ: \Lo_{n+1}^{1} \to \cU_I$
is given by the $QR$ decomposition:
given $L \in \Lo_{n+1}^{1}$ let $Q = \bQ(L) \in \SO_{n+1}$
be the unique matrix for which there exists $R \in \Up_{n+1}^{+}$
with $L = QR$.

For $Q_0 \in \SO_{n+1}$, let $\cU_{Q_{0}} = Q_{0} \cU_I$;
let $\phi_{Q_0}: \cU_{Q_0} \to \Lo_{n+1}^{1}$ take
$Q$ to $L = \bL(Q_0^{-1} Q)$ so that we can write $Q = Q_0LU$,
$U \in \Up_{n+1}^{+}$.
Conversely, let $\psi_{Q_0}: \Lo_{n+1}^{1} \to \cU_{Q_{0}}$,
$\psi_{Q_0}(L) = {Q_0} \bQ(L)$.
We have $\psi_{Q_0} = \phi_{Q_0}^{-1}$;
the family
$(\cU_P,\phi_P)_{P \in B_{n+1}^{+}}$ is an atlas for $\SO_{n+1}$.

For each $Q_0 \in \SO_{n+1}$, the set 
$\Pi^{-1}[\cU_{Q_0}]\subset\Spin_{n+1}$ has two
contractible connected components: we call them $\cU_{z_0}$ and $\cU_{-z_0}$
where $z_0 \in \Spin_{n+1}$, $\Pi(z_0) = Q_0$,
$z_0 \in \cU_{z_0}$ and $-z_0 \in \cU_{-z_0}$.
Let $\phi_{z_0}: \cU_{z_0} \to \Lo_{n+1}^{1}$ be
$\phi_{z_0} = \phi_{Q_0} \circ \Pi$:
the family
$(\cU_z,\phi_z)_{z \in \widetilde B_{n+1}^{+}}$ is an atlas for $\Spin_{n+1}$.
In particular,
the preimage of the set $\cU_I$ under $\Pi: \Spin_{n+1} \to \SO_{n+1}$
has two contractible components $\cU_{\pm 1}$.
We abuse notation and write $\bL: \cU_1 \to \Lo_{n+1}^{1}$
and $\bQ: \Lo_{n+1}^{1} \to \cU_1$ for the diffeomorphisms
obtained by composition.

Recall that $\fa_i \in \so_{n+1}$ is
the matrix with only two nonzero entries given in Equation \ref{equation:fa}. 
Let $\fl_i \in \lo_{n+1}$ be the matrix with only one nonzero entry
$(\fl_i)_{i+1,i} = 1$.
Let $X_{\fa_i}$ and $X_{\fl_i}$ be the left-invariant vector fields
in $\SO_{n+1}$ and $\Lo_{n+1}^{1}$ generated by $\fa_i$ and $\fl_i$,
respectively:
$$ X_{\fa_i}(Q) = Q \fa_i, \qquad X_{\fl_i}(L) = L \fl_i. $$

\begin{lemma}
\label{lemma:al}
The diffeomorphisms $\bL: \cU_I \to \Lo_{n+1}^{1}$ and
$\bQ: \Lo_{n+1}^{1} \to \cU_I$ take the vector fields
$X_{\fa_i}$ and $X_{\fl_i}$ to smooth positive multiples of each other.
\end{lemma}

\begin{proof}
Given $Q_{0}\in\cU_{I}$, take a short arc of the 
integral line of $X_{\fa_{i}}$ through $Q_{0}$: 
let $\epsilon>0$ be sufficiently small so that 
$Q(t)=Q_{0}\exp(t\fa_{i})\in\cU_{I}$ for 
$-\epsilon<t<\epsilon$. Also write 
$\bL(Q(t))=L(t)\in\Lo^{1}_{n+1}$, so that 
$L(t)=Q(t)R(t)$ for a smooth path 
$R:(-\epsilon,\epsilon)\to\Up^{+}_{n+1}$. 
Differentiating the last equation, we have
\[(L(t))^{-1}L'(t)=(R(t))^{-1}\fa_{i}R(t)+(R(t))^{-1}R'(t).\]
Since the left hand side is in $\lo^{1}_{n+1}$ 
and the rightmost summand of the right hand side 
is in $\up^{+}_{n+1}$, it is readily seen that 
$L'(t)=(R(t)_{ii}/R(t)_{i+1,i+1})L(t)\fl_{i}$.
\end{proof}


A smooth curve $\Gamma: J \to \SO_{n+1}$ or 
$\Spin_{n+1}$ is \emph{locally convex}
if, for all $t \in J$,
the logarithmic derivative $(\Gamma(t))^{-1} \Gamma'(t)$
is tridiagonal with strictly positive subdiagonal entries,
i.e., a positive linear combination of $\fa_i$, $1 \le i \le n$. 
Notice that in this case $\gamma=\Gamma e_{1}: J\to\Ss^{n}$ 
is locally convex in the sense of the Introduction; 
conversely, if $\gamma:J\to\Ss^{n}$ is smooth and locally convex 
(in the sense of the Introduction), then $\Gamma=\Frenet{\gamma}$ 
is locally convex (in the sense we just introduced).
In other works \cite{Saldanha-Shapiro, Saldanha3}, 
the alternate name \emph{holonomic} was also used 
to designate this kind of curves. The reason for calling 
$\gamma=\Gamma e_{1}$ locally convex in the first place 
is given 
in Appendix \ref{appendix:convex}.

Similarly, a smooth 
curve $\Gamma: J \to \Lo_{n+1}^{1}$ 
is \emph{locally convex} if, for all $t \in J$,
the logarithmic derivative $(\Gamma(t))^{-1} \Gamma'(t)$
is a positive linear combination of $\fl_i$, $1 \le i \le n$.

\begin{example}
\label{example:fh}
Set $\fh = \fh_L-\fh_L^{\transpose} \in \so_{n+1}$, where
\begin{equation}
\label{equation:fhL}
\fh_L = \sum_{k\in\nmesmo} \sqrt{k(n+1-k)}\; \fl_k \in \lo_{n+1}, \qquad
\fn = \sum_{k\in\nmesmo}  \fl_k \in \lo_{n+1}.
\end{equation}
We have
$$ [\fh_L,\fh_L^\transpose] = 
\begin{pmatrix}
-n & & & \\ & \ddots & \\ & & n-2 & \\ & & & n
\end{pmatrix}, \quad
[\fh_l, [\fh_L,\fh_L^\transpose]] = -2 \fh_L, \quad
[\fh_l^\transpose, [\fh_L,\fh_L^\transpose]] = 2 \fh_L^\transpose.
$$
It is not hard to verify that, for $n = 1$ and
$\theta \in (-\frac{\pi}{2}, \frac{\pi}{2})$,
\begin{equation}
\label{equation:fhfhL}
\exp(\theta(\fh_L - \fh_L^\transpose)) =
\exp(\tan(\theta) \fh_L)
\exp(\log(\sec(\theta)) [\fh_L,\fh_L^\transpose])
\exp(-\tan(\theta) \fh_L^\transpose).
\end{equation}
The symmetric product induces a Lie algebra homomorphism
$S: \slalgebra_2 \to \slalgebra_{n+1}$ with $S(\fh_L) = \fh_L$
(that is, taking $\fh_L \in \RR^{2\times 2}$ 
to $\fh_L \in \RR^{(n+1)\times (n+1)}$) and
$S(\fh_L^\transpose) = \fh_L^\transpose$.
We therefore also have a Lie group homomorphism 
$S: \widetilde{\SL_2} \to \widetilde{\SL_{n+1}}$, 
$S(\exp(t\fu)) = \exp(t S(\fu))$ for all $\fu \in \slalgebra_2$, $t\in\RR$.
Equation \ref{equation:fhfhL} therefore holds for any value of $n$.
We therefore have $\exp(\theta\fh) \in \cU_1$ for all 
$\theta \in (-\frac{\pi}{2}, \frac{\pi}{2})$, with
$$ \bL(\exp(\theta \fh)) = \exp(\tan(\theta) \fh_L); \qquad
\bQ(\exp(t \fh_L)) = \exp(\arctan(t) \fh). $$

Also, the equation 
$\exp(\frac{\pi}{2}\fh)=\acute{\eta}$, 
which is trivially true for $n=1$, can be obtained 
for arbitrary $n$ using the Lie group homomorphism $S$  
above and noticing that $S(\acute{\eta})=\acute{\eta}$.  

For $z_0 \in \Spin_{n+1}$, the curve 
$\Gamma_{z_0,\fh}(t) = z_0 \exp(t\fh)$
is locally convex and satisfies 
$\Gamma_{z_0,\fh}(\frac{\pi}{2})=z_{0}\acute{\eta}$, 
$\Gamma_{z_0,\fh}(\pi)=z_{0}\hat{\eta}$. 
For $L_0 \in \Lo^{1}_{n+1}$, the curves
$\Gamma_{L_0, \fh_L}(t) = L_0 \exp(t\fh_L)$ 
and $\Gamma_{L_0, \fn}(t) = L_0 \exp(t\fn)$
are locally convex.
Notice that the $(i,j)$ entry of
either $\Gamma_{L_0, \fh_L}(t)$ or $\Gamma_{L_0, \fn}(t)$
is a polynomial of degree $(i - j)$ in the variable $t$.

\end{example}

One advantage of working with triangular coordinates
is that there is then a simple integration formula.
Indeed, let $\Gamma: [t_0,t_1] \to \Lo_{n+1}^{1}$ satisfy
$(\Gamma(t))^{-1} \Gamma'(t) = \sum_i \beta_i(t) \fl_i$.
Then
\[ (\Gamma(t))_{i+1,i} = (\Gamma(t_0))_{i+1,i} +
\int_{t_0}^t \beta_i(\tau) d\tau. \]
More generally,
\begin{equation}
\label{equation:explicitGamma}
((\Gamma(t_0))^{-1} \Gamma(t))_{i+l,i} =
\int_{t_0 \le \tau_1 \le \cdots \le \tau_l \le t}
\beta_{i+l-1}(\tau_1) \cdots \beta_{i}(\tau_{l}) d\tau_1 \cdots d\tau_l. 
\end{equation}
The set of smooth locally convex curves then has an easy generalization.
A curve $\Gamma$ is \emph{locally convex} if it is absolutely continuous
and its logarithmic derivative $(\Gamma(t))^{-1} \Gamma'(t)$
is almost everywhere a positive linear combination of $\fl_i$, $1 \le i \le n$.
In this case the functions $\beta_i: J \to (0,+\infty)$
belong to $L^1(J;\RR)$ and Equation \ref{equation:explicitGamma} above
makes perfect sense;
notice that if each $\beta_i$ is in $L^1(J;\RR)$
then $(\beta_{i_1}, \ldots, \beta_{i_l})$ is in $L^1(J;\RR^l)$.
Equivalently, a curve $\Gamma: [t_0,t_1] \to \Lo_{n+1}^{1}$
is locally convex if and only if
there exist finite absolutely continuous (positive) Borel measures
$\mu_1, \ldots, \mu_n$ on $J = [t_0,t_1]$ such that  $\mu_i(\tilde J)>0$ 
for any index $i$
and for any nondegenerate interval $\tilde J \subseteq J$
such that
\begin{align}
\label{equation:explicitGammamu}
((\Gamma(t_0))^{-1} \Gamma(t))_{i+l,i}
&= (\mu_{i+l-1} \times \cdots \times \mu_i)(\Delta), \\
\Delta &= \{(\tau_1,\ldots,\tau_l) \in [t_0,t]^k \;|\;
t_0 \le \tau_1 \le \cdots \le \tau_l \le t \}. \nonumber
\end{align}
Similarly, a curve $\Gamma:J\to\Spin_{n+1}$ is 
\emph{locally convex} if it is absolutely continuous and 
its logarithmic derivative $(\Gamma(t))^{-1}\Gamma'(t)$ 
is almost everywhere a positive linear combination of the 
vectors $\fa_i$, $i\in\nmesmo$. Equivalently, 
$\Gamma:J\to\Spin_{n+1}$ is locally convex if and only if, 
near any point $t_\bullet \in J$, there is a system of triangular coordinates 
$\Gamma_L(t)=\bL(z_0^{-1}\Gamma(t))$ with 
$\Gamma_L$ locally convex in the previous sense.

Let $\nmaisum^{(k)}$ be the set of subsets
$\bi \subseteq \nmaisum$ with $|\bi| = k$;
let $\sum(\bi)$ be the sum of the elements of the set $\bi$.
The $k$-th exterior (or alternating) power $\Lambda^k(V)$ of $V = \RR^{n+1}$
has a basis indexed by $\bi \in \nmaisum^{(k)}$.
For $\bi_0, \bi_1 \in \nmaisum^{(k)}$,
write:
$$ \bi_0 \overset{j}{\to} \bi_1 \quad \iff \quad
j \in \bi_1, \; j+1 \notin \bi_1, \;
\bi_0 =(\bi_1 \smallsetminus \{j\}) \cup \{j+1\}. $$
Notice that $\bi_0 \overset{j}{\to} \bi_1$ implies
$\sum(\bi_0) = 1+ \sum(\bi_1)$.
With respect to the basis above, the matrix of the 
linear endomorphism 
$\Lambda^k(\fl_i)\in\mathfrak{gl}(\Lambda^{k}(V))$ given by 
\[\Lambda^k(\fl_i)(v_{1}\wedge\cdots\wedge v_{k})
=\sum_{j\in\llbracket k\rrbracket}v_{1}
\wedge\cdots\wedge\fl_{i}(v_{j})
\wedge\cdots\wedge v_{k}\]
has nonzero entries all equal $1$ and in positions
$(\bi_0, \bi_1)$ such that $\bi_0 \overset{j}{\to} \bi_1$.
Write $\bi_1 \triangleleft \bi_0$ if there exists $j$
such that $\bi_0 \overset{j}{\to} \bi_1$ and
define a partial order in $\nmaisum^{(k)}$
by taking the transitive closure.
Equivalently, for
$\bi_j = \{i_{j1} < i_{j2} < \cdots < i_{jk} \}$
we have
$$ \bi_1 \le \bi_0 \quad\iff\quad
i_{11} \le i_{01}, \; i_{12} \le i_{02}, \; \cdots \;, i_{1k} \le i_{0k}. $$
If $\bi_0 \ge \bi_1$, $\sum(\bi_0) = l + \sum(\bi_1)$, write
$$ \bi_0 \overset{(j_1,\ldots,j_l)}{\longrightarrow} \bi_1 \quad\iff\quad
\exists \bj_0,\ldots,\bj_l, \;
\bi_0 = \bj_0 \overset{j_1}{\to} \bj_1 \to \cdots \to
\bj_{l-1} \overset{j_l}{\to} \bj_l = \bi_1; $$
notice that given $\bi_0$ and $\bi_1$
there may exist many such $l$-tuples $(j_1, \ldots, j_l)$.
Order the indices $\bi$ consistently with the partial order
introduced above (or, more directly, order the subsets $\bi$
increasing in the sum of their elements).
The matrix $\Lambda^k(\fl_i)$ is then strictly lower triangular.

If $L \in \Lo_{n+1}^{1}$ and $\bi_0, \bi_1 \in \nmaisum^{(k)}$,
define $L_{\bi_0,\bi_1}$ to be the $k \times k$ submatrix of $L$
obtained by selecting the rows in $\bi_0$ and the columns in $\bi_1$.
The $(\bi_0,\bi_1)$ entry of $\Lambda^k(L)$ is 
$\det(L_{\bi_0,\bi_1})$.
Clearly, $\bi_0 \not\ge \bi_1$ implies
$\det(L_{\bi_0,\bi_1}) = 0$;
also, $L_{\bi,\bi}$ is lower triangular with diagonal entries
equal to $1$ and therefore $\det(L_{\bi,\bi}) = 1$.
The matrix $\Lambda^k(L)$ is therefore lower triangular
with diagonal entries equal to $1$.
Furthermore, the map $\Lambda^k: \Lo_{n+1}^{1} \to \Lo_{\binom{n+1}{k}}^1$
is a group homomorphism.
The following result generalizes
Equations \ref{equation:explicitGamma} and \ref{equation:explicitGammamu} above.

\begin{lemma}
\label{lemma:explicitGamma}
Let $\Gamma: [t_0,t_1] \to \Lo_{n+1}^{1}$ be locally convex
with $\Gamma(t_0) = L_0$ and let
$\beta_i(t) = ((\Gamma(t))^{-1}\Gamma'(t))_{i+1,i}$,
$\mu_i(J) = \int_J \beta_i(t) dt$.
Let $\bi_0, \bi_1 \in \nmaisum^{(k)}$
with $\bi_0 \ge \bi_1$ and $l = (\sum \bi_0) - (\sum \bi_1)$.
Then
\begin{align*}
\det((L_0^{-1} \Gamma(t))_{\bi_0,\bi_1})
&= (\Lambda^k(L_0^{-1} \Gamma(t)))_{\bi_0,\bi_1} 
= \sum_{\bi_0 \overset{(j_1, \ldots, j_l)}{\longrightarrow} \bi_1}
(\mu_{j_1} \times \cdots \times \mu_{j_l})(\Delta); \\
\Delta &= \{(\tau_1,\ldots,\tau_l) \in [t_0,t]^l \;|\;
t_0 \le \tau_1 \le \cdots \le \tau_l \le t \}. 
\end{align*}
\end{lemma}

\begin{proof}
These are straightforward computations.
\end{proof}




\section{Totally positive matrices}
\label{sect:totallypositive}

A matrix $L \in \Lo_{n+1}^{1}$ is \emph{totally positive} if
for all $k \in \nmaisum$ and for all
indices $\bi_0, \bi_1 \in \nmaisum^{(k)}$,
$$ \bi_0 \ge \bi_1 \quad\implies\quad \det(L_{\bi_0,\bi_1}) > 0. $$
 
Let $\Pos_{\eta} \subset \Lo_{n+1}^{1}$
be the set of totally positive matrices.

Totally positive matrices were introduced 
independently in \cite{Gantmacher-Krein} and 
\cite{Schoenberg} and have since found widespread 
applications \cite{Ando, Brenti, Karlin, Lusztig3}. 
The concept of totally positive elements has been 
generalized to a reductive group $G$ and its flag 
manifold by G. Lusztig \cite{Lusztig1, Lusztig2, Lusztig3} 
and to Grassmannians by A. Postnikov \cite{Postnikov1, Postnikov2}. 
The relation of the subject with intersections of 
Bruhat (Schubert) cells has been studied in 
\cite{Fomin-Zelevinsky1, Rietsch, 
Shapiro-Shapiro-Vainshtein1, Shapiro-Shapiro-Vainshtein2}. 
See also \cite{Fomin-Zelevinsky2}. Our particular definition 
is analogous to that of \cite{Berenstein-Fomin-Zelevinsky}, 
where $G=N=\Up^{1}_{n+1}$ and 
$N_{>0}=\Pos_{\eta}^{\transpose}$: this one is a 
good reference for facts mentioned 
without proof in the present section.

Let $\jacobi_i(t) = \exp(t \fl_i)$: for any reduced word
$\eta = a_{i_1}a_{i_2}\cdots a_{i_m}$, $m = \inv(\eta) = n(n+1)/2$,
the map
$$ (0,+\infty)^m \to \Pos_{\eta}, \qquad
(t_1, t_2, \ldots, t_m) \mapsto
\jacobi_{i_1}(t_1)\jacobi_{i_2}(t_2) \cdots \jacobi_{i_m}(t_m) $$
is a diffeomorphism.
Moreover, there exists a stratification 
of its closure $\overline{\Pos_{\eta}}$:
$$ \overline{\Pos_{\eta}} =
\{ L \in \Lo_{n+1}^{1} \;|\;
\forall \bi_0, \bi_1,
\; ((\bi_0 \ge \bi_1) \to (\det(L_{\bi_0,\bi_1}) \ge 0)) \}
= \bigsqcup_{\sigma \in S_{n+1}} \Pos_{\sigma}; $$
$\Pos_\sigma \subset \Lo_{n+1}^{1}$ is a smooth manifold of dimension $\inv(\sigma)$,
and if $\sigma = a_{i_1}\cdots a_{i_k}$ is a reduced word
(so that $k = \inv(\sigma)$)
then the map
$$ (0,+\infty)^k \to \Pos_{\sigma}, \qquad
(t_1, t_2, \ldots, t_k) \mapsto
\jacobi_{i_1}(t_1)\jacobi_{i_2}(t_2) \cdots \jacobi_{i_k}(t_k) $$
is a diffeomorphism.
Equivalently, if $\sigma_1 \triangleleft \sigma_0 = \sigma_1 a_{i_k}$ then the map
$$ \Pos_{\sigma_1} \times (0,+\infty) \to \Pos_{\sigma_0}, \qquad
(L,t_k) \mapsto L \jacobi_{i_k}(t_k) $$
is a diffeomorphism.

Different reduced words yield different diffeomorphisms
but the same set $\Pos_\sigma$:
the equation
\begin{equation}
\label{equation:ababab}
\jacobi_{i}(t_1) \jacobi_{i+1}(t_2) \jacobi_i(t_3) =
\jacobi_{i+1}\left(\frac{t_2t_3}{t_1+t_3}\right) \jacobi_i(t_1+t_3) 
\jacobi_{i+1}\left(\frac{t_1t_2}{t_1+t_3}\right) 
\end{equation}
provides the transition between adjacent parameterizations
(i.e., between reduced words connected by the local move
in Equation \ref{equation:reducedword2};
the local move in Equation \ref{equation:reducedword1}
corresponds to a mere relabeling).

In general,
the sets $\Pos_{\sigma} \subset \Lo_{n+1}^{1}$ are 
neither subgroups nor semigroups and should not be confused
with the subgroups
$\Lo_\sigma = \Lo_{n+1}^1 \cap (P_\sigma^{-1} \Up_{n+1}^1 P_\sigma)$.
For instance, $\Pos_{e} = \{I\}$ consists of a single point and
$\Pos_{a_i} = \{ \jacobi_i(t), t > 0 \}$ is an open half line;
in this case, $\Pos_{a_i} \subset \Lo_{a_i}$.
For $n = 2$ and
\begin{equation}
\label{equation:Lxyz}
L(x,y,z) =
\begin{pmatrix} 1 & 0 & 0 \\ x & 1 & 0 \\ z & y & 1 \end{pmatrix} 
\end{equation}
we have
\begin{gather*}
\Pos_{[a_1a_2]} = \{ L(x,y,0) \;|\; x, y > 0 \}; \qquad
\Pos_{[a_2a_1]} = \{ L(x,y,xy) \;|\; x, y > 0 \}; \\
\Pos_{[a_1a_2a_1]} = \{ L(x,y,z) \;|\; x, y > 0; \; 0 < z < xy \}.
\end{gather*}
On the other hand,
\[ \Lo_{[a_1a_2]} = \{ L(x,0,z) \;|\; x, z \in \RR \}; \qquad
\Lo_{[a_2a_1]} = \{ L(0,y,z) \;|\; y, z \in \RR \}. \]
If $L \in \Pos_{\sigma}$ then there exist matrices $U_1, U_2 \in \Up_{n+1}$
such that $L = U_1 P_{\sigma} U_2$, but the converse is not at all true,
not even if we pay attention to signs of diagonal entries
of the matrices $U_i$.
In \cite{Shapiro-Shapiro-Vainshtein1, Shapiro-Shapiro-Vainshtein2} 
it is shown that the set of matrices
which admit such a decomposition is almost always disconnected;
each cell $\Pos_{\sigma}$ is contractible,
and so is its closure $\overline{\Pos_\sigma}$;
see also Lemma \ref{lemma:posbruhat} below.

\begin{lemma}
\label{lemma:possigma}
Consider $\sigma \in S_{n+1}$, $k \in \nmaisum$ and 
indices $\bi_0, \bi_1 \in \nmaisum^{(k)}$.
If there exists $\sigma_1 = a_{j_1}\cdots a_{j_{l_1}} \le \sigma$,
$\inv(\sigma_1) = l_1$, such that
$\bi_0 \overset{(j_1,\ldots,j_{l_1})}{\longrightarrow} \bi_1$
then, for all $L \in \Pos_\sigma$,
$(\Lambda^k(L))_{\bi_0,\bi_1} > 0$.
Conversely, if no such $\sigma_1$ exists 
then, for all $L \in \Pos_\sigma$,
$(\Lambda^k(L))_{\bi_0,\bi_1} = 0$.
\end{lemma}

\begin{proof}
Write a reduced word $\sigma = a_{i_1}\cdots a_{i_l}$.
Assume first that such $\sigma_1$ exists
and that $j_1 = i_{x_1}, \ldots, j_{l_1} = i_{x_{l_1}}$
(where of course $1 \le x_1 < \cdots < x_{l_1} \le l$).
Set 
$$ \bi_0 = \bj_0 \overset{j_1}{\to} \bj_1 \to \cdots \to
\bj_{l_{1}-1} \overset{j_{l_1}}{\to} \bj_{l_1} = \bi_1; \qquad
L = \jacobi_{i_1}(t_1) \cdots \jacobi_{i_{\tilde l}}(t_{l})
\in \Pos_\sigma. $$
We have $ (\Lambda^k(L))_{\bi_0,\bi_1} \ge t_{x_1} \cdots t_{x_{l_1}} > 0 $,
as desired.

Conversely, assume that
$L = \jacobi_{i_1}(t_1) \cdots \jacobi_{i_{l}}(t_{l})$,
$(\Lambda^k(L))_{\bi_0,\bi_1} > 0$.
We have 
$$ (\Lambda^k(L))_{\bi_0,\bi_1} = \sum_{\bi_0 =
\bj_0 \ge \cdots \ge \bj_{l} = \bi_1}
\left( (\Lambda^k(\jacobi_{i_1}(t_1)))_{\bj_0,\bj_1} \cdots
(\Lambda^k(\jacobi_{i_{l}}(t_{l}))
)_{\bj_{l - 1},\bj_{l}} \right).  $$
Consider $(\bj_0, \ldots, \bj_{l})$ such that the above product
is positive.
Let $x_1, \ldots, x_{l_1}$ be such that
$\bj_{x_1-1} > \bj_{x_1},\ldots, \bj_{x_{l_1}-1} > \bj_{x_{l_1}}$:
this obtains a reduced word for $\sigma_1$.
\end{proof}

\begin{lemma}
\label{lemma:positivesemigroup}
Consider $L_0, L_1 \in \Lo_{n+1}^{1}$.
If $L_0 \in \Pos_{\sigma_0}$ and $L_1 \in \Pos_{\sigma_1}$
then $L_0L_1 \in \Pos_{\sigma_0\vee\sigma_1}$.
Thus, $\Pos_{\sigma_0}\Pos_{\sigma_1} = \Pos_{\sigma_0 \vee \sigma_1}$.

In particular, 
if $L_0 \in \Pos_{\eta}$ and $L_1 \in \overline{\Pos_{\eta}}$
then $L_0L_1, L_1L_0 \in \Pos_{\eta}$.
If $L_0 \in \overline{\Pos_{\eta}}$ and $L_1 \in \overline{\Pos_{\eta}}$
then $L_0L_1 \in \overline{\Pos_{\eta}}$.
\end{lemma}

The operation $\vee$ is the one in Example \ref{example:vee}.

\begin{proof}
The first claim can be proved by induction on $l = \inv(\sigma_1)$;
the case $l = 0$ is trivial.
For the case $l = 1$, consider $\sigma_1 = a_i$ and two cases.
If $\sigma_0 \vee a_i = \sigma_0$, we take a reduced word
$\sigma_0 = a_{i_1}\cdots a_{i_k}$ with $i_k = i$. Then
$$ L_0L_1 = (\lambda_{i_1}(t_1) \cdots \lambda_{i_k}(t_k))
\lambda_{i}(t) = \lambda_{i_1}(t_1) \cdots \lambda_{i_k}(t_k + t)
\in \Pos_{\sigma_0}. $$
The case $\sigma_0 \vee a_i \ne \sigma_0$ is even more direct.
The induction step is now easy.

The other claims follow from the first, 
but a direct proof may be instructive:
consider $\bi_0, \bi_1 \in \nmaisum^{(k)}$, $\bi_0 \ge \bi_1$.
If $L_0 \in \Pos_{\eta}$ and $L_1 \in \overline{\Pos_{\eta}}$ we have
$$ (\Lambda^k(L_0L_1))_{\bi_0\bi_1} =
(\Lambda^k(L_0))_{\bi_0\bi_1} (\Lambda^k(L_1))_{\bi_1\bi_1} +
\sum_{\bi_0 \ge \bi > \bi_1}
(\Lambda^k(L_0))_{\bi_0\bi} (\Lambda^k(L_1))_{\bi\bi_1} > 0, $$
as desired; the other cases are similar.
\end{proof}

Write $L_0 \le L_1$ if $L_0^{-1} L_1 \in \overline{\Pos_{\eta}}$
and $L_0 \ll L_1$ if $L_0^{-1} L_1 \in \Pos_{\eta}$;
notice that $L_0^{-1} L_1 \in \Pos_{\eta}$ is in general
not equivalent to $L_1 L_0^{-1} \in \Pos_{\eta}$.
Lemma \ref{lemma:positivesemigroup} implies that these are partial orders:
\begin{equation}
\label{equation:positivesemigroup}
L_0 \le L_1 \le L_2 \;\implies\; L_0 \le L_2; \qquad
L_0 \le L_1 \ll L_2 \;\implies\; L_0 \ll L_2.
\end{equation}

\begin{lemma}
\label{lemma:totallypositive}
Consider $L_0, L_1 \in \Lo_{n+1}^{1}$.
We have that $L_0 \ll L_1$ if and only if
there exists a locally convex curve $\Gamma: [0,1] \to \Lo_{n+1}^{1}$
with $\Gamma(0) = L_0$ and $\Gamma(1) = L_1$.
\end{lemma}


\begin{proof}
We first prove that the existence of $\Gamma$ implies $L_0 \ll L_1$.
Given $\Gamma$ and $\bi_0, \bi_1 \in \nmaisum^{(k)}$ with $\bi_0 > \bi_1$,
Lemma \ref{lemma:explicitGamma} gives us a formula 
for $(L_0^{-1} L_1)_{\bi_0,\bi_1} > 0$:
$L_0^{-1} L_1$ is therefore totally positive.
Notice that the curve $\Gamma$ is not required to be smooth.

Conversely,
let $\fh_L = \sum_i c_i \fl_i \in \lo_{n+1}^{1}$ for fixed positive $c_i$.
Consider a small closed ball of radius $r > 0$
centered at $L_0^{-1}L_1$ and contained in $\Pos_\eta$,
the image of a continuous map
$h: \BB^m \to \Pos_\eta \subset \Lo_{n+1}^{1}$
with $h(0) = L_0^{-1}L_1$ such that the topological degree of
$h|_{\Ss^{m-1}}$ around $L_0^{-1}L_1$ equals $+1$
(here $m = \dim(\Lo_{n+1}^{1})$).
Consider a fixed reduced word $\eta = a_{i_1} \cdots a_{i_m}$.
Define continuous functions $\tau_i: \BB^m \to (0,+\infty)$
such that $h(s) = \jacobi_{i_1}(\tau_1(s))\cdots \jacobi_{i_m}(\tau_m(s))$.
For $\epsilon \ge 0$, let
$$ \Lambda_\epsilon(s)(t) = m \tau_{j}(s) \fl_{i_j} + \epsilon \fh_L, \quad
t \in \left(\frac{j-1}{m},\frac{j}{m}\right). 
$$
Integrate to obtain functions 
$$ \Gamma_\epsilon(s): [0,1] \to \Lo_{n+1}^1, \quad
\Gamma_\epsilon(s)(0) = L_0, \quad
(\Gamma_\epsilon(s)(t))^{-1} (\Gamma_\epsilon(s))'(t) =
\Lambda_\epsilon(s)(t). $$
Notice that $\Gamma_\epsilon(s)$ is a locally convex curve if $\epsilon > 0$.
Define $h_\epsilon(s) = L_0^{-1} \Gamma_\epsilon(s)(1)$:
clearly $h_0 = h$, i.e., $\Gamma_0(s)(1) = L_0 h(s)$.
By continuity, there exists $\epsilon > 0$ such that for all $s \in \BB^m$
we have $|h_\epsilon(s) - h_0(s)| < r/2$.
The topological degree of $h_\epsilon|_{\Ss^{m-1}}$
around $L_0^{-1}L_1$ equals $+1$.
There exists therefore $s_\epsilon \in \BB^m$
with $h_\epsilon(s_\epsilon) = L_0^{-1}L_1$.
We have that $\Gamma = \Gamma_{\epsilon}(s_\epsilon)$,
$\Gamma: [0,1] \to \Lo_{n+1}^{1}$
is a locally convex curve with
$\Gamma(0) = L_0$, $\Gamma(1) = L_1$.
\end{proof}

We know by now that if $L_0 \in \Pos_\sigma$ for $\sigma \ne \eta$
and $\Gamma: [0,1] \to \Lo_{n+1}^1$ is a locally convex curve
with $\Gamma(0) = L_0$ then $\Gamma(t) \in \Pos_\eta$ for all $t > 0$.
The following lemma shows that,
at least from the point of view of certain entries,
the curve $\Gamma$ goes in with positive speed.

\begin{lemma}
\label{lemma:positivespeed}
Given $\sigma \in S_{n+1}$, $\sigma \ne \eta$,
there exist $k \in \nmaisum$ and indices
$\bi_0, \bi_1, \bi_2 \in \nmaisum^{(k)}$ and $j \in \nmesmo$
such that $\bi_0 \ge \bi_1 > \bi_2$,
$\bi_1 \overset{j}{\to} \bi_2$ and, 
for all locally convex curves $\Gamma: [0,1] \to \Lo_{n+1}^1$
with $\Gamma(0) \in \Pos_\sigma$ and $\Gamma'(0) \ne 0$ (and well defined),
if $g(t) = (\Lambda^k(\Gamma(t)))_{\bi_0,\bi_2}$ then
$g(0) = 0$ and $g'(0) > 0$.
\end{lemma}

\begin{proof}
Consider $k$ and a pair of indices $(\bi_0,\bi_3)$, $\bi_0 \ge \bi_3$,
$\bi_\ast \in \nmaisum^{(k)}$
such that $(\Lambda^k(L))_{\bi_0,\bi_3} = 0$ for $L \in \Pos_\sigma$
(see Lemma \ref{lemma:possigma}).
Keep $k$ and $\bi_0$ fixed and search for $\bi_2 \le \bi_0$ maximal
such that $(\Lambda^k(L))_{\bi_0,\bi_2} = 0$ for $L \in \Pos_\sigma$.
Maximality implies that there exists $\bi_1$,
$\bi_0 \ge \bi_1 > \bi_2$
and an index $j$ such that $\bi_1 \overset{j}{\to} \bi_2$ and
$(\Lambda^k(L))_{\bi_0,\bi_1} > 0$ for $L \in \Pos_\sigma$.

Let $L_0=\Gamma(0)$, $c_0 = (\Lambda^k(L_0))_{\bi_0,\bi_1} > 0$.
Write $h_i(t) = (L_0^{-1} \Gamma(t))_{i+1,i}$
so that $h_i(0) = 0$ and $h'_i(0) = c_i > 0$.
Clearly, $g(0)=0$ and, for all $t>0$, 
\[g(t) \ge
(\Lambda^k(L_0))_{\bi_0,\bi_1} (\Lambda^k(L_0^{-1} \Gamma(t)))_{\bi_1,\bi_2} =
c_0 h_j(t) = c_0 c_j (t+ o(t))\]
so that $g'(0) \ge c_0c_j > 0$, as desired.
\end{proof}

\begin{rem}
\label{rem:explicitpositivespeed}
We now present an explicit construction. 
Given $\sigma\neq\eta$, take $k$ minimal such that 
$(n-k+2)^\sigma \neq k$. Set then $j=(n-k+2)^\sigma -1$. 
Equivalently, $k$ is minimal such that 
$\Lambda^k(L)_{\bi_0,\bi_3}=0$ for $L\in\Pos_\sigma$, 
$\bi_0=\{n-k+2,\ldots,n+1\}$ and $\bi_3=\{1,\ldots,k\}$.
If we follow the proof of Lemma \ref{lemma:positivespeed}, we have 
$\bi_1=\{1,\ldots,k-1,j+1\}$ and $\bi_2=\{1,\ldots,k-1,j\}$.
\end{rem}

Let 
\begin{align*}
\Neg_{\sigma} &= X \Pos_{\sigma} X = \{L^{-1}, L \in \Pos_{\sigma^{-1}}\} \\
&= \{ \jacobi_{i_1}(t_1)\jacobi_{i_2}(t_2) \cdots \jacobi_{i_k}(t_k);\;
t_1, t_2, \ldots, t_k \in (-\infty,0) \} 
\end{align*}
where $X = \diag(1,-1,1,-1,\ldots)$ and
$\sigma = a_{i_1}a_{i_2}\cdots a_{i_k}$ is any reduced word
(therefore $k = \inv(\sigma)$).
Of course, each cell $\Neg_\sigma\subset \Lo^{1}_{n+1}$ 
is a contractible submanifold of dimension $\inv(\sigma)$, 
forming the stratification
$$ \overline{\Neg_{\eta}} = \bigsqcup_{\sigma \in S_{n+1}} \Neg_{\sigma}. $$
Notice that
$\overline{\Pos_\eta} \cap \overline{\Neg_\eta} = \{I\}$.

\begin{lemma}
\label{lemma:transition}
Consider an interval $J$ and a locally convex curve
$\Gamma: J \to \Lo_{n+1}^{1}$.
\begin{enumerate}
\item{If $t_{-1} < t_0 < t_1$ and
$\Gamma(t_0) \in \Pos_{\sigma} \subset \partial\Pos_{\eta}$
for some $\sigma \ne \eta$ then
$\Gamma(t_1) \in \Pos_{\eta}$ and
$\Gamma(t_{-1}) \notin \overline{\Pos_{\eta}}$.}
\item{If $t_{-1} < t_0 < t_1$ and
$\Gamma(t_0) \in \Neg_{\sigma} \subset \partial\Neg_{\eta}$
for some $\sigma \ne \eta$ then
$\Gamma(t_{-1}) \in \Neg_{\eta}$ and
$\Gamma(t_1) \notin \overline{\Neg_{\eta}}$.}
\item{If $t_0 < t < t_1$ then 
$\Gamma(t) \in (\Gamma(t_0) \Pos_{\eta}) \cap (\Gamma(t_1) \Neg_{\eta})$.}
\end{enumerate}
\end{lemma}

\begin{proof}
As in the first item, assume $\Gamma(t_0)\in\Pos_{\sigma}$, 
$\sigma\ne\eta$. From Lemma \ref{lemma:totallypositive}, 
$\Gamma(t_0)\ll \Gamma(t_1)$ and, by definition,  
$\Gamma(t_0)^{-1}\Gamma(t_1)\in\Pos_{\eta}$. 
By Lemma \ref{lemma:positivesemigroup}, 
$\Gamma(t_1)=\Gamma(t_0)\Gamma(t_0)^{-1}\Gamma(t_1)\in\Pos_{\eta}$, 
proving the first claim. 
Assume by contradiction that $\Gamma(t_{-1})\in\overline{\Pos_{\eta}}$: 
from the claim just proved, $\Gamma(t_0)\in\Pos_\eta$, 
a contradiction.
The second item is analogous.
The third item follows from the previous ones.
\end{proof}

\begin{lemma}
\label{lemma:posline}
Consider a reduced word $a_{i_1} \cdots a_{i_m} = \eta$;
consider 
$$ L = \jacobi_{i_1}(t_1) \cdots \jacobi_{i_m}(t_m) \in \Pos_\eta, \qquad
t_1, \ldots, t_m > 0. $$
Then $\jacobi_{i_1}(t) \ll L$ if and only if $t < t_1$ and
$\jacobi_{i_1}(t) \le L$ if and only if $t \le t_1$.
\end{lemma}

\begin{proof}
Let $\sigma_1 = a_{i_1}\eta = a_{i_2} \cdots a_{i_m} \triangleleft \eta$; let
$$ L_1 = \jacobi_{i_1}(-t_1) L =
\jacobi_{i_2}(t_2) \cdots \jacobi_{i_m}(t_m) \in \Pos_{\sigma_1} \subset
\overline{\Pos_\eta}. $$
By definition, $\jacobi_{i_1}(t) \ll L$ if and only if
$\jacobi_{i_1}(t_1-t) L_1 \in \Pos_{\eta}$:
this clearly holds for $t < t_1$.
For $t = t_1$, we have 
$\jacobi_{i_1}(t_1-t) L_1 = L_1 \in \Pos_{\sigma_1}$
and therefore $\jacobi_{i_1}(t) \le L$, $\jacobi_{i_1}(t) \not\ll L$.

Finally, assume by contradiction that for $t > t_1$ we have
$\jacobi_{i_1}(t_1-t) L_1 \in \Pos_{\sigma} \subset \overline\Pos_\eta$.
If $a_{i_1}\sigma < \sigma$ consider a reduced word
$\sigma = a_{i_1}a_{j_2}\cdots a_{j_k}$ and write
$$ \jacobi_{i_1}(t_1-t) L_1 =
\jacobi_{i_1}(\tau_1)\jacobi_{j_2}(\tau_2)\cdots\jacobi_{j_k}(\tau_k)
$$
so that
$$ L_1 =
\jacobi_{i_1}(t-t_1+\tau_1)\jacobi_{j_2}(\tau_2)\cdots\jacobi_{j_k}(\tau_k)
\in \Pos_{\sigma},
$$
which implies $\sigma = \sigma_1$, contradicting $a_{i_1}\sigma < \sigma$.
We thus have $a_{i_1}\sigma > \sigma$:
consider a reduced word $\sigma = a_{j_1}\cdots a_{j_k}$  and write
$$ \jacobi_{i_1}(t_1-t) L_1 =
\jacobi_{j_1}(\tau_1)\jacobi_{j_2}(\tau_2)\cdots\jacobi_{j_k}(\tau_k) $$
so that
$$ L_1 =
\jacobi_{i_1}(t-t_1)\jacobi_{j_1}(\tau_1)
\jacobi_{j_2}(\tau_2)\cdots\jacobi_{j_k}(\tau_k)
\in \Pos_{a_{i_1}\sigma},
$$
which implies $a_{i_1}\sigma = \sigma_1$,
contradicting $a_{i_1}\sigma > \sigma$.
\end{proof}

\section{Accessibility in triangular coordinates}
\label{sect:acctriangle}

For $L_{\bfx} \in \Pos_{\eta} \subset \Lo_{n+1}^{1}$
we shall be interested in the \emph{interval}
$$ [I,L_{\bfx}) = 
\overline{\Pos_{\eta}} \cap (L_{\bfx} \Neg_{\eta})
= \{ L \in \Lo_{n+1}^{1} \;|\; I \le L \ll L_{\bfx} \}
= \bigsqcup_{\sigma\in S_{n+1}} \Ac_{\sigma}(L_{\bfx}) $$
where the strata $\Ac_\sigma(L_{\bfx}) \subset \Pos_\sigma$ are
$$ \Ac_{\sigma}(L_{\bfx}) = [I,L_{\bfx}) \cap \Pos_\sigma =
\{ L \in \Pos_\sigma \;|\; L \ll L_{\bfx} \}. $$
The sets $\Ac_{\sigma}(L_{\bfx})$ will be called \emph{accessibility sets},
suggesting that for $L \in \Pos_\sigma$,
$L \in \Ac_\sigma(L_{\bfx})$ if and only if there exists
a locally convex curve $\Gamma: [0,1] \to \Lo_{n+1}^1$
with $\Gamma(0) = L$ and $\Gamma(1) = L_{\bfx}$.

\begin{example}
\label{example:Ac}
Take $n = 2$ and write 
$L(x,y,z)$ as in Equation \ref{equation:Lxyz}:
\begin{gather*}
L_{\bfx} = L(x,y,z) =
\jacobi_1(c_1) \jacobi_2(c_2) \jacobi_1(c_3) =
\jacobi_2(\tilde c_1) \jacobi_1(\tilde c_2) \jacobi_2(\tilde c_3), \\
c_1 = x-\frac{z}{y}, \quad c_2 = y,  \quad c_3 = \frac{z}{y}, \qquad
\tilde c_1 = y-\frac{z}{x},\quad \tilde c_2 = x,\quad \tilde c_3 = \frac{z}{x}.
\end{gather*}
We can now explicitly describe the strata 
$\Ac_{\sigma}=\Ac_{\sigma}(L_{\bfx})$.
The first stratum is a point:
$\Ac_e = \{I\}$.
Next we have line segments:
$$ \Ac_{a_1} = \{ \jacobi_1(t_1) \;|\; t_1 \in (0, c_1) \}, \quad
\Ac_{a_2} = \{ \jacobi_2(\tilde t_1) \;|\; t_2 \in (0, \tilde c_1) \}. $$
The next strata are surfaces:
\begin{align*}
{\Ac_{[a_1a_2]}} &=
\left\{ \jacobi_1(t_1) \jacobi_2(t_2) \;|\;
t_1 \in (0,c_1), t_2 \in (0,g_2(t_1)) \right\},
\quad g_2(t_1) = \frac{c_2c_3}{c_1+c_3-t_1}, \\
{\Ac_{[a_2a_1]}} &=
\left\{ \jacobi_2(\tilde t_1) \jacobi_1(\tilde t_2) \;|\;
\tilde t_1 \in (0,\tilde c_1), \tilde t_2 \in (0,\tilde g_2(\tilde t_1))
\right\},
\quad \tilde g_2(\tilde t_1) =
\frac{\tilde c_2\tilde c_3}{\tilde c_1+\tilde c_3-\tilde t_1}.
\end{align*}
Translating this parametrization back to $(x,y,z)$ coordinates
shows that $\Ac_{[a_1a_2]}$ is contained in the plane $z = 0$
and $\Ac_{[a_2a_1]}$ is contained in the hyperbolic paraboloid $z = xy$.
Finally, the open stratum $\Ac_\eta$ can be described as
\begin{gather*}
\Ac_\eta =
\left\{ \jacobi_1(t_1) \jacobi_2(t_2) \jacobi_1(t_3) \;|\;
t_1 \in \left(0,c_1 \right),
t_2 \in \left(0,g_2(t_1) \right),
t_3 \in \left(0,g_3(t_1,t_2) \right)\right\} \\
\phantom{\Ac_\eta} =
\left\{ \jacobi_2(\tilde t_1) \jacobi_1(\tilde t_2) \jacobi_2(\tilde t_3) \;|\;
\tilde t_1 \in \left(0,\tilde c_1 \right),
\tilde t_2 \in \left(0,\tilde g_2(\tilde t_1) \right),
\tilde t_3 \in \left(0,\tilde g_3(\tilde t_1,\tilde t_2) \right)\right\}, \\
g_3(t_1,t_2) = \frac{c_2(c_1-t_1)}{c_2-t_2}, \qquad
\tilde g_3(\tilde t_1,\tilde t_2) =
\frac{\tilde c_2(\tilde c_1-\tilde t_1)}{\tilde c_2-\tilde t_2}.
\end{gather*}
\end{example}

A \emph{quasiproduct} is a finite sequence $(X_j)_{1 \le j \le k}$
of open sets $X_j \subset (0,+\infty)^j$ such that there exist
a constant $c_1 \in (0,+\infty)$
and continuous functions $g_j: X_{j-1} \to (0,+\infty)$
for $2 \le j \le k$ such that
$X_1 = (0,c_1)$ and
$$ X_j = \{ (t_1, \ldots, t_{j-1}, t_j) \in X_{j-1} \times (0,+\infty) \;|\;
t_{j} < g_j(t_1, \ldots, t_{j-1}) \}, \quad
2 \le j \le k. $$
Notice that $X_k$ is homeomorphic to $\RR^k$.

\begin{lemma}
\label{lemma:triangularquasiproduct}
If $L_{\bfx} \in \Pos_\eta$, each stratum $\Ac_\sigma(L_{\bfx})$
is an open, bounded and contractible subset of $\Pos_\sigma$.
Moreover, if $\sigma = \sigma_k = a_{i_1} \cdots a_{i_k}$ is a reduced word
and $\sigma_j =  a_{i_1} \cdots a_{i_j}$, $j \le k$, then
$$ \Ac_{\sigma_j}(L_{\bfx}) = 
\{ \jacobi_{i_1}(t_1) \cdots
\jacobi_{i_j}(t_{j})
\;|\;
(t_1,\ldots,t_{j}) \in X_j \} $$
where the sequence $(X_j)_{1 \le j \le k}$ is a quasiproduct;
the functions $g_i: X_{i-1} \to (0,+\infty)$
are rational and bounded.
\end{lemma}


Example \ref{example:Ac} above illustrates this claim for $n = 2$.

\begin{proof}
Notice that $I \le L \ll L_{\bfx}$ implies
that $L \in \overline{\Pos_\eta}$ and that
there exists $\tilde L \in \Pos_\eta$ with $L \tilde L = L_{\bfx}$.
Computing $(L_{\bfx})_{ij}$ in this product yields
$0 \le (L)_{ij} \le (L_{\bfx})_{ij}$:
it follows that the interval $[I,L_{\bfx})$ is bounded.

The proof is by induction on $k = \inv(\sigma)$; the case $k = 1$ is easy.
Write
$$ X_j = \{ (t_1, \ldots, t_{j}) \in (0,+\infty)^{j} \;|\;
\jacobi_{i_1}(t_1) \cdots  \jacobi_{i_{j}}(t_{j})
\ll L_{\bfx} \}. $$
We assume by induction that $(X_j)_{1\le j\le k-1}$ is a quasiproduct;
we need to construct the function $g_k: X_{k-1} \to (0,+\infty)$
that obtains $X_k$.

Let $\eta = a_{j_1} \cdots a_{j_m}$ be a reduced word with $j_1 = i_k$. 
Given $(t_1, \ldots , t_{k-1}) \in X_{k-1}$, let
$$ L_{\sigma_{k-1}} = \jacobi_{i_1}(t_1) \cdots \jacobi_{i_{k-1}}(t_{k-1})
\in \Ac_{\sigma_{k-1}}(L_{\bfx}) \subset \Pos_{\sigma_{k-1}} $$
and write
$$ L_{\sigma_{k-1}}^{-1} L_{\bfx} =
\jacobi_{j_1}(\tau_1) \cdots \jacobi_{j_m}(\tau_m) \in \Pos_\eta $$
so that $\tau_1 > 0$ is a function of $(t_1, \ldots, t_{k-1})$:
define $g_k(t_1, \ldots, t_{k-1}) = \tau_1$.
It follows from Equation \ref{equation:ababab} that $g$ is a rational function.
As in Lemma \ref{lemma:posline},
$\jacobi_{i_k}(t) \ll L_{\sigma_{k-1}}^{-1} L_{\bfx}$
if and only if $t < g_k(t_1, \ldots, t_{k-1})$, as claimed.
\end{proof}



\section{Bruhat cells}
\label{sect:bruhatcell}

For $\sigma \in S_{n+1}$, let $\Bru_{\sigma} \subset \SO_{n+1}$
be the unsigned Bruhat cell 
$$ \Bru_{\sigma} =  \{ Q \in \SO_{n+1} \;|\;
\exists U_0, U_1 \in \Up_{n+1}, Q = U_0 P_{\sigma} U_1 \}; $$
notice that this set is not connected. The set 
$\Bru_\sigma\subset\SO_{n+1}$ 
is the lift of the corresponding Schubert cell 
$\mathcal{C}_\sigma\subset\GL_{n+1}/\Up_{n+1}$ 
in the complete flag manifold 
under the inclusion map. 
These cells, particularly the intersection of translated Bruhat 
cells, have been extensively studied \cite{Chevalley, Demazure, Fomin-Zelevinsky1, Konstant, Rietsch, Shapiro-Shapiro-Vainshtein1, Shapiro-Shapiro-Vainshtein2}.
As in \cite{Goulart, Saldanha3, Saldanha-Shapiro},
the \emph{signed Bruhat cell} $\Bru_{Q_0} \subset \SO_{n+1}$
for $Q_0 \in B^{+}_{n+1}$ is
$$ \Bru_{Q_0} = \{ Q \in \SO_{n+1} \;|\;
\exists U_0, U_1 \in \Up_{n+1}^{+}, Q = U_0 Q_0 U_1 \} $$
where $\Up^{+}$ is the group of upper triangular matrices
with positive diagonal.
The signed Bruhat cell $\Bru_{Q_{0}}$ is homeomorphic 
to the Schubert cell $\mathcal{C}_{\sigma_{Q_0}}$: 
the signed Bruhat cells are therefore contractible and disjoint.
Each unsigned Bruhat cell is a disjoint union of $2^n$ signed Bruhat cells.
The preimage of each cell by $\Pi: \Spin_{n+1} \to \SO_{n+1}$
is a disjoint union of two contractible components:
we call these connected components the 
(lifted) Bruhat cells in $\Spin_{n+1}$:
for $z \in \widetilde B^{+}_{n+1}$, let $\Bru_z$ be the connected component of $\Pi^{-1}[\Bru_{\Pi(z)}]$ containing $z$.

The group $\Up_{n+1}^{+}$ of upper triangular matrices
with positive diagonal entries acts on $\SO_{n+1}$:
define $Q^U = \bQ(U^{-1}Q)$.
This action preserves Bruhat cells and may be lifted to an action
on the spin group $\Spin_{n+1}$: we write $z^U = \bQ(U^{-1}z)$.
Also, if $U \in \Up_{n+1}^{+}$ and
$\Gamma: [0,1] \to \Spin_{n+1}$ is a locally convex curve
then $\Gamma^U: [0,1] \to \Spin_{n+1}$,
$\Gamma^U(t) = \bQ(U^{-1}\Gamma(t))$, is also a locally convex curve.
Also, the nilpotent subgroup $\Up_{n+1}^{1}$ acts simply transitively
on open Bruhat cells $\Bru_{q\acute\eta}$
and transitively on any Bruhat cell (\cite{Saldanha-Shapiro}).
The map $\Spin_{n+1} \to \Spin_{n+1}$, $z \mapsto z^U$,
can be considered to be induced from the projective transformation
$$ \Ss^n \to \Ss^n, \qquad v \mapsto \frac{U^{-1}v}{|U^{-1}v|}; $$
we thus say that $\Up_{n+1}^{+}$ acts on $\Spin_{n+1}$
(or $\Bru_\sigma$ or $\Bru_{z_0}$) and on spaces of 
locally convex curves by \emph{projective transformations}.

\begin{rem}
\label{rem:projtrans}
Consider $z_0, z_1, \tilde z_0, \tilde z_1 \in \Spin_{n+1}$ such that
$z_0^{-1}z_1, \tilde z_0^{-1}\tilde z_1 \in \Bru_{\acute\eta}$.
Let $U \in \Up_{n+1}^{1}$ be the only such matrix for which
$(z_0^{-1}z_1)^U = \tilde z_0^{-1}\tilde z_1$.
We define a projective transformation,
a homeomorphism from $\cL_n(z_0;z_1)$ to $\cL_n(\tilde z_0;\tilde z_1)$,
taking $\Gamma \in \cL_n(z_0;z_1)$ to
$\tilde z_0 (z_0^{-1}\Gamma)^U \in \cL_n(\tilde z_0;\tilde z_1)$.
We are particularly interested in the restriction
$\cL_{n,\conv}(z_0;z_1) \to \cL_{n,\conv}(\tilde z_0;\tilde z_1)$.
\end{rem}

The following result is a simple corollary of these observations;
compare with Lemma \ref{lemma:totallypositive}.

\begin{lemma}
\label{lemma:convex1}
For any $z \in \Bru_{\acute\eta}$ there exists a locally convex curve
$\Gamma: [0,1] \to \Spin_{n+1}$,
$\Gamma(0) = 1$, $\Gamma(\frac12) = z$, $\Gamma(1) = \hat\eta$
and $\Gamma(t) \in \Bru_{\acute\eta}$ for all $t \in (0,1)$.

Moreover, if $h: K \to  \Bru_{\acute\eta}$ is a continuous function
then there exists a continuous function $H: K \times [0,1] \to \Spin_{n+1}$ 
such that for any $s \in K$ the locally convex curve
$\Gamma_s: [0,1] \to \Spin_{n+1}$, $\Gamma_s(t) = H(s,t)$,
satisfies
$\Gamma_s(0) = 1$, $\Gamma_s(\frac12) = h(s)$, $\Gamma_s(1) = \hat\eta$
and $\Gamma_s(t) \in \Bru_{\acute\eta}$ for all $t \in (0,1)$.
\end{lemma}

\begin{proof}
As in Example \ref{example:fh}, take
$$ \fh = \sum_{i\in\nmesmo} \sqrt{i(n+1-i)}\; \fa_i \in \so_{n+1}, \qquad
\Gamma_0(t) = \exp(\pi t \fh). $$
Recall that
$\Gamma_0(0) = 1$, $\Gamma_0(\frac12) = \acute\eta$, $\Gamma_0(1) = \hat\eta$.
Equation \ref{equation:fhfhL} implies that, 
for $t\in (0,1)$, $\Gamma_0(t)=\exp\left(\pi\left(t-\frac12\right)\fh\right)=U_{1}(t)\acute{\eta}U_{2}(t)\in\Bru_{\acute{\eta}}$, 
where 
\begin{align*}
U_{1}(t) & =\acute{\eta}\exp\left(-\cot(\pi t)\fh_{L}\right)\acute{\eta}^{-1}\in\Up^{1}_{n+1}, \\
U_{2}(t) & =\exp(-\log(\sin(\pi t))[\fh_L,\fh_L^{\transpose}])
\exp(\cot(\pi t)h_L^{\transpose})\in\Up^{+}_{n+1}.
\end{align*}
Define $h_U: K \to  \Up_{n+1}^1$ by $\acute\eta^{h_U(s)} = h(s)$;
define $H(s,t) = (\Gamma_0(t))^{h_U(s)}$ and $\Gamma_s = \Gamma_0^{h_U(s)}$.
\end{proof}

If $q \in \Quat_{n+1}$ then $\Bru_q = \{q\}$.
If $z = q \acute\eta \in \widetilde B_{n+1}^{+}$, $q \in \Quat_{n+1}$,
then $\Bru_z = \cU_z$.
If $z = q \acute a_i$,
$q \in \Quat_{n+1}$,
then $\Bru_z = \{q \alpha_i(\theta), \theta \in (0,\pi)\}$
where $\alpha_i(\theta) = \exp(\theta \fa_i)$
and $\fa_i \in \so_{n+1}$ is given by Equation \ref{equation:fa}
(recall that $\alpha_i(\frac{\pi}{2}) = \acute a_i$).
We generalize this observation below.
The reader will notice the similarities between these results
and the discussion above concerning total positivity in the triangular group.

\begin{lemma}
\label{lemma:bruhatstep}
Let $q \in \Quat_{n+1}$,
$\sigma_0, \sigma_1 \in S_{n+1}$ with
$\sigma_1 \triangleleft \sigma_0 = \sigma_1 a_i$,
$z_0 = q \acute\sigma_0$, $z_1 = q \acute\sigma_1$.
Then the map
$$ \Phi: \Bru_{z_1} \times (0,\pi) \to \Bru_{z_0}, \qquad
\Phi(z,\theta) = z\alpha_i(\theta) $$
is a diffeomorphism. Similarly, if
$\sigma_1 \triangleleft \sigma_0 = a_i \sigma_1$
then the map
$$ \Phi: (0,\pi) \times \Bru_{z_1} \to \Bru_{z_0}, \qquad
\Phi(\theta,z) = \alpha_i(\theta)z $$
is a diffeomorphism.
\end{lemma}

Before presenting the proof, we see a few applications.

\begin{coro}
\label{coro:bruhatstep}
Let $\sigma = a_{i_1}\cdots a_{i_k}$ be a reduced word
(so that $k = \inv(\sigma)$).
Let $q \in \Quat_{n+1}$.
Then the map
$$ \Psi_{(q;i_1,\ldots,i_k)}:(0,\pi)^k \to \Bru_{q\acute\sigma}, \qquad
(\theta_1, \ldots, \theta_k) \mapsto
q \alpha_{i_1}(\theta_1) \cdots \alpha_{i_k}(\theta_k) $$
is a diffeomorphism.
\end{coro}

\begin{proof}
The proof is by induction on $k$.
The cases $k \le 1$ are easy and have been discussed above.
The induction step is provided by Lemma \ref{lemma:bruhatstep}.
\end{proof}

A crucial difference between this case and the triangular case
is that $(0,+\infty)$ and $\Pos_\eta$ are semigroups
(i.e., closed under sums and products, respectively)
but $(0,\pi)$ and $\Bru_{\acute\eta}$ are not.

\begin{coro}
\label{coro:bruhatproduct}
Consider $\sigma_0, \sigma_1 \in S_{n+1}$, $\sigma = \sigma_0\sigma_1$.
If $\inv(\sigma) = \inv(\sigma_0)+\inv(\sigma_1)$ then
$\Bru_{\acute\sigma_0}\Bru_{\acute\sigma_1} = \Bru_{\acute\sigma}$;
moreover, the map
$$ 
\Bru_{\acute\sigma_0} \times \Bru_{\acute\sigma_1} \to \Bru_{\acute\sigma},
\qquad (z_0,z_1) \mapsto z_0z_1 $$
is a diffeomorphism.
\end{coro}

\begin{proof}
This follows directly from Corollary \ref{coro:bruhatstep}.
\end{proof}




\begin{lemma}
\label{lemma:posbruhat}
Consider $\sigma \in S_{n+1}$.
Then $\bQ[\Pos_{\sigma}] \subset \Bru_{\acute\sigma}$.
Furthermore, if $\sigma \ne e$ then
$\acute\sigma$ does not belong to $\bQ[\Pos_{\sigma}]$.
Similarly, 
$\bQ[\Neg_{\sigma}] \subset \Bru_{\grave{\sigma}}$;
if $\sigma \ne e$ then $\grave{\sigma}$ does not belong to $\bQ[\Neg_{\sigma}]$.
\end{lemma}

\begin{proof}
The case $\sigma = e$ is trivial;
for $\sigma = a_i$ we have
$\Pos_{\sigma} = \{\lambda_i(t), t > 0\}$
and $\bQ(\lambda_i(t)) = \alpha_i(\arctan(t))$
(where $\alpha_i(\theta) = \exp(\theta \fa_i)$ and
$\lambda_i(t) = \exp(t \fl_i)$). We thus have
$$ \lim_{t \to +\infty} \bQ(\lambda_i(t)) =
\alpha_i\left(\frac{\pi}{2}\right) = \acute a_i, $$
as desired.

We proceed to the induction step.
Assume $\sigma_k = a_{i_1} \cdots a_{i_k}$ (a reduced word)
and $\sigma_{k-1} = a_{i_1} \cdots a_{i_{k-1}} \triangleleft
\sigma_k = \sigma_{k-1} a_{i_k}$.
Consider $L_k \in \Pos_{\sigma_k}$; 
write $L_k = L_{k-1} \jacobi_{i_k}(t_k)$,
$t_k \in (0,+\infty)$, $L_{k-1} \in \Pos_{\sigma_{k-1}}$.
By induction, we have $\bQ(L_{k-1}) = z_{k-1} \in \Bru_{\acute\sigma_{k-1}}$.
Consider the curves $\Gamma_L: [0,t_k] \to \Lo_{n+1}^{1}$ and
$\Gamma: [0,t_k] \to \Spin_{n+1}$ defined by
$\Gamma_L(t) = L_{k-1} \jacobi_{i_k}(t)$ and $\Gamma = \bQ \circ \Gamma_L$.
In particular, $\Gamma(0) = z_{k-1}$.
The curve $\Gamma_L$ is tangent to the vector field $X_{\fl_{i_k}}$
and therefore, from Lemma \ref{lemma:al},
the curve $\Gamma$ is tangent to the vector field $X_{\fa_{i_k}}$.
We thus have $\Gamma(t) = z_{k-1} \alpha_{i_k}(\theta(t))$
for some smooth increasing function
$\theta: [0,+\infty) \to [0,+\infty)$.
But $z_{k-1} \in \cU_1$ implies
$z_{k-1} \alpha_{i_k}(\pi) = z_{k-1} \hat a_i \in \cU_{\hat{a}_i}$ 
and therefore $z_{k-1} \alpha_{i_k}(\pi)\notin\cU_1$. 
Thus, we have $\theta: [0,+\infty) \to [0,\pi)$.
From Lemma \ref{lemma:bruhatstep},
$z_k = \bQ(L_k) \in \Bru_{\acute\sigma_k}$, as desired.

Clearly, for $\sigma \ne e$ we have $\acute\sigma \notin \cU_1$,
implying $\acute\sigma \notin \bQ[\Pos_{\sigma}]$.
The claims concerning $\Neg_{\sigma}$
follow from the claims for $\Pos_{\sigma}$
either by taking inverses or by similar arguments.
\end{proof}

\begin{coro}
\label{coro:zkLk}
Consider
$\sigma_{k-1} \triangleleft \sigma_k = \sigma_{k-1} a_{i_k} \in S_{n+1}$.
Consider $z_{k-1} \in \Bru_{\acute\sigma_{k-1}}$
and $z_k \in \Bru_{\acute\sigma_k}$,
$z_k = z_{k-1} \alpha_{i_k}(\theta_k)$, $\theta_k \in (0,\pi)$.
If $z_k \in \bQ[\Pos_{\sigma_k}]$ then $z_{k-1} \in \bQ[\Pos_{\sigma_{k-1}}]$
and $z_{k-1} \alpha_{i_k}(\theta) \in \bQ[\Pos_{\sigma_k}]$
for all $\theta \in (0,\theta_k]$.
\end{coro}

\begin{proof}
Let $\sigma_{k-1} = a_{i_1}\cdots a_{i_{k-1}}$ be a reduced word.  Let 
$$ L_k = \bL(z_k) = \lambda_{i_1}(t_1)\cdots
\lambda_{i_{k-1}}(t_{k-1}) \lambda_{i_k}(t_k). $$
Define $\tilde L_{k-1} = \lambda_{i_1}(t_1)\cdots
\lambda_{i_{k-1}}(t_{k-1})$ and $\tilde z_{k-1} = \bQ(\tilde L_{k-1})$.
The curve $\Gamma_L: [0,t_k] \to \Lo_{n+1}^1$,
$\Gamma_L(t) = \tilde L_{k-1} \lambda_{i_k}(t)$
is taken to $\Gamma = \bQ \circ \Gamma_L$ with
$\Gamma(t_k) = z_k$ and
$\Gamma(t) = \tilde z_{k-1} \alpha_{i_k}(\theta(t))$
for some strictly increasing function $\theta$.
Invertibility of the map $\Phi$ in Lemma \ref{lemma:bruhatstep}
implies that $\tilde z_{k-1} = z_{k-1}$.
Furthermore, $z_{k-1} \alpha_{i_k}(\theta) = \Gamma(t)$
for some $t \in (0,t_k]$.
\end{proof}

\begin{proof}[Proof of Lemma \ref{lemma:bruhatstep}]
We prove the first claim; the second one is similar.
Notice that $\Phi(z_1,\frac{\pi}{2}) = z_0$.

We first prove that for all $z \in \Bru_{z_1}$ and $\theta \in (0,\pi)$
we have $\Phi(z,\theta) \in \Bru_{z_0}$.
Given the initial remark and connectivity,
it suffices to prove that $\Pi(\Phi(z,\theta)) \in \Bru_{\sigma_0}$
(the unsigned Bruhat cell).
Abusing the distinction between $z \in \Spin_{n+1}$ and $\Pi(z) \in \SO_{n+1}$,
consider $z = U_1 q \acute \sigma_1 U_2 \in \Bru_{z_1}$ and $\theta \in (0,\pi)$.
We have $\Phi(z,\theta) = U_1 q \acute \sigma_1 U_2 \alpha_i(\theta)$.
We have $U_2 \alpha_i(\theta) = \acute a_i \lambda_i(t) U_3$
for some $t \in \RR$ and $U_3 \in \Up_{n+1}^{+}$
and therefore
$\Phi(z,\theta) = U_1 q \acute \sigma_0 \lambda_i(t) U_3$.
But since $\sigma_1 \triangleleft \sigma_0$, we have
$\acute \sigma_0 \lambda_i(t) = U_4 \acute \sigma_0$
where $U_4 \in \Up_{n+1}$ has at most a single nonzero nondiagonal entry
at position
$(i^{\sigma_1^{-1}}, (i+1)^{\sigma_1^{-1}}) =
((i+1)^{\sigma_0^{-1}}, i^{\sigma_0^{-1}})$.
We have $\Phi(z,\theta) = U_1 q U_4 \acute\sigma_0 U_3$, as desired.

At this point we know that 
$\Phi: \Bru_{z_1} \times (0,\pi) \to \Bru_{z_0}$
is a smooth function.
It is also injective.
Indeed, assume $z \alpha_i(\theta) = \tilde z \alpha_i(\tilde\theta)$.
If $\theta < \tilde\theta$ we have both
$z \in \Bru_{z_1}$ and
$z = \tilde z \alpha_i(\tilde\theta - \theta) \in \Bru_{z_0}$,
contradicting the disjointness of the cells.
The case $\theta > \tilde\theta$ is similar and the case
$\theta = \tilde\theta$ is trivial.

Given $U_2 \in \Up_{n+1}^{+}$, the matrix $\acute a_i U_2$ is almost upper,
with a positive entry in position $(i+1,i)$.
There exist unique $r > 0$ and $\theta \in (0,\pi)$ such that
$(\acute a_i U_2)_{i+1,i} = r\sin(\theta)$,
$(\acute a_i U_2)_{i+1,i+1} = r\cos(\theta)$.
The matrix $U_3 = \acute a_i U_2 \alpha_i(-\theta)$
therefore also belongs to $\Up_{n+1}^{+}$.
Let $\theta_i: \Up_{n+1}^{+} \to (0,\pi)$, $U_2\mapsto \theta$, 
be the smooth function defined by the above argument.

Given $z \in \Bru_{z_0}$, write $z = q U_1 \acute\sigma_0 U_2$,
$U_i \in \Up_{n+1}^{+}$.
Notice that 
$$ z \alpha_i(-\theta_i(U_2)) =
q U_1 \acute\sigma_1 (\acute a_i U_2 \alpha_i(-\theta_i(U_2))) =
q U_1 \acute\sigma_1 U_3 \in \Bru_{\sigma_1}. $$
Thus, $\Phi(z \alpha_i(-\theta_i(U_2)), \theta_i(U_2)) = z$,
proving surjectivity of $\Phi$.
Injectivity implies that even though $U_2$ is not well defined
(as a function of $z$), $\theta_i(U_2)$ is well defined 
(and smooth, again as a function of $z$):
this gives a formula for $\Phi^{-1}$ and proves its smoothness.
\end{proof}


\begin{rem}
\label{rem:bigtheta}
The following function constructed in the proof will be used later.
For $z_0 = q \acute\sigma_0$,
$q \in \Quat_{n+1}$, $\sigma_0 \in S_{n+1}$, $a_i \le_L \sigma_0$
there is a real analytic function $\Theta_i: \Bru_{z_0} \to (0,\pi)$.
Set $\sigma_1 \triangleleft \sigma_0 = \sigma_1 a_i$
and $z_1 = q \acute\sigma_1$;
we can define $\Theta_i(z) = \theta \in (0,\pi)$
if and only if $z \alpha_i(-\theta) \in \Bru_{z_1}$.
\end{rem}



\begin{lemma}
\label{lemma:pathcoordinates}
Consider $z_0 = q \acute\sigma \in \tilde B_{n+1}^{+} \subset \Spin_{n+1}$,
$q \in \Quat_{n+1}$, $\sigma \in S_{n+1}$,
$\sigma \ne \eta$, $k = \inv(\eta) - \inv(\sigma) > 0$.
There exists a smooth function 
$f = (f_1, \ldots, f_k): \cU_{z_0} \to \RR^k$
with the following properties.
For all $z \in \cU_{z_0}$,
the derivative $Df(z)$ is surjective.
For all $z \in \cU_{z_0}$,
$z \in \Bru_{z_0}$ if and only if
$f(z) = 0$.
For any smooth locally convex curve
$\Gamma: (-\epsilon,\epsilon) \to \cU_{z_0}$
we have $(f_k \circ \Gamma)'(t) > 0$ for all $t \in (-\epsilon,\epsilon)$.
\end{lemma}


\begin{proof}
We present an explicit contruction of the 
coordinate functions $f_i$.
Write
$\Lo_{n+1}^{1} = \Lo_{\sigma^{-1}} \Lo_{\sigma^{-1}\eta}$,
i.e., write $L \in \Lo_{n+1}^{1}$ as $L = L_1L_2$,
$L_1 \in \Lo_{\sigma^{-1}}$, $L_2 \in  \Lo_{\sigma^{-1}\eta}$
(see Equation \ref{equation:Upsigma} in Section \ref{sect:symmetric} 
for the subgroups $\Lo_\sigma \subseteq \Lo_{n+1}^1$,
$\Up_\sigma \subseteq \Up_{n+1}^{1}$).
Notice that if $L_1 \in \Lo_{\sigma^{-1}}$ then $z_0 L_1 = U_1 z_0$
for $U_1 = z_0 L_1 z_0^{-1} \in \Up_{\sigma}$.
Thus, every $z \in \cU_{z_0}$ can be uniquely written as
$z = \bQ(U_1 z_0 L_2)$, $U_1 \in \Up_{\sigma}$,
$L_2 \in \Lo_{\sigma^{-1}\eta}$.
Notice that if $U_1, \tilde U_1 \in \Up_{\sigma}$
and  $L_2 \in  \Lo_{\sigma^{-1}\eta}$ then
$\bQ(U_1 z_0 L_2)$ and
$\bQ(\tilde U_1 z_0 L_2)$
belong to the same Bruhat cell.
Also, $z = \bQ(U_1 z_0 L_2) \in \Bru_{z_0}$
if and only if $L_2 = I$.
The function $f$ can be defined in terms of
$z_0 L_2 \in z_0 \Lo_{\sigma^{-1}\eta}$;
in other words, we define an affine function
$f_L:  z_0 \Lo_{\sigma^{-1}\eta} \to \RR^k$ and
set $f(\bQ(U_1 z_0 L_2)) = f_L(z_0 L_2)$.

We describe a generic element of the set $z_0 \Lo_{\sigma^{-1}\eta}$,
or, more concretely, of $\Pi(z_0) \Lo_{\sigma^{-1}\eta} \subset \GL_{n+1}$.
Start with the matrix $Q_0 = \Pi(z_0)$.
In order to obtain
$M \in Q_0 \Lo_{\sigma^{-1}\eta} = z_0 \Lo_{\sigma^{-1}\eta}$,
we are free to change the empty entries of $Q_0$
which are below and to the left of nonzero entries of $Q_0$.
Call these entries $x_1, \ldots, x_k$,
where we number them as we read, top to bottom and left to right.
Apply the sign of the entry of $Q_0$ on the same row.
Thus, for instance, the matrix $Q_0$ below yields to following $M$:
\[ Q_0 = \begin{pmatrix}
0 & -1 & 0 & 0 \\ 0 & 0 & 0 & -1 \\ -1 & 0 & 0 & 0 \\ 0 & 0 & 1 & 0
\end{pmatrix}; \qquad
M = \begin{pmatrix}
0 & -1 & 0 & 0 \\ 0 & -x_1 & 0 & -1 \\ -1 & 0 & 0 & 0 \\ x_2 & x_3 & 1 & 0
\end{pmatrix} \in z_0 \Lo_{\sigma^{-1}\eta}. \]
Finally, set $f_L(M) = (x_1, \ldots, x_k)$.
If $x_k$ is in position $(i,j)$ set
$\tilde k = n - i + 2$,
$\bi_0 = \{i, \ldots, n+1\}$ and $\bi_2 = \{1, \ldots, \tilde k - 1, j\}$
(see Lemma \ref{lemma:positivespeed}).
The desired property of
$f_k = \pm (\Lambda^{\tilde k}(M))_{\bi_0,\bi_2}$
follows from
Remark \ref{rem:explicitpositivespeed}.
Equivalently, 
if $\Gamma(t) = z_0 \bQ(\Gamma_L(t))$ then
$f_k(\Gamma(t)) = (\Gamma_L(t))_{j+1,j}$,
which is clearly strictly increasing with positive derivative.
\end{proof}

\begin{rem}
\label{rem:explicitpathcoordinates}
The explicit construction of the function $f$ will be used again.
We shall also consider the smooth projection
$\Pi_{z_0}: \cU_{z_0} \to \Bru_{z_0} \subset \cU_{z_0}$ defined
by $\Pi_{z_0}(\bQ(U_1 z_0 L_2)) = \bQ(U_1 z_0)$.
\end{rem}

\section{Chop, advance and the singular set}
\label{sect:chopadvance}

The maps $\chop, \adv: \Spin_{n+1} \to \acute\eta \Quat_{n+1} \subset
\widetilde B_{n+1}^{+}$
are defined by 
\[\adv(z)=q_a \acute\eta, \quad
\chop(z) = q_c \grave\eta, \quad
z\in\Bru_{z_0}\subset\Bru_{\sigma_0}, \quad 
z_0=q_a \acute\sigma_0=q_c \grave\sigma_0. \]
For  $\sigma_1=\eta\sigma_0$, we have $\adv(z)=z_0 \longacute(\sigma_1^{-1})=z_0 (\grave\sigma_1)^{-1}$ 
and $\chop(z)\acute\sigma_1=z_0$. 
In particular, $\adv(z)=\chop(z)\hat\sigma_1$.

\begin{lemma}
\label{lemma:chopadvance}
For $z \in \Spin_{n+1}$,
let $\Gamma: (-\epsilon, \epsilon) \to \Spin_{n+1}$
be a locally convex curve such that $\Gamma(0) = z$.
There exists $\epsilon_a \in (0,\epsilon)$ such that
for all $t \in (0,\epsilon_a]$, $\Gamma(t) \in \Bru_{\adv(z)}$.
There exists $\epsilon_c \in (0,\epsilon)$ such that
for all $t \in [-\epsilon_c,0)$, $\Gamma(t) \in \Bru_{\chop(z)}$.
\end{lemma}

\begin{proof}
If necessary, apply a projective transformation so that
$z=q_a \bQ(L_0)$, $L_0\in\Pos_{\sigma_0}$.
For any locally convex curve $\Gamma$ as in the statement,
there exists $\epsilon_a\in(0,\epsilon)$ such that
the restriction $\Gamma|_{[-\epsilon_a,\epsilon_a]}$ can be
written in triangular coordinates: 
$\Gamma(t)=q_a \bQ(\Gamma_L(t))$, $\Gamma_L(0)=L_0$.
It follows from Lemma \ref{lemma:transition}
that $\Gamma_L(t) \in \Pos_{\eta}$ for any $t \in (0,\epsilon_a]$.
Thus, $\Gamma(t)\in\Bru_{\adv(z)}$ for all $t\in(0,\epsilon_a]$.
The proof for $\chop$ is similar.
\end{proof}

The chopping map was introduced in \cite{Saldanha-Shapiro}, 
where a different combinatorial description is given, together 
with the topological characterization in Lemma \ref{lemma:chopadvance}.
The notations $\ba = \grave\eta = \chop(1)$ 
and $A=\Pi(\ba)$ are used there; 
$A$ is called the \emph{Arnold matrix}.

For $z_0, z_1 \in \Spin_{n+1}$, let $\cL_n(z_0;z_1)$
be the space of locally convex curves $\Gamma: [0,1] \to \Spin_{n+1}$
such that $\Gamma(0) = z_0$ and $\Gamma(1) = z_1$.
There are many possible choices for the exact definition 
and topology of these spaces.
As a small example, one can admit only smooth curves and 
consider the (Fr\'echet) topology induced by the family of 
$C^k$-seminorms in $C^{\infty}([0,1],\Ss^n)$. 
As a large example, one can admit all the absolutely continuous 
curves (as in Section \ref{sect:triangle}, next to Equation 
\ref{equation:explicitGammamu}) with the required condition 
on the logarithmic derivative. For most arguments, the exact 
choice is immaterial. This issue is discussed in detail in 
Appendix \ref{appendix:Hilbert}.

Clearly, $\cL_n(z_0;z_1)$ is homeomorphic to $\cL_n(1;z_0^{-1}z_1)$;
let $\cL_n(z) = \cL_n(1;z)$.
If $z \in \Bru_{z_0}$ then a projective transformation
yields a homeomorphism between $\cL_n(z)$ and $\cL_n(z_0)$.
In \cite{Saldanha-Shapiro} (Prop. 6.4)
it is proved that if $\chop(z_0) = \chop(z_1)$ then
$\cL_n(z_0)$ and $\cL_n(z_1)$ are homeomorphic (for nice topologies).

Define
\[ \cL_n = \bigsqcup_{q \in \Quat_{n+1}} \cL_n(q). \]
Thus, understanding $\cL_n$ is in a sense sufficient
to understand all spaces
$\cL_n(z_0;z_1)$.
Moreover, the particularly interesting cases
$\cL_n(I)$ and, for $n$ odd, $\cL_n(-I)$,
are explicitly contained in $\cL_n$.
It follows from Proposition \ref{prop:convex} in 
Appendix \ref{appendix:convex} that 
$q = \hat\eta$ is the unique $q \in \Quat_{n+1}$
for which $\cL_n(q)$ contains convex curves:
the space $\cL_n$ has precisely $1 + 2^{(n+1)}$ connected components.

Recall that $\Bru_z \subset \Spin_{n+1}$ is open if and only if
$z \in \acute\eta\Quat_{n+1}$.
Define the {singular set} $\Sing_{n+1} \subset \Spin_{n+1}$
as the complement of the open cells:
$$ \Sing_{n+1} =
\Spin_{n+1} \smallsetminus \Bru_\eta 
=\bigsqcup_{z \in \widetilde B_{n+1}^{+} 
\smallsetminus (\acute\eta\Quat_{n+1})}\Bru_z. $$
The \emph{singular set} of a locally convex curve 
$\Gamma:[t_0,t_1]\to\Spin_{n+1}$ is
\[\sing(\Gamma) = \Gamma^{-1}[\Sing_{n+1}] \smallsetminus \{t_0,t_1\}.\]

\begin{lemma}
\label{lemma:itinerary}
Given a locally convex curve $\Gamma:[t_0,t_1]\to\Spin_{n+1}$, 
the set $\sing(\Gamma)$ is finite. 
\end{lemma}
\begin{proof}
We know from Lemma \ref{lemma:chopadvance} 
that for each $\tau\in [t_0,t_1]$ there exists 
$\epsilon>0$ such that 
$(\tau-\epsilon,\tau+\epsilon)\cap\sing(\Gamma)\subseteq\{\tau\}.$
Take a finite subcover of $[t_0,t_1]$.
\end{proof}


A curve $\Gamma\in\cL_n(z_0;z_1)$ is \emph{convex} if and only if 
$\sing(z_0^{-1}\Gamma) = \emptyset$. 
We prove in Appendix \ref{appendix:convex} 
the equivalence between this notion of convexity and
the geometric, more classical one used for instance in 
\cite{Khesin-Shapiro2, Saldanha3, Shapiro-Shapiro, Shapiro}.
In particular, if $\Gamma:J_0\to\Spin_{n+1}$ is convex 
and $J_1\subset J_0$ then $\Gamma|_{J_1}$ is convex. 
Also, if $\Gamma:J_0\to\Spin_{n+1}$ is locally convex 
with $J_0=[t_0,t_1]$ and the restrictions 
$\Gamma|_{[t_0+\epsilon,t_1-\epsilon]}$ are convex
(for all $\epsilon>0$), then 
$\Gamma$ is also convex.
It then follows that $\sing(\Gamma)=\emptyset$ implies 
that $\Gamma$ is convex.
The reciprocal is not true. Indeed, 
for any $\sigma\in S_{n+1}$, $\sigma\ne\eta$, take 
$\Gamma:[-\epsilon,\epsilon]\to\Spin_{n+1}$, 
$\Gamma(t)=\acute\sigma\exp(t\fh)$. For small 
$\epsilon>0$, $\Gamma$ is convex but 
$0\in\sing(\Gamma)$ (see also Example \ref{example:inout} below).

The following lemma is closely related to the known fact
that convex curves form a connected component of $\cL_n(\hat\eta)$:
it shows more generally that, when a curve is deformed,
points in the singular set
may join or split but never vanish or appear out of nowhere.

\begin{lemma}
\label{lemma:novanishingletter}
Let $K$ be a compact set;
let $H: K \to \cL_n(z_0;z_1)$ be a continuous function.
Let
$$ K_1 = \bigsqcup_{s \in K} (\{s\} \times \sing(H(s))) =
\{(s,t) \in K \times (0,1) \;|\; H(s)(t) \in \Sing_{n+1} \}. $$
Then $K_1$ is a compact set and satisfies the following condition:
\begin{gather*}
\forall (s_0,t_0) \in K_1, \; \forall \epsilon > 0, \;
\exists \delta > 0, \; \forall s \in K, \\
|s-s_0| < \delta \quad\to\quad
(\exists t \in (0,1), \; (s,t) \in K_1, \; |t-t_0| < \epsilon). 
\end{gather*}
\end{lemma}

We comment a little before the proof.
Recall that the Hausdorff distance between two compact nonempty sets
$X, Y \subset (0,1)$ is
$$ d_{\cH}(X,Y) = \max \left(
(\sup_{x\in X} \inf_{y \in Y} |x-y|),
(\sup_{y\in Y} \inf_{x \in X} |x-y|) \right). $$
Let $\cH((0,1)) \subset \cP((0,1))$ be the set of compact subsets of $(0,1)$;
this is a metric space with the Hausdorff distance where
the empty set is an isolated point.

\begin{coro}
\label{coro:hausdorff1}
The map $\sing: \cL_n(z_0;z_1) \to \cH((0,1))$ is continuous.
\end{coro}

\begin{proof}
It follows from the condition in Lemma \ref{lemma:novanishingletter}
that the composite $\sing \circ H$ is continuous.
Since $\cL_n(z_0;z_1)$ is metrizable and $K$ is arbitrary,
this implies the continuity of the map $\sing: \cL_n(z_0;z_1) \to \cH((0,1))$.
\end{proof}




Let $\cL_{n,\conv}(z_0;z_1) \subset \cL_n(z_0;z_1)$ 
be the subset of convex curves. 
In particular, we write $\cL_{n,\conv}(z)=\cL_{n,\conv}(1;z)$. 
The following result is well known 
\cite{Anisov, Shapiro-Shapiro, Shapiro} and is presented here
for completeness and as an example of an application.

\begin{lemma}
\label{lemma:convex2}
If $z \in \tilde{B}^+_{n+1}$ then the subset
$\cL_{n,\conv}(z) \subset \cL_n(z)$
is either empty or a contractible connected component.
It is nonempty if and only if
$\chop(z)=\acute\eta$.
\end{lemma}

\begin{proof}
From Corollary \ref{coro:hausdorff1}
and the fact that $\emptyset \in \cH((0,1))$ is an isolated point
it follows that $\cL_{n,\conv}(z)$ is a union of connected components.
It suffices to prove that it is contractible.

Consider first the case $z \in \Bru_{\acute\eta}$.
By applying a projective transformation we may assume
$z = \acute\eta = \exp(\frac{\pi}{2}\fh)$.
Take $\Gamma_0 \in \cL_{n,\conv}(\acute\eta)$,
$\Gamma_0(t) = \exp(\frac{\pi}{2}t\fh)$.
For $s \in (0,1]$, let $U_s \in \Up_{n+1}^{1}$ be such that
$\acute\eta^{U_s} = \Gamma_0(s)$.
For $\Gamma_1 \in \cL_{n,\conv}(\acute\eta)$ and $s \in (0,1]$
define $\Gamma_s  \in \cL_{n,\conv}(\acute\eta)$ by:
$$ \Gamma_s(t) = \begin{cases}
(\Gamma_1(\frac{t}{s}))^{U_s}, & t \in [0,s], \\
\Gamma_0(t), & t \in [s,1]. \end{cases} $$
The map $[0,1] \to  \cL_{n,\conv}(\acute\eta)$,
$s \mapsto \Gamma_s$ is continuous (even at $s = 0$).

The general case follows from the chopping lemma
(Prop. 6.4 in \cite{Saldanha-Shapiro});
see also Remark \ref{rem:explicitcontraction}.
\end{proof}

\begin{proof}[Proof of Lemma \ref{lemma:novanishingletter}]
Write $\tilde H(s,t) = \fF_{H(s)}(t) \in \Spin_{n+1}$
so that $\tilde H: K \times [0,1] \to \Spin_{n+1}$ is continuous.
We assume without loss of generality that
$\tilde H(s,1) = z_1 \in \Quat_{n+1}$ is fixed for all $s \in K$.
Notice that $K_2 \subseteq K \times [0,1]$ defined by
$$ K_2 = K_1 \cup (K \times \{0,1\}) = \tilde H^{-1}[X], \qquad
X = \sqcup_{\sigma \ne \eta} \Bru_{\sigma}, $$
is closed and therefore compact.
Furthermore, the sets $A_0 = \tilde H^{-1}(\Bru_{\adv(1)})$ and
$A_1 = \tilde H^{-1}(\Bru_{\chop(z_1)})$ are open and disjoint from $K_2$.
From Lemma \ref{lemma:chopadvance},
for each $s \in K$ there exists $\epsilon_s > 0$
such that
$\{s\} \times (0,\epsilon_s) \subset A_0$
and
$\{s\} \times (1-\epsilon_s,1) \subset A_1$.
By compactness of $K$ there exists $\epsilon_{\ast} > 0$ such that
$K \times (0,\epsilon_{\ast}) \subset A_0$ and
$K \times (1-\epsilon_{\ast},1) \subset A_1$,
implying the compactness of $K_1 = K_2 \smallsetminus (K \times \{0,1\})$.

Consider now $(s_0,t_0) \in K_1$.
Take $t_1 \in (t_0,1]$ (resp. $t_{-1} \in [0,t_0)$)
minimal (resp. maximal) such that
$(s_0,t_{1}) \in K_2$ (resp. $(s_0,t_{-1}) \in K_2$);
recall that $(s_0,t) \in K_2$ if and only if $\sigma(H(s_0),t) \ne e$.
Take $\epsilon_1 < \min(\epsilon, |t_0-t_1|, |t_0-t_{-1}|)$ and
$\tilde z \in \Quat_{n+1}$ such that 
$\tilde H(s_0,t_0-\epsilon_1) \in \Bru_{\chop(\tilde z)}$.
Applying a projective transformation, we may assume that
$\tilde H(s_0,t_0-\epsilon_1) \in \tilde z \bQ[\Neg_\eta] \subset \cU_{z_1}$.
Define $\Gamma_{s_0}(t) = \bL(\tilde z^{-1} \tilde H(s_0,t))$
in the maximal interval around $t_0 - \epsilon_1$.
The curve $\Gamma_{s_0}$ is locally convex and satisfies
$\Gamma_{s_0}(t_0 - \epsilon_1) \in \Neg_{\eta}$:
it can not go to infinity without first leaving $\Neg_{\eta}$
and it can not leave $\Neg_{\eta}$ before $t_0$.
It follows that there exists $\epsilon_2 < \epsilon_1$
such that $\Gamma_{s_0}$ is defined in $[t_0-\epsilon_2, t_0+\epsilon_2]$,
$\Gamma_{s_0}(t_0) \in \partial\Neg_{\eta}$.
Take $r > 0$ such that the balls of radius $r$ centered 
at $\Gamma_{s_0}(t_0 \mp \epsilon_2)$ are contained in
$\Neg_\eta$ and $\Lo_{n+1}^1 \smallsetminus \overline{\Neg_\eta}$,
respectively.
Take $\delta > 0$ such that $|s-s_0| < \delta$ implies
$\tilde H(s,t) \in \cU_{z_1}$
for all $t \in [t_0 - \epsilon_2, t_0 + \epsilon_2]$,
$|\Gamma_s(t_0-\epsilon_2) - \Gamma_{s_0}(t_0-\epsilon_2)| < r$ and
$|\Gamma_s(t_0+\epsilon_2) - \Gamma_{s_0}(t_0+\epsilon_2)| < r$
(where $\Gamma_s(t) = \bL(\tilde z^{-1} \tilde H(s,t))$).
The locally convex curve $\Gamma_s$ must cross $\partial\Neg_\eta$
at some $t \in (t_0-\epsilon_2,t_0+\epsilon_2)$:
we have $(s,t) \in K_1$, as desired.
\end{proof}

\section{Accessibility in the spin group}
\label{sect:ac}

For $z_{\bfx} \in \bQ[\Pos_\eta] \subseteq \cU_1 \cap \Bru_{\acute\eta}$
and $\sigma \in S_{n+1}$ we define
$$ \Ac_\sigma(z_{\bfx}) = \bQ[\Ac_\sigma(\bL(z_{\bfx}))]
\subset \bQ[\Pos_\sigma] \subset \Bru_{\acute\sigma}. $$
For each $z \in \Ac_\sigma(z_{\bfx})$ there exists a locally convex curve
$\Gamma: [0,1] \to \cU_1 \subset \Spin_{n+1}$
with $\Gamma(0) = z$ and $\Gamma(1) = z_{\bfx}$.
Indeed, just take a locally convex curve
$\Gamma_L: [0,1] \to \Lo_{n+1}^{1}$
with $\Gamma_L(0) = \bL(z)$ and $\Gamma_L(1) = \bL(z_{\bfx})$
and define $\Gamma = \bQ \circ \Gamma_L$.
Similarly, for $z \in \bQ[\Pos_\sigma] \smallsetminus \Ac_\sigma(z_{\bfx})$
no such curve exists.

For $z_{\bfx} \in \Bru_{\acute\eta}$, choose $U \in \Up_{n+1}^{+}$
such that $z_{\bfx} = z_0^U$, $z_0  \in \bQ[\Pos_\eta]$.
For $\sigma \in S_{n+1}$ define
$\Ac_{\sigma}(z_{\bfx}) = (\Ac_{\sigma}(z_0))^U$;
this turns out to be well defined and the properties above still hold.
We want to define $\Ac_{\sigma}(z_{\bfx})$ for any 
$z_{\bfx} \in \chop^{-1}[\{\acute\eta\}]$.
This will require a certain detour.
We shall first present a topological construction (using curves),
then an algebraic one (using coordinates)
and then finally prove their equivalence.

For $q\in \Quat_{n+1}$, set
\begin{align*}
\Bru^{0}_{q\acute\eta} &= \adv^{-1}[\{q\acute\eta\}] 
= \bigsqcup_{\sigma \in S_{n+1}} \Bru_{q\acute\sigma}, \\
\Bru^{1}_{q\grave\eta} &= \chop^{-1}[\{q\grave\eta\}]
= \bigsqcup_{\sigma \in S_{n+1}}
\Bru_{q\grave\sigma}.
\end{align*}
A locally convex curve $\Gamma: [0,1] \to \Spin_{n+1}$
satisfying $\Gamma(t) \in \Bru_{\acute\eta}$ for $t \in (0,1)$
will necessarily satisfy
$\Gamma(0) \in \Bruadv$ and $\Gamma(1) \in \Bruchop$.
Notice that $\Bru_{\acute\eta} \subseteq \Bruadv \cap \Bruchop$.
Given $\sigma \in S_{n+1}$, we have
$$ \Bru_\sigma \cap \Bruadv = \Bru_{\acute\sigma}, \qquad
\Bru_\sigma \cap \Bruchop = \Bru_{\hat\eta(\hat\sigma)^{-1}\acute\sigma} $$
and therefore
$$ \Bruadv \cap \Bruchop =
\bigsqcup_{\sigma \in S_{n+1}, \hat\sigma = \hat\eta} \Bru_{\acute\sigma}. $$
Recall that, from Lemma \ref{lemma:smalln}, 
$\Bruadv \cap \Bruchop = \Bru_{\acute\eta}$ precisely for $n \le 3$.
In order to extend locally convex curves in $\Bru_{\acute\eta}$
to the boundary and not mix up entry points with exit points
we define a new larger space:
$$ \Brujojo =
( (\adv^{-1}[\{\acute\eta\}] \times \{0\}) \sqcup
(\chop^{-1}[\{\acute\eta\}] \times \{1\}) )/\sim $$
where $(z,0)\sim(z,1)$ for $z\in\Bru_{\acute\eta}$ (and only there).
We abuse notation by writing
$$ \Bruadv \subset \Brujojo, \qquad \Bruchop \subset \Brujojo; $$
in \emph{this} context, $\Bruadv \cap \Bruchop = \Bru_{\acute\eta}$.
A locally convex curve $\Gamma: [0,1] \to \Brujojo$ corresponds to 
a locally convex curve $\Gamma_1: [0,1] \to \Spin_{n+1}$
satisfying $\Gamma_1(t) \in \Bru_{\acute\eta}$ for $t \in (0,1)$
with $\Gamma(0) = (\Gamma_1(0),0)$, $\Gamma(1) = (\Gamma_1(1),1)$.

For $z_0 \in \Bruadv$ and $z_1 \in \Bruchop$,
let $\cLjojo(z_0;z_1) \subset \cL_n(z_0;z_1)$
be the set of locally convex curves $\Gamma: [0,1] \to \Spin_{n+1}$
such that $\Gamma(0) = z_0$, $\Gamma(1) = z_1$
and $\Gamma(t) \in \Bru_{\acute\eta}$ for all $t \in (0,1)$.
For $z_0, z_1 \in \Brujojo$,
write $z_0 \ll z_1$ if and only if
$z_0 \in \Bruadv$, $z_1 \in \Bruchop$,
and $\cLjojo(z_0;z_1) \ne \emptyset$
(compare with Lemma \ref{lemma:totallypositive}).

\begin{lemma}
\label{lemma:cLjojo}
Consider $z_0 \in \Bruadv$ and $z_1 \in \Bruchop$.
The set $\cLjojo(z_0;z_1)$ is either empty or contractible.
If $z_0 \ll z_1$ then $z_0^{-1}z_1 \in \Bruchop$
and $\cLjojo(z_0;z_1) = \cL_{n,\conv}(z_0;z_1)$.
\end{lemma}


\begin{example}
\label{example:inout}
Recall from Lemma \ref{lemma:convex2} that
$\cL_{n,\conv}(z_0;z_1) \subset \cL_n(z_0;z_1)$
is a contractible connected component if $z_0^{-1}z_1 \in \Bruchop$
and is empty otherwise.
It is entirely possible to have
$z_0 \in \Bruadv$,  $z_1 \in \Bruchop$,
$z_0^{-1}z_1 \in \Bruchop$ and $z_0 \not\ll z_1$
so that $\cLjojo(z_0;z_1) = \emptyset$.
In this case, there are convex curves in $\cL(z_0;z_1)$
but they never belong to $\cLjojo(z_0;z_1)$.

A simple case is $n = 6$, $\fh$ as in Example \ref{example:fh},
$z_0 = \exp(\frac{3\pi}{4}\fh)$ and $z_1 = \exp(\frac{\pi}{4}\fh)$.
Recall that for $n = 6$ we have $\hat\eta = \exp(\pi\fh) = 1$.
We have $z_0^{-1}z_1 = \exp(\frac{\pi}{2}\fh) = \acute\eta$
and the curve $\Gamma(t) = z_0 \exp(\frac{\pi}{2} t\fh)$ is convex.
\end{example}
 
\begin{proof}[Proof of Lemma \ref{lemma:cLjojo}]
By Corollary \ref{coro:hausdorff1},
the function $\sing: \cL_n(z_0;z_1) \to \cH((0,1))$ is continuous.
By definition, $\cLjojo(z_0;z_1) = \sing^{-1}[\{\emptyset\}]$.
Since $\emptyset$ is an isolated point in $\cH((0,1))$,
the set $\cLjojo(z_0;z_1)$ is a union of
connected components of $\cL_n(z_0;z_1)$.

We know from \cite{Saldanha-Shapiro} that if $\Gamma$ is not convex
then $\Gamma$ is in the same connected component as
$\Gamma$ with added loops,
which clearly does not have empty singular set.
Thus the only connected component of $\cL_n(z_0;z_1)$
which may be contained in $\cLjojo(z_0;z_1)$
is $\cL_{n,\conv}(z_0;z_1)$.
\end{proof}

Given $z_{\bfx} \in \Bruchop$ and $\sigma \in S_{n+1}$,
consider $\Bru_{\acute\sigma} \subset \Bruadv$:
let
$$ \Ac_{\sigma}(z_{\bfx}) =
\{ z \in \Bru_{\acute\sigma} \subset \Bruadv \;|\; z \ll z_{\bfx} \}. $$

\begin{lemma}
\label{lemma:acstep}
Consider $z_{\bfx} \in \Bruchop$.
Consider $\sigma_{k-1} \triangleleft
\sigma_k = \sigma_{k-1} a_{i_k} \in S_{n+1}$,
$\inv(\sigma_k) = k$.
Consider $z_{k-1} \in \Bru_{\acute\sigma_{k-1}}$
and $z_k = z_{k-1} \alpha_{i_k}(\theta_k) \in \Bru_{\acute\sigma_k}$,
$\theta_k \in (0,\pi)$.
If $z_k \ll z_{\bfx}$ then
$z_{k-1}\alpha_{i_k}(\theta) \ll z_{\bfx}$
for all $\theta \in [0,\theta_k]$.
\end{lemma}

\begin{proof}
From Corollary \ref{coro:zkLk},
if $z_k \in \bQ[\Pos_{\sigma_k}]$
then $z_{k-1}\alpha_{i_k}(\theta) \in
\bQ[\Pos_{\sigma_{k-1}} \sqcup \Pos_{\sigma_{k}}]$.
In this case, take $L_{k} = \bL(z_{k})$ 
and $L_{\theta} = \bL(z_{k-1} \alpha_{i_k}(\theta) )$.
Consider a locally convex curve $\Gamma \in \cLjojo(z_k;z_{\bfx})$.
As in the proof of Lemma \ref{lemma:chopadvance},
take $\Gamma_L(t) = \bL(\Gamma(t))$ so that
$\Gamma_L(0) = L_k$ and, using Lemma \ref{lemma:transition},
there exists $\epsilon > 0$ such that
$\Gamma_L$ is well defined in $[0,2\epsilon]$,
$L_{\epsilon} = \Gamma_L(\epsilon) \in \Pos_\eta$.
We have $L_{\theta} \le L_k \ll L_\epsilon$
and therefore $L_{\theta} \ll L_\epsilon$
(Lemma \ref{lemma:positivesemigroup} and
Equation \ref{equation:positivesemigroup}).
By Lemma \ref{lemma:totallypositive},
there exists a locally convex curve
$\Gamma_{\epsilon}: [0,\epsilon] \to \Lo_{n+1}^1$,
$\Gamma_{\epsilon}(0) = L_{\theta}$ and
$\Gamma_{\epsilon}(\epsilon) = L_{\epsilon}$.
Define 
$$ \Gamma_1(t) = \begin{cases}
\bQ(\Gamma_{\epsilon}(t)), & t \in [0,\epsilon], \\
\Gamma(t), & t \in [\epsilon,1]. \end{cases} $$
Notice that for $t \in (0,\epsilon)$ we have
$\Gamma_{\epsilon}(t) \in \Pos_{\eta}$ and therefore
$\Gamma_1(t) \in \Bru_{\acute\eta}$.
The curve $\Gamma_1: [0,1] \to \Spin_{n+1}$
is locally convex and satisfies
$\Gamma_1(0) = z_{k-1}\alpha_{i_k}(\theta)$, $\Gamma(1) = z_{\bfx}$
and $\Gamma(t) \in \Bru_{\acute\eta}$ for all $t \in (0,1)$.
By definition, $z_{k-1} \alpha_{i_k}(\theta) \ll  z_{\bfx}$.

In general, there is an upper matrix $U \in \Up_{n+1}^{1}$
such that the correponding projective transformation
takes $z_k$ to $z_k^U = \bQ(U^{-1} z_k) \in \bQ[\Pos_{\sigma_k}]$,
reducing to the previous case.
\end{proof}

We now present an algebraic definition.
Consider $z_{\bfx} \in \Bruchop$.
Consider $\rho_0 \in S_{n+1}$ such that
$z_{\bfx} \in \Bru_{\hat\eta(\acute\rho_0)^{-1}}$,
$y_0 = z_{\bfx}^{-1} \hat\eta \in \Bru_{\acute\rho_0}$.
We first define sets
$\Ac_{(i_1,\ldots,i_k)}(z_{\bfx}) \subseteq \Bru_{\acute\sigma}$
where $\sigma = a_{i_1}\cdots a_{i_k}$ is a reduced word.
When $z_{\bfx}$ is fixed (and thus so are $y_0$ and $\rho_0$)
we write for simplicity 
$\Ac_{(i_1,\ldots,i_k)} = \Ac_{(i_1,\ldots,i_k)}(z_{\bfx})$.

Set $\sigma_j = a_{i_1} \cdots a_{i_j}$,
$\sigma_{j-1} \triangleleft \sigma_j = \sigma_{j-1} a_{i_j}$;
set $\rho_j = \rho_0 \vee \sigma_j$ so that
either $\rho_{j-1} = \rho_j$
or $\rho_{j-1} \triangleleft \rho_j = \rho_{j-1} a_{i_j}$.
Set 
$\Ac_{()} = \Bru_{1} = \{1\}$.
We assume $\Ac_{(i_1,\ldots,i_{k-1})}$ defined
and proceed to construct $\Ac_{(i_1,\ldots,i_k)}$:
\begin{gather*}
\Ac_{(i_1,\ldots,i_k)} =
\{ z_{k-1} \alpha_{i_k}(\theta_k) \;|\;
z_{k-1} \in  \Ac_{(i_1,\ldots,i_{k-1})}, \;
\theta_k \in (0,\vartheta_{i_k}(z_{k-1})) \}; 
\\
\vartheta_{i_k}: \Ac_{(i_1,\ldots,i_{k-1})} \to (0,\pi]; \qquad
\vartheta_{i_k}(z_{k-1}) = 
\begin{cases}
\pi, & \rho_{k-1} \triangleleft \rho_k, \\
\pi-\Theta_{i_k}(y_0 z_{k-1}), & \rho_{k-1} = \rho_k;
\end{cases}
\end{gather*}
notice here that
$\Theta_{i_k}: \Bru_{\acute\rho_k} \to (0,\pi)$
is well defined if $\rho_{k-1} = \rho_k$
(see Remark \ref{rem:bigtheta} for the definition of $\Theta_{i_k}$).

\begin{lemma}
\label{lemma:quasiproduct}
The sets $\Ac_{(i_1, \ldots, i_j)}$, $1 \le j \le k$,
defined above satisfy
$$ \Ac_{(i_1,\ldots,i_j)} =
\{ \alpha_{i_1}(\theta_1) \cdots \alpha_{i_j}(\theta_j) \;|\;
(\theta_1, \ldots, \theta_j) \in X_j \}
\subseteq
\Bru_{\acute\sigma_j} \cap
(y_0^{-1} \Bru_{\acute\rho_j}) $$
where $(X_j)_{1 \le j \le k}$ is a quasiproduct;
$\Ac_{(i_1,\ldots,i_j)}$ is diffeomorphic to $\RR^j$.
\end{lemma}

Notice that the inclusion in the statement is necessary to make sense
of the definition of $\vartheta_{i_k}$.
The reader should compare this result with
Lemma \ref{lemma:triangularquasiproduct}.

\begin{proof}
The proof is by induction on $k$; the case $k = 0$ is trivial.
Take
$z_k = z_{k-1} \alpha_{i_k}(\theta_k) \in \Ac_{(i_1, \ldots, i_k)}$,
$z_{k-1} \in  \Ac_{(i_1, \ldots, i_{k-1})}$,
$\theta_k \in (0,\vartheta_{i_k}(z_{k-1}))$.
We assume by induction hypothesis that
$\Ac_{(i_1, \ldots, i_{k-1})} \subseteq \Bru_{\acute\sigma_{k-1}}$.
We therefore have
$z_k \in \Bru_{\acute\sigma_{k-1}} \Bru_{\acute a_{i_k}}
= \Bru_{\acute\sigma_k}$
(the last equation follows from
Lemma \ref{lemma:bruhatstep}).
We also assume by induction hypothesis that
$y_{k-1} = y_0 z_{k-1} \in \Bru_{\acute\rho_{k-1}}$.
If $\rho_{k-1} \triangleleft \rho_k$,
Lemma \ref{lemma:bruhatstep} implies
$y_k = y_0 z_{k} = y_{k-1} \alpha_{i_k}(\theta_k)
\in \Bru_{\acute\rho_{k-1}} \Bru_{\acute a_{i_k}} = 
\Bru_{\acute\rho_k}$.
If $\rho_{k-1} = \rho_k$,
let $\tilde\rho \triangleleft \rho_k = \tilde\rho a_{i_k}$.
The map
$\Theta_{i_k}: \Bru_{\acute\rho_{k-1}} \to (0,\pi)$ is well defined:
take $\tilde\theta = \Theta_{i_k}(y_{k-1})$
and $\tilde y \in \Bru_{\tilde\rho}$
such that
$y_{k-1} = \tilde y \alpha_{i_k}(\tilde\theta)$.
By our recursive definition,
$\tilde\theta + \theta_k < \pi$;
by Lemma \ref{lemma:bruhatstep},
$y_{k} =
\tilde y \alpha_{i_k}(\tilde\theta + \theta_k) \in \Bru_{\acute\rho_k}$.
\end{proof}



We now prove that the two definitions are equivalent.

\begin{lemma}
\label{lemma:twodefs}
Consider $z_{\bfx} \in \Bruchop$ and
$\sigma_k = a_{i_1}\cdots a_{i_k}$ a reduced word in $S_{n+1}$.
Then $\Ac_{\sigma_k}(z_{\bfx}) = \Ac_{(i_1,\ldots,i_k)}(z_{\bfx})$.
\end{lemma}

\begin{proof}
The proof is by induction on $k$; the case $k = 0$ is trivial.
Assume therefore
$\Ac_{\sigma_{k-1}}(z_{\bfx}) = \Ac_{(i_1,\ldots,i_{k-1})}(z_{\bfx})$
for $\sigma_{k-1} = a_{i_1}\cdots a_{i_{k-1}}$.

Consider $z_k = z_{k-1} \alpha_{i_k}(\theta_k)$,
$z_{k-1} \in \Bru_{\acute\sigma_{k-1}}$, $z_k \in \Bru_{\acute\sigma_k}$,
$\theta_k \in (0,\pi)$.
It follows from Lemma \ref{lemma:acstep}
that $z_k \in \Ac_{\sigma_k}$ implies
$z_{k-1} \in \Ac_{\sigma_{k-1}}$
and therefore
$$ \Ac_{\sigma_k} \subseteq \Ac_{\sigma_{k-1}} \Bru_{\acute a_{i_k}},
\qquad
\Ac_{(i_1,\ldots,i_k)} \subseteq \Ac_{\sigma_{k-1}} \Bru_{\acute a_{i_k}};
$$
we have to prove that these two sets are equal.

Given $z_{k-1} \in \Ac_{\sigma_{k-1}}$,
let $J_{z_{k-1}} \subseteq (0,\pi)$ be the set such that,
for all $\theta_k \in (0,\pi)$,
$$ \theta_k \in J_{z_{k-1}} \qquad \Longleftrightarrow \qquad
z_{k-1} \alpha_{i_k}(\theta_k) \in \Ac_{\sigma_k}.  $$
It follows from Lemma \ref{lemma:acstep}
that $J_{z_{k-1}}$ is either empty
or an initial interval.

We claim that $J_{z_{k-1}}$ is not empty.
By appying a projective transformation,
we may assume that $z_{k-1} \in \bQ[\Pos_{\sigma_{k-1}}]$
so that $z_{k-1} = \bQ(L_{k-1})$, $L_{k-1} \in \Pos_{\sigma_{k-1}}$.
Take $\Gamma \in \cLjojo(z_{k-1};z_{\bfx})$.
Define $\Gamma_L = \bL \circ \Gamma$,
with maximal connected domain containing $t = 0$.
Consider $t_\bullet > 0$ in this domain
and $L_{\bullet} = \Gamma_L(t_\bullet)$,
$L_{\bullet} \in \Pos_{\eta}$, $L_{k-1} \ll L_{\bullet}$.
Take $t_k > 0$ such that
$L_{k-1}\lambda_{i_k}(t_k) \ll L_{\bullet}$;
define $\theta_k > 0$ by
$\bQ(L_{k-1}\lambda_{i_k}(t_k)) = z_{k-1} \alpha_{i_k}(\theta_k)$.
Take $\Gamma_{L,1}:[0,t_\bullet]\to \Lo_{n+1}^1$ locally convex
such that $\Gamma_{L,1}(0) = L_{k-1}\lambda_{i_k}(t_k)$
and $\Gamma_{L,1}(t_\bullet) = L_{\bullet}$.
Finally, take $\Gamma_1: [0,1] \to \Spin_{n+1}$,
$$ \Gamma_1(t) = \begin{cases}
\bQ(\Gamma_{L,1}(t)), & t \in [0,t_\bullet], \\
\Gamma(t), & t \in [t_\bullet,1]. \end{cases} $$
This curve is locally convex so that
$z_{k-1} \alpha_{i_k}(\theta_k) \in \Ac_{\sigma_k}$
and $\theta_k \in J_{z_{k-1}}$, as claimed.

We claim that $J_{z_{k-1}}$ is open.
Assume by contradiction $\theta_k^\star = \max(J_{z_{k-1}})$,
$z_k^\star = z_{k-1} \alpha_{i_k}(\theta_k^\star)$.
By appying a projective transformation,
we may assume that $z_k^\star \in \bQ[\Pos_{\sigma_{k}}]$.
As in the previous paragraph, take a curve $\Gamma$
going from $z_k^\star$ to $z_{\bfx}$,
use $\bL$ to take its initial segment to $\Lo_{n+1}^{1}$
and slightly perturb it to obtain
$\theta_k \in J_{z_{k-1}}$, $\theta_k > \theta_k^\star$.
The argument is so similar that we feel a repetition is pointless.

At this point we know that there exists a function
$\tilde\vartheta_{i_k}: \Ac_{\sigma_{k-1}} \to (0,\pi]$
such that $J_{z_{k-1}} = (0,\tilde\vartheta_{i_k}(z_{k-1}))$.
We are left with proving that
$\vartheta_{i_k} = \tilde\vartheta_{i_k}$.

We first prove that
$\tilde\vartheta_{i_k}(z_{k-1}) \le \vartheta_{i_k}(z_{k-1})$
for all $z_{k-1}$.
If $\rho_{k-1} \triangleleft \rho_k$ then
$\vartheta_{i_k}(z_{k-1}) = \pi$ and we are done.
If $\rho_{k-1} = \rho_k$,
take $\theta_k^{\bullet} = \vartheta_{i_k}(z_{k-1})$,
$z_k^{\bullet} = z_{k-1}\alpha_{i_k}(\theta_k^{\bullet})$
and $y_k^{\bullet} = y_0 z_k$.
Recall that in this case there exists $\rho_{\bullet} \in S_{n+1}$,
$\rho_{\bullet} \triangleleft \rho_{k-1} = \rho_{k} = \rho_{\bullet} a_{i_k}$.
By definition of $\vartheta_{i_k}$,
$y_k^{\bullet} \in \Bru_{\acute\rho_{\bullet} \hat a_{i_k}}$
so that $\adv(y_k^\bullet) = {q^\bullet \acute\eta}$
for $q^\bullet \in \Quat_{n+1}$, $q^\bullet \ne 1$.
Thus any locally convex curve starting at $y_k^{\bullet}$
immediately enters $\Bru_{q^\bullet \acute\eta}$.
Thus there is no convex curve going from $y_k^{\bullet}$ to $\hat\eta$
and therefore no convex curve going from $z_k^{\bullet}$ to $z_{\bfx}$.
It follows that $z_k^{\bullet} \notin \Ac_{\sigma_k}(z_{\bfx})$
and therefore $\theta_k^{\bullet} \ge \tilde\vartheta_{i_k}(z_{k-1})$,
proving our claim.

We finally prove that
$\tilde\vartheta_{i_k}(z_{k-1}) \ge \vartheta_{i_k}(z_{k-1})$.
Consider $\theta_k <  \vartheta_{i_k}(z_{k-1})$,
$z_k = z_{k-1} \alpha_{i_k}(\theta_k)$
and $y_k = y_{k-1} \alpha_{i_k}(\theta_k) = y_0 z_k \in \Bru_{\acute\rho_k}$.
Notice that
$z_{k-1} \alpha_{i_k}(\theta) \in \Bruadv$ and
$y_{k-1} \alpha_{i_k}(\theta) \in \Bruadv$
for all $\theta \in [0,\theta_k]$.
By compactness and Lemma \ref{lemma:chopadvance},
there exists $c > 0$ such that for all $\theta \in [0,\theta_k]$
and for all $t \in (0,c]$ we have
both $z_{k-1} \alpha_{i_k}(\theta) \exp(t\fh) \in \Bru_{\acute\eta}$ and
$y_{k-1} \alpha_{i_k}(\theta) \exp(t\fh) \in \Bru_{\acute\eta}$.
Apply Lemma \ref{lemma:convex1} to construct 
a family $H: [0,\theta_k] \times [\frac12,1] \to \Spin_{n+1}$
of convex curves $H(\theta): [\frac12,1] \to \Spin_{n+1}$
going from $y_{k-1} \alpha_{i_k}(\theta) \exp(c\fh)$ to $\hat\eta$.
Extend this to 
$H: [0,\theta_k] \times [0,1] \to \Spin_{n+1}$
by defining $H(\theta)(t) = 
y_{k-1} \alpha_{i_k}(\theta) \exp(2ct\fh) \in \Bru_{\acute\eta}$
for $t \in [0,\frac12]$.
For each $\theta \in [0,\theta_k]$
the arc $H(\theta): [0,1] \to \Spin_{n+1}$ is convex:
indeed, for $t \in (0,1)$ we have  that
$H(\theta)(t) \in \Bru_{\acute\eta}$; the claim follows
from Proposition \ref{prop:convex}.


Multiply by $y_0^{-1}$ to obtain the family $y_0^{-1}H$
of convex curves $\Gamma_{\theta} = y_0^{-1} H(\theta): [0,1] \to \Spin_{n+1}$
going from $z_{k-1} \alpha_{i_k}(\theta)$ to $z_{\bfx}$.
We prove that for all $\theta$ we have
$\Gamma_{\theta} \in \cLjojo(z_{k-1} \alpha_{i_k}(\theta); z_{\bfx})$,
i.e., that $\Gamma_{\theta}(t) \in \Bru_{\acute\eta}$
for all $t \in (0,1)$.
We know that $\Gamma_0$ is convex and that 
$z_{k-1} \in \Ac_{\sigma_{k-1}}(z_{\bfx})$
and therefore from Lemma \ref{lemma:cLjojo},
that $\Gamma_0 \in \cLjojo(z_{k-1}; z_{\bfx})$.
We know by construction that 
$\Gamma_{\theta}(t) \in \Bru_{\acute\eta}$
for all $t \in (0,\frac12)$.
Apply again Lemma \ref{lemma:convex1} to construct 
convex arcs $\tilde\Gamma_{\theta}:[0,\frac12] \to \Spin_{n+1}$
from $\tilde\Gamma_\theta(0) = 1$ to
$\tilde\Gamma_\theta(\frac12) = \Gamma_\theta(\frac12)$.
Extend $\tilde\Gamma_\theta$ to $[0,1]$
by $\tilde\Gamma_\theta(t) = \Gamma_\theta(t)$ for $t \in [\frac12,1]$.
We have $\sing(\tilde\Gamma_0) = \emptyset$.
Also, from Corollary \ref{coro:hausdorff1},
$\sing(\tilde\Gamma_\theta)$ is a continuous function of $\theta$.
Since $\emptyset \in \cH((0,1))$ is an isolated point,
we have $\sing(\tilde\Gamma_\theta) = \emptyset$ for all $\theta$,
as desired.
This implies that $z_k \ll z_{\bfx}$
and therefore $\theta_k < \tilde\vartheta_{i_k}(z_{k-1})$.
Since this holds for any $\theta_k <  \vartheta_{i_k}(z_{k-1})$
we have $\tilde\vartheta_{i_k}(z_{k-1}) \ge \vartheta_{i_k}(z_{k-1})$,
completing our proof.
\end{proof}

\begin{rem}
\label{rem:explicitcontraction}
We saw in Lemmas \ref{lemma:convex2} and 
\ref{lemma:cLjojo} that,
given $z_0 \in \Bruadv$ and $z_1 \in \Bruchop$, 
the set $\cLjojo(z_0;z_1)$ is either empty
or equal to $\cL_{n,\conv}(z_0;z_1)$ and contractible.
In Lemma \ref{lemma:cLjojo} we saw an explicit contraction
if $z_0^{-1}z_1 \in \Bru_{\acute\eta}$
but otherwise used the chopping lemma.
We now present a more explicit contraction in general.

For any $\Gamma \in \cLjojo(z_0;z_1)$, we have
$(\Gamma(0))^{-1}\Gamma(\frac12)  \in \Bru_{\acute\eta}$ and
$(\Gamma(\frac12))^{-1}\Gamma(1)  \in \Bru_{\acute\eta}$.
Apply the contraction in the proof of Lemma \ref{lemma:convex2}
to each arc, leaving $\Gamma(\frac12)$ fixed.
This takes us to a set of curves parametrized
by $\Gamma(\frac12) \in z_0 \Ac_{\eta}(z_0^{-1}z_1)$.
We now know that $\Ac_{\eta}(z_0^{-1}z_1)$ is diffeomorphic to $\RR^m$
(with a rather explicit diffeomorphism).
\end{rem}

\section{Itineraries and paths of curves}
\label{sect:paths}

Given $\Gamma \in \cL_n(z_0;z_\bullet)$, 
$\sing(\Gamma) = \{t_1 < \cdots < t_\ell\}$,
let its \emph{itinerary} and \emph{path} be
$$ \iti(\Gamma) = (\sigma_1, \ldots, \sigma_\ell), \quad
\pathiti(\Gamma) = (z_1, \ldots, z_\ell), \quad
z_j = \Gamma(t_j) \in \Bru_{\eta\sigma_j}. $$
Thus, $\iti(\Gamma)$ is a word in
$\Word_n = (S_{n+1}\smallsetminus \{e\})^{\ast}$;
here $S_{n+1} \smallsetminus \{e\}$ is the alphabet.
For $w = (\sigma_1, \ldots, \sigma_\ell) \in \Word_n$,
let its \emph{length} be $\ell(w) = \ell=|\sing(\Gamma)| \in \NN$.
From Lemma \ref{lemma:cLjojo}, 
$\Gamma \in \cL_n(\hat\eta)$ is convex if and only if 
$\iti(\Gamma)$ is the empty word of length $0$.
Define the stratification
$$ \cL_n = \bigsqcup_{w \in \Word_n} \cL_n[w];
\qquad
\cL_n[w] = 
\{ \Gamma \in \cL_n \;|\; \iti(\Gamma) = w \}. $$
One of our aims is to describe these strata.

\begin{example}
Consider again the case $n = 2$.
We draw the locally convex curves $\gamma: J \to \Ss^2$
instead of the corresponding locally convex curves
$\Gamma = \fF_{\gamma}: J \to \Ss^2$.
A letter $a = a_1$ corresponds to the curve $\gamma$
transversally crossing the equator (i.e., the great circle $x_3 = 0$)
at a point different from $\pm e_1$.
A letter $b = a_2$ occurs when the tangent geodesic (great circle)
to $\gamma$ at $t$ includes the points $\pm e_1$
but the $x_3$-coordinate of $\gamma(t)$ is non-zero.
A letter $[ab]$ indicates that the curve is tangent to the equator,
but not at $\pm e_1$.
A letter $[ba]$ declares that the curve crosses the equator transversally
and not at $\pm e_1$.
Finally, $[aba]$ proclaims that the curve is tangent to the equator
at $\pm e_1$.
Figure \ref{fig:aba} shows a two-parameter family
of (portions of) curves in $\cL_2$ illustrating all these cases.
Notice that we also write a word as a string of letters. 
For instance, $baba=(b,a,b,a)$ and $b[ab]=(b,ab)$. 
Square brackets are used to avoid confusion between, say, 
$a[ba]=(a,ba)$, $[aba]=(aba)$ and $aba=(a,b,a)$, 
of respective lengths $2$, $1$ and $3$. 
\end{example}

\begin{figure}[ht]
\def\svgwidth{10cm}
\centerline{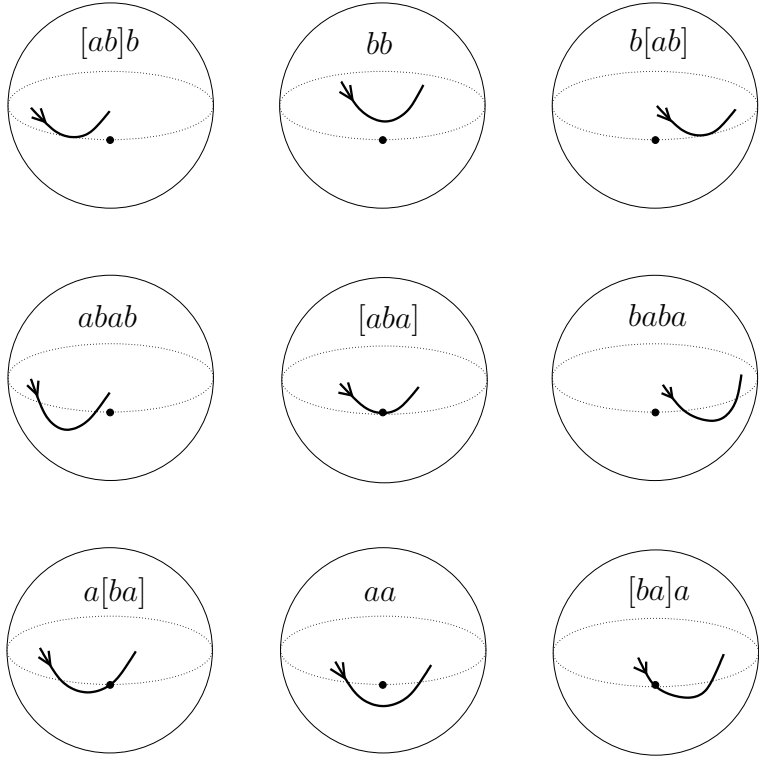}
\caption{A family of curves in $\cL_2$.
The equator is dashed and the fat dot indicates $e_1$.
The vector $e_2$ is at the right.}
\label{fig:aba}
\end{figure}

Given the path $(z_1, \ldots, z_\ell)$ of some $\Gamma \in \cL_n$,
it is easy to determine
the corresponding itinerary $w = (\sigma_1, \ldots, \sigma_\ell)$.
Conversely, given an itinerary 
$w = (\sigma_1, \ldots, \sigma_\ell) \in \Word_n$,
define $B(w,j) \in \widetilde B_{n+1}^{+}$ 
for $j \in \ZZ$, $0 \le j \le \ell+1$ 
and $B(w,j+\frac12) \in \widetilde B_{n+1}^{+}$ 
for $j \in \ZZ$, $0 \le j \le \ell$ by
\begin{equation}
\label{eq:Bwj}
\begin{gathered}
B(w,0) = 1, \quad B\left(w,\frac12\right) = \acute\eta, \quad
B(w,j) = B\left(w,j-\frac12\right) \acute\sigma_j, \\
B\left(w,j+\frac12\right) = B\left(w,j-\frac12\right) \hat\sigma_j, \quad
B(w,\ell+1) = B\left(w,\ell+\frac12\right) \acute\eta.
\end{gathered}
\end{equation}
In particular, we have $B(w,\ell+1) =
\acute\eta\hat w\acute\eta\in\Quat_{n+1}$, 
where we define the hat of a word by 
$\hat w=\hat\sigma_1\cdots\hat\sigma_\ell\in\Quat_{n+1}$.
We adopt here the conventions $t_0 = 0$, $t_{\ell+1} = 1$,
$z_0 = 1$, $z_{\ell+1} = B(w,\ell+1)$, $\sigma_0 = \sigma_{\ell+1} = \eta$.
It follows from Lemma \ref{lemma:chopadvance}
that if $\Gamma \in \cL_n[w]$ and
$\sing(\Gamma) = \{t_1 < \cdots < t_\ell \}$ then
\[
\Gamma \in \cL_n(\acute\eta\hat w \acute\eta), \qquad
\Gamma(t_j) \in \Bru_{B(w,j)}, \qquad
\forall t \in (t_j,t_{j+1}), \; \Gamma(t) \in \Bru_{B\left(w,j+\frac12\right)}. \]
Thus, if $\Gamma \in \cL_n[w]$ then
\( \pathiti(\Gamma) \in \Bru_{B(w,1)} \times \cdots \times \Bru_{B(w,\ell)} \).

Given $w=(\sigma_1,\cdots,\sigma_\ell) \in \Word_n$ 
and $j \in \ZZ$, $0 \le j \le \ell$, 
define $q_j \in \Quat_{n+1}$ by
\begin{gather*}
B(w,j) = q_j \longacute(\eta\sigma_j) \in q_j \Bruadv, \qquad
B\left(w,j+\frac12\right) = q_j \acute\eta, \\
B(w,j+1) = q_j \hat\eta \longgrave(\eta\sigma_{j+1})
\in q_j \Bruchop. 
\end{gather*}
A sequence $(z_1, \ldots, z_\ell) \in \Bru_{B(w,1)} \times \cdots \times \Bru_{B(w,\ell)}$
is an \emph{accessible path} for $w$ if
$$ \forall j \in \llbracket \ell \rrbracket, \;
q_j^{-1} z_j \in \Ac_{\eta\sigma_j}(q_j^{-1} z_{j+1}). $$
Let $\Pathiti(w) \subseteq \Bru_{B(w,1)} \times \cdots \times \Bru_{B(w,\ell)}$
be the set of accessible paths for $w$.

\begin{lemma}
\label{lemma:accessible}
Consider $w \in \Word_n$.
For any $ \Gamma \in \cL_n[w]$, 
$\pathiti(\Gamma)$ is accessible, i.e.,
belongs to $\Pathiti(w)$.
\end{lemma}

\begin{proof}
Consider $t_j < t_{j+1}$ and the arc $q_j^{-1} \Gamma|_{[t_j,t_{j+1}]}$.
Except for the modified domain, this arc belongs to
$\cLjojo(q_j^{-1}\Gamma(t_j); q_j^{-1}\Gamma(t_{j+1}))$
and therefore
$q_j^{-1}\Gamma(t_j) \in \Ac_{\eta\sigma_j}(q_j^{-1}\Gamma(t_{j+1}))$,
as desired.
\end{proof}

\begin{lemma}
\label{lemma:path}
Consider $w \in \Word_n$.
For any accessible path $(z_1,\ldots,z_\ell) \in \Pathiti(w)$, the set
$\{ \Gamma \in \cL_n[w] \;|\; \pathiti(\Gamma) = (z_1, \ldots, z_\ell) \}$
is contractible (and nonempty).
\end{lemma}

\begin{proof}
The set of sets $\{t_1 < \cdots < t_\ell \} \in \cH((0,1))$ is contractible.
At this point, the values of $q_j \in \Quat_{n+1}$,
of $t_j < t_{j+1}$,
of $z_j \in q_j \Bru_{\longacute(\eta\sigma_j)}$
and of $z_{j+1} \in q_j \Bruchop$ 
with $q_j^{-1}z_j \in \Ac_{\eta\sigma_j}(q_j^{-1} z_{j+1})$
are all given. 
The set of arcs
$\Gamma: [t_j,t_{j+1}] \to \Spin_{n+1}$
with $\Gamma(t_j) = z_j$, $\Gamma(t_{j+1}) = z_{j+1}$
and $\Gamma(t) \in q_j \Bru_{\acute\eta}$ for all $t \in (t_j,t_{j+1})$
is homeomorphic to $\cLjojo(q_j^{-1}z_j; q_j^{-1}z_{j+1})$;
by Lemma \ref{lemma:cLjojo}, this set is contractible
(with an explicit contraction given by Remark \ref{rem:explicitcontraction}).
Concatenate the above arcs to construct $\Gamma$;
this yields the desired result.
\end{proof}

\begin{lemma}
\label{lemma:Path}
Consider $w \in \Word_n$;
the set $\Pathiti(w) \subseteq \Bru_{B(w,1)} \times \cdots \times \Bru_{B(w,\ell)}$
is diffeomorphic to $\RR^d$,
$d = \inv(\eta\sigma_1) + \cdots + \inv(\eta\sigma_\ell)$.
In particular, $\Pathiti(w)$ is contractible (and nonempty).
\end{lemma}

\begin{proof}
Start constructing the set from the $\ell$-th coordinate $\Bru_{B(w,\ell)}$
and proceed backwards.
Use Lemma \ref{lemma:quasiproduct}
for the inductive step.
The set $\Pathiti(w)$ is parametrized by a quasiproduct.
\end{proof}


\begin{lemma}
\label{lemma:cLnw}
Consider $w \in \Word_n$; the set
\( \cL_n[w] \subset \cL_n(\acute\eta\hat w\acute\eta) \)
is contractible (and nonempty).
\end{lemma}

\begin{proof}
Let $\Pathiti_1(w) \subset \cL_n[w]$ be the set of paths $\Gamma$
such that the arcs $\Gamma|_{[t_{i-1},t_i]}$ are the base points of
the contractible sets $\cL_{n,\conv}(\Gamma(t_{i-1});\Gamma(t_i))$.
Here we assume that $\sing(\Gamma) = \{t_1 < \cdots < t_\ell \}$;
we may use the construction in Remark \ref{rem:explicitcontraction}
to select a basepoint.

Lemma \ref{lemma:path} gives us a deformation retract
from $\cL_n[w]$ to $\Pathiti_1(w)$,
a homotopy
$H_0: [0,1] \times \cL_n[w] \to \cL_n[w]$
which starts with an arbitrary curve $\Gamma_0 \in \cL_n[w]$
and deforms it: $\Gamma_s = H_0(s,\Gamma_0)$ for $s \in [0,1]$.
The homotopy satisfies
$\sing(\Gamma_s) = \sing(\Gamma_0) = \{ t_1 < \cdots < t_\ell \}$ and
$\pathiti(\Gamma_s) = \pathiti(\Gamma_0)$ for all $s \in [0,1]$.
We have $\Gamma_1 \in \Pathiti_1(w)$, i.e.,
the arcs $\Gamma_1|_{[t_{i-1},t_i]}$ are the base points of
the contractible sets
$\cL_{n,\conv}(\Gamma_0(t_{i-1});\Gamma_0(t_i))$.
Also, if $\Gamma_0 \in \Pathiti_1(w)$
then $\Gamma_s = \Gamma_0$ for all $s \in [0,1]$.

Let $\Pathiti_2(w) \subset \Pathiti_1(w)$
be the set of paths $\Gamma \in \Pathiti_1(w)$
such that $\sing(\Gamma) = \{ \frac{1}{\ell+1} < \cdots < \frac{\ell}{\ell+1} \}$.
There is an easy deformation retract
$H_1: [1,2] \times \Pathiti_1(w) \to \Pathiti_1(w)$
from $\Pathiti_1(w)$ to $\Pathiti_2(w)$:
affinely reparametrize each interval $[t_{i-1},t_i]$.

Lemma \ref{lemma:accessible} shows that
$\Pathiti_2(w)$ is homeomorphic to $\Pathiti(w)$:
the homeomorphism takes $\Gamma$ to $\pathiti(\Gamma)$.
Lemma \ref{lemma:Path} shows us how to construct a homotopy
$\tilde H_2: [2,3] \times \Pathiti(w) \to \Pathiti(w)$
with $\tilde H_2(2,p) = p$ and $\tilde H_2(3,p) = p_0$
where $p_0 \in \Pathiti(w)$ is a base point.
Compose with the homeomorphism above to define
a deformation retract from $\Pathiti_2(w)$ to a point.
Concatenate $H_0, H_1, H_2$ to construct the desired contraction.
\end{proof}

In Appendix \ref{appendix:Hilbert} 
we endow $\cL_n(z_0;z_1)$ and $\cL_n$ with smooth Hilbert manifold structures.
The construction is similar to that given in \cite{Klingenberg2}
for a different but similar space of curves.
Our structure allows for curves $\Gamma$ which are, say,
piecewise $C^1$ with logarithmic derivative 
$(\Gamma(t))^{-1}\Gamma'(t)$ a positive linear combination
of the vectors $\fa_i$.
Here by piecewise $C^1$ we mean that there exists a finite family 
of compact intervals $[0,t_1], [t_1,t_2], \ldots, [t_k,1]$
covering $[0,1]$
such that $\Gamma$ is of class $C^1$ in each interval $[t_i,t_{i+1}]$.

The proof of Lemma \ref{lemma:cLnw} above
obtains a rather explicit contraction.
A slightly shorter proof is possible using
the metrizable topological manifold structure provided by 
Lemma \ref{lemma:submanifold} below, Theorem 15 of \cite{Palais} 
and the long exact sequence of homotopy groups for the fibration 
$\pathiti:\cL_n[w]\to\Pathiti(w)$, via  
Lemmas \ref{lemma:accessible}, \ref{lemma:path} and \ref{lemma:Path}. 
We prefer to be more self-contained.
For everything that follows, however, we cannot postpone 
any longer the adoption of a manifold structure 
for the spaces $\cL_n(z_0;z_1)$ (see Appendix \ref{appendix:Hilbert}).

Let $M_0$ be a (finite or infinite dimensional) manifold and $M_1 \subset M_0$:
the subset $M_1$ is
a (globally) \emph{collared topological submanifold of codimension $d$}
if and only if there exists an open set $\hat A_0$,
$M_1 \subset \hat A_0 \subset M_0$,
which is a \emph{tubular neighborhood} of $M_1$
(based on \cite{Brown}).
We say that $\hat A_0$ as above is a tubular neighborhood if
there exist
an open neighborhood $B \subset \RR^d$, $0 \in B$,
a continuous projection $\Pi: \hat A_0 \to M_1 \subset \hat A_0$
and a continuous map $\hat F: \hat A_0 \to B$
such that the map $(\Pi,\hat F): \hat A_0 \to M_1 \times B$ is a homeomorphism.
Recall that $\Pi$ being a projection implies $\Pi\circ\Pi = \Pi$.
Compact oriented surfaces of class $C^2$ contained in $M_0 = \RR^3$
are examples of collared topological submanifolds:
in this case $\Pi$ can be taken to be the normal projection.

For $\sigma \in S_{n+1} \smallsetminus \{e\}$ set
$\dim(\sigma) = \inv(\sigma) - 1$;
for $w = (\sigma_1, \ldots, \sigma_\ell) \in \Word_n$ set
$\dim(w) = \dim(\sigma_1) + \cdots + \dim(\sigma_\ell)$.
In Section \ref{sect:valid} we shall construct a CW complex $\cD_n$
with one cell $c_w$ for each word $w \in \Word_n$;
we shall have $\dim(c_w) = \dim(w)$.

\begin{lemma}
\label{lemma:submanifold}
Consider $w \in \Word_n$;
The set
\( \cL_n[w] \subset \cL_n(\acute\eta\hat w\acute\eta) \)
is a collared topological submanifold of codimension $\dim(w)$.
\end{lemma}

It turns out that $\cL_n[w]$ is not a submanifold of class $C^1$;
see Remark \ref{rem:cLnwC1} below.
Before proving Lemma \ref{lemma:submanifold}
we state and prove a preliminary result.

\begin{lemma}
\label{lemma:singleletter}
Consider $\sigma \in S_{n+1}$, $\sigma \ne \eta$
and $z_0 = q \acute\sigma \in \tilde B_{n+1}^{+}$, $q \in \Quat_{n+1}$.
If $\Gamma: [-\epsilon,\epsilon] \to \cU_{z_0}$ is locally convex
and there exists $t_1 \in [-\epsilon,\epsilon]$ with 
$\Gamma(t_1) \in \Bru_{z_0}$ then $\sing(\Gamma) = \{t_1\}$.
\end{lemma}

\begin{proof}
Consider a projective transformation $\phi$ for which
$\phi(z_0) \in q \bQ[\Pos_\sigma] \subset \cU_q$.
By continuity, there exists an open set $A \subset \cU_{z_0}$, $z_0 \in A$,
such that $\phi[A] \subset \cU_q$.
We may furthermore assume that $z \in \phi[A] \cap \Bru_\sigma$
implies $z \in q \bQ[\Pos_\sigma]$.

Consider $\Gamma$ as in the statement.
Apply triangular coordinates to define the curve
$\Gamma_L: [-\epsilon,\epsilon] \to \Lo_{n+1}^{1}$,
$\Gamma(t) = z_0 \bQ(\Gamma_L(t))$.
For $\lambda \in [1,+\infty)$ define
\[ \Gamma_L^\lambda(t) = \diag(1,\lambda^{-1},\ldots,\lambda^{-(n+1)})
\Gamma_L(t) \diag(1,\lambda,\ldots,\lambda^{n+1}) \]
and the curve $\Gamma^\lambda(t) = z_0 \bQ(\Gamma_L^\lambda(t))$.
Notice that $\Gamma^\lambda: [-\epsilon,\epsilon] \to  \cU_{z_0}$
is locally convex and satisfies $\sing(\Gamma^\lambda) = \sing(\Gamma)$
and $\iti(\Gamma^\lambda) = \iti(\Gamma)$.
Given $t_0 \in [-\epsilon,\epsilon]$ we have
$\lim_{\lambda \to +\infty} \Gamma^\lambda(t_0) = z_0$;
by compactness, there exists $\lambda_0$ such that
$\Gamma^{\lambda_0}[[-\epsilon,\epsilon]] \subset A$.
The curve $\tilde\Gamma = \phi\circ\Gamma^{\lambda_0}$ therefore
admits triangular coordinates
$\tilde\Gamma_L: [-\epsilon,\epsilon] \to \Lo_{n+1}^{1}$,
$\tilde\Gamma(t) = q \bQ[\tilde\Gamma_L(t)]$.
We have $\tilde\Gamma(t_1) \in \phi[A] \cap \Bru_\sigma$
and therefore $\tilde\Gamma_L(t_1) \in \Pos_\sigma$.
From Lemma \ref{lemma:transition},
$t > t_1$ implies $\tilde\Gamma_L(t) \in \Pos_\eta$,
i.e., $\sing(\tilde\Gamma) \cap (t_1,\epsilon] = \emptyset$.
Thus $t_1$ is the last element of $\sing(\Gamma)$.

A similar argument using the sets $\Neg_{\ast}$ instead of $\Pos_{\ast}$
proves that $t_1$ is also the first element of $\sing(\Gamma)$.
\end{proof}



\begin{proof}[Proof of Lemma \ref{lemma:submanifold}]
For $w = \sigma_1\cdots\sigma_\ell = (\sigma_1, \ldots, \sigma_\ell)$
and $2j \in \ZZ \cap [0,2\ell+2]$,
set $B(w,j) \in \tilde B_{n+1}^{+}$ 
as in Equation \ref{eq:Bwj} above;
in particular, $B(w,\ell+1) = q_w = \acute\eta\hat w\acute\eta \in \Quat_{n+1}$.
We first define an open subset $\cA_w^\sharp \subset \cL_n(q_w) \times (0,1)$.
A pair $(\Gamma,\tilde\epsilon)$ belongs to $\cA_w^\sharp$ if there
exist $0 = \tilde t_0 < \tilde t_1 < \cdots
< \tilde t_\ell < \tilde t_{\ell+1} = 1$ such that:
\begin{enumerate}
\item{For each $i$, $\tilde t_{i+1} - \tilde t_i < 8\tilde\epsilon$.}
\item\label{item:Bwi}{Each arc $\Gamma|_{[\tilde t_{i}-2\tilde\epsilon,
\tilde t_{i}+2\tilde\epsilon]}$ is convex,
with image in $\cU_{B(w,i)}$.
In particular, for $\tilde z_i = \Gamma(\tilde t_i)$ we have
$\tilde z_i \in \cU_{B(w,i)}$. }
\item{Each arc $\Gamma|_{[\tilde t_{i}+\tilde\epsilon,
\tilde t_{i+1}-\tilde\epsilon]}$ is convex,
with image in $\cU_{B(w,i+\frac12)} = \Bru_{B(w,i+\frac12)}$.}
\item{For $f_{i,k_i}: \cU_{B(w,i)} \to \RR$ and $k_i$
as in Lemma \ref{lemma:pathcoordinates},
we have $f_{i,k_i}(\tilde z_i) = 0$.}
\item{Let $\Pi_{B(w,i)}: \cU_{B(w,i)} \to \Bru_{B(w,i)} \subset \cU_{B(w,i)}$
be the smooth projection defined in Remark \ref{rem:explicitpathcoordinates}.
Set $\check z_i = \Pi_{B(w,i)}(\tilde z_i)$.
There exist convex arcs in $\cU_{B(w,i)}$
from $\Gamma(\tilde t_{i}-\frac{\tilde\epsilon}{2})$ to $\check z_i$
and from $\check z_i$ to $\Gamma(\tilde t_{i}+\frac{\tilde\epsilon}{2})$.}
\end{enumerate}
For $\Gamma \in  \cL_n(q_w)$, set
\( J_\Gamma = \{ \tilde\epsilon \in (0,1) \;|\; 
(\Gamma,\tilde\epsilon) \in \cA_w^\sharp \} \);
clearly, $J_\Gamma$ is either an open interval or empty;
for $\Gamma \in \cL_n[w]$,
$J_\Gamma$ is an interval of the form $J_\Gamma = (0,\epsilon)$
for some $\epsilon > 0$.
Set 
\[ \cA_w = \{ \Gamma \in \cL_n(q_w) \;|\; J_\Gamma \ne \emptyset \}
\subseteq \cL_n(q_w), \]
an open subset.
For $\Gamma \in \cA_w$ the times $\tilde t_i$ are well defined
and, from Lemma \ref{lemma:pathcoordinates}, the functions 
$\Gamma \mapsto \tilde t_i$ are also continuous.
For $\Gamma \in \cA_w$, set $\epsilon = \epsilon_\Gamma = \sup J_\Gamma$;
from the continuity of $\tilde t_i$ and 
several uses of Lemma \ref{lemma:positivespeed},
the function $\Gamma \mapsto \epsilon$ is also continuous. 
For instance, checking that the pair $(\Gamma, \tilde\epsilon)$ 
satisfies item \ref{item:Bwi} above involves looking for the smallest 
$t>\tilde t_i$ such that, for $\Gamma_L$ defined by 
$\Gamma(t)=B(w,i+\frac12)\bQ(\Gamma_L(t))$, we have 
$\Gamma_L(t)\in\partial\Neg_\eta$: 
Lemmas \ref{lemma:positivespeed} and \ref{lemma:transition} 
imply that this $t$ is a continuous function of $\Gamma$.
We have $\tilde z_i = \Gamma(\tilde t_i) \in \cU_{B(w,i)}$; 
we define $\tilde z_i^{-} = \Gamma(\tilde t_{i}-\frac{\epsilon}{2})$,
$\tilde z_i^{+} = \Gamma(\tilde t_{i}+\frac{\epsilon}{2})$,
and $\check z_i = \Pi_{B(w,i)}(\tilde z_i)$.  
Also, the map 
$\Gamma \mapsto (\tilde z_i, \tilde z_i^{-}, \tilde z_i^{+}, \check z_i)$
is continuous.
For $\Gamma \in \cA_w$, $\epsilon = \epsilon_\Gamma$,
$\tilde t_i$, $\tilde z_i$ and $\check z_i$
as above we therefore have the following properties:
\begin{enumerate}
\item{For each $i$, $\tilde t_{i+1} - \tilde t_i \le 8\epsilon$.}
\item{Each arc $\Gamma|_{[\tilde t_{i}-\epsilon,
\tilde t_{i}+\epsilon]}$ is convex,
with image in $\cU_{B(w,i)}$; also, $\tilde z_i \in \cU_{B(w,i)}$. }
\item{Each arc $\Gamma|_{[\tilde t_{i}+\epsilon,
\tilde t_{i+1}-\epsilon]}$ is convex,
with image in $\cU_{B(w,i+\frac12)}$.}
\item{For $f_{i,k_i}: \cU_{B(w,i)} \to \RR$,
we have $f_{i,k_i}(\tilde z_i) = 0$.}
\item{There exist convex arcs in $\cU_{B(w,i)}$
from $\tilde z_i^{-}$ to $\check z_i$
and from $\check z_i$ to $\tilde z_{i}^{+}$.}
\end{enumerate}

Define $F_i: \cA_w \to \RR^{d_i}$, $d_i = k_i - 1$,
$F_i(\Gamma) = 
(f_{i,1}(\tilde z_i), \ldots, f_{i,d_i}(\tilde z_i))$
and
$F: \cA_w \to \RR^d$, $F(\Gamma) 
= (F_1(\Gamma), \ldots, F_\ell(\Gamma))$, 
$d=\dim(w)=d_1+\cdots+d_\ell$ 
(the functions $f_{i,\ast}$ are constructed in Lemma \ref{lemma:pathcoordinates}).
We claim that, for $\Gamma \in \cA_w$,
$F(\Gamma) = 0$ if and only if $\Gamma \in \cL_n[w]$.

Indeed, if $\Gamma \in \cA_w$ and $F(\Gamma) = 0$ we have
$\tilde z_i = \Gamma(\tilde t_i) \in \Bru_{\eta\sigma_i}$.
We already know that 
$\{\tilde t_1 < \cdots < \tilde t_\ell \} \subseteq
\sing(\Gamma) \subset \bigcup_i (\tilde t_i-\epsilon,\tilde t_i+\epsilon)$.
By Lemma \ref{lemma:singleletter} we have
$\sing(\Gamma) = \{ \tilde t_1 < \cdots <  \tilde t_{\ell} \}$
and therefore $\iti(\Gamma) = w$.

We now construct a projection
$\Pi: \cA_w \to \cL_n[w] \subset \cA_w$.
Given $\Gamma \in \cA_w$, the curve $\check\Gamma = \Pi(\Gamma)$,
$\check\Gamma: [0,1] \to \Spin_{n+1}$,
will coincide with $\Gamma$ except in the intervals
$[\tilde t_i -\frac{\epsilon}{2}, \tilde t_i + \frac{\epsilon}{2}]$
and will satisfy $\check\Gamma(\tilde t_i) = \check z_i$.
The restrictions
$\check\Gamma|_{[\tilde t_i - \frac{\epsilon}{2}, \tilde t_i]}$ and
$\check\Gamma|_{[\tilde t_i, \tilde t_i+\frac{\epsilon}{2}]}$ 
will be convex arcs contained in $\cU_{B(w,i)}$
joining $\tilde z_i^{-}$ to $\check z_i$
and $\check z_i$ to $\tilde z_i^{+}$, respectively.
These two convex arcs are obtained from the convex arcs
$\Gamma|_{[\tilde t_i - \frac{\epsilon}{2}, \tilde t_i]}$ and
$\Gamma|_{[\tilde t_i, \tilde t_i+\frac{\epsilon}{2}]}$ 
by projective transformations as in Remark \ref{rem:projtrans}.
Notice that $\epsilon_{\check\Gamma} = \epsilon_{\Gamma}$.
If $\Gamma \in \cL_n[w]$ we have $\check z_i = \tilde z_i$
and therefore $\check\Gamma = \Gamma$.

We now have a continuous map
$(\Pi,F): \cA_w \to \cL_n[w] \times \RR^d$;
let $\cB_w \subseteq \cL_n[w] \times \RR^d$ be its image.
We construct the inverse map $\Phi: \cB_w \to \cA_w$;
in the process we see that the set
$\cB_w$ is an open neighborhood of $\cL_n[w] \times \{0\}$.
Indeed, given $\check\Gamma \in \cL_n[w]$
construct $\epsilon$, $\tilde t_i$,
$\check z_i = \check\Gamma(\tilde t_i)$
and $\tilde z_i^{\pm} = \check\Gamma(\tilde t_i \pm \frac{\epsilon}{2})$
as above.
Given $\bfx = (\bfx_1,\ldots,\bfx_\ell) \in \RR^{d_1+\cdots+d_\ell}$,
there exist unique $\tilde z_i \in \cU_{B(w,i)}$ 
with $\Pi_{B(w,i)}(\tilde z_i) = \check z_i$
and $f_i(\tilde z_i) = \bfx_i$.
If there exist convex arcs contained in $\cU_{B(w,i)}$
from $\tilde z_i^{-}$ to $\tilde z_i$ and
from $\tilde z_i$ to $\tilde z_i^{+}$ then
$(\check\Gamma,\bfx) \in \cB_w$ and the curve
$\tilde\Gamma = \Phi(\check\Gamma,\bfx)$ is constructed as before.
More precisely,
$\tilde\Gamma$ coincides with $\check\Gamma$ except
in the intervals
$[\tilde t_i-\frac{\epsilon}{2},\tilde t_i+\frac{\epsilon}{2}]$.
The convex arcs 
$\tilde\Gamma|_{[\tilde t_i - \frac{\epsilon}{2}, \tilde t_i]}$ and
$\tilde\Gamma|_{[\tilde t_i, \tilde t_i+\frac{\epsilon}{2}]}$ 
are obtained from the arcs
$\check\Gamma|_{[\tilde t_i - \frac{\epsilon}{2}, \tilde t_i]}$ and
$\check\Gamma|_{[\tilde t_i, \tilde t_i+\frac{\epsilon}{2}]}$ 
by projective transformations.

Recall that there exists a natural diffeomorphism
from the product of triangular groups
$\Up_{\eta\sigma_i} \times \Lo_{\sigma_i^{-1}}$
to $\cU_{B(w,i)}$,
taking $(U_1,L_2)$ to $\bQ(U_1 B(w,i) L_2)$
(see the proof of Lemma \ref{lemma:pathcoordinates}).
Endow $\Up_{\eta\sigma_i}$ and $\Lo_{\sigma_i^{-1}}$
with the euclidean metrics coming from the usual sets of coordinates
(i.e., the entries)
and use the above diffeomorphism to endow  $\cU_{B(w,i)}$
with a flat euclidean metric.
Similarly, endow the cartesian product
$\cU^3_{B(w,i)} =  \cU_{B(w,i)} \times \cU_{B(w,i)} \times \cU_{B(w,i)}$
with a flat euclidean metric.
Let $\cW_i \subset \cU^3_{B(w,i)}$
be the open set of triples $(z_i^{-},z_i,z_i^{+})$ such that
there exist convex arcs contained in $\cU_{B(w,i)}$ 
from $z_i^{-}$ to $z_i$ and from $z_i$ to $z_i^{+}$.
Let $\delta_i: \cW_i \to (0,+\infty)$ be the continuous function
taking a triple  $(z_i^{-},z_i,z_i^{+}) \in \cW_i$
to one half of the distance
(in the flat euclidean metric constructed above)
from the complement $\cU^3_{B(w,i)} \smallsetminus \cW_i$, i.e.,
$\delta_i(z_i^{-},z_i,z_i^{+}) =
\frac12 d((z_i^{-},z_i,z_i^{+}), \cU^3_{B(w,i)} \smallsetminus \cW_i)$.
Given $\Gamma \in \cL_n[w]$,
define $\delta(\Gamma) = \min_i \delta_i(z_i^{-},z_i,z_i^{+})$
where, as above, $z_i^{\pm} = \Gamma(t_i\pm \frac{\epsilon}{2})$.
Notice that $\delta: \cL_n[w] \to (0,+\infty)$ is continuous (see for instance the proof of Lemma \ref{lemma:distance})
and that if $|\bfx| \le \delta(\Gamma)$ then
$(\Gamma,\bfx) \in \cB_w$ (by construction).

Let $\BB^d \subset \DD^d \subset \RR^d$ be the open and closed balls
of radius $1$, respectively.
Define $\hat\Phi: \cL_n[w] \times \DD^d \to \cA_w$ by
$\hat\Phi(\Gamma,\bfx) = \Phi(\Gamma,\delta(\Gamma) \bfx)$.
Let $\hat\cA_w = \hat\Phi[\cL_n[w] \times \BB^d] \subset \cA_w$
and $\hat F: \hat\cA_w \to \BB^d$
so that $(\Pi,\hat F) = \hat\Phi^{-1}: 
\hat\cA_w \to \cL_n[w] \times \BB^d$.
This completes the construction of
the tubular neighborhood of $\cL_n[w]$.
\end{proof}

\begin{rem}
\label{rem:Ftransversal}
The maps $F_w = F: \cA_w \to \RR^d$ and 
$\hat F_w = \hat F: \hat{\cA}_w \to \BB^d$ 
constructed in the proof of Lemma \ref{lemma:submanifold}
will be used again in the future.
More precisely, in Section \ref{sect:transversal} below
we follow the constructions of $F$ above  
and of the map $f$ in Remark \ref{rem:explicitpathcoordinates} 
in order to describe examples of strata $\cL_n[w]$.
\end{rem}

\begin{rem}
\label{rem:cLnwC1}
In the proof of Lemma \ref{lemma:submanifold} above,
the auxiliary times $\tilde t_i=\tilde t_{i,\Gamma}$ are continuous functions of $\Gamma$
but it is easy to construct finite dimensional smooth families
of curves $\Gamma_s$ (with each curve not of class $C^1$)
such that $\tilde t_{i,\Gamma_s}$ is a continuous function of the variable $s$,
but is not a function of class $C^1$ of $s$.
Similarly, the subset $\cL_n[w] \subset \cL_n$ is not
a submanifold of class $C^1$.

A more restrictive topology on $\cL_n$ for which
all curves are necessarily of class $C^k$ (for some $k \ge 1$)
solves this difficulty but creates others.
More precisely, the subset $\cL_n[w] \subset \cL_n$ is then
a submanifold of class $C^k$.
On the other hand,
several proofs as given above no longer apply:
some examples are the proof of Lemma \ref{lemma:path}
and the construction of $\Pi$ and $\Phi$ in the proof
of Lemma \ref{lemma:submanifold}.
These difficulties are surmountable,
but imply longer proofs.

Well known results from the homotopy theory of infinite-dimensional manifolds 
(see \cite{Burghelea-Henderson, Burghelea-Kuiper,
Burghelea-Saldanha-Tomei1, Henderson, Moulis})
combine to show that different versions
of the space $\cL_n$, with different manifold structures,
are pairwise homotopy equivalent 
and therefore homeomorphic.
This is discussed in Appendix \ref{appendix:Hilbert};
in particular we cite Theorem 2 of \cite{Burghelea-Saldanha-Tomei1}
(Fact \ref{fact:BST})
and Theorem 0.1 of \cite{Burghelea-Henderson}
(Fact \ref{fact:BH}).
\end{rem}

\section{Transversal sections}
\label{sect:transversal}

The proof that $\cL_n[w] \subset \cL_n$ is a topological submanifold
implicitly gives us (topologically) transversal sections.
We now construct an \emph{explicit} transversal section,
starting with the case $w = (\sigma)$;
in this case we write $\cL_n[\sigma] = \cL_n[w]$.
The construction roughly corresponds to going back
to Lemma \ref{lemma:submanifold},
then to Lemma \ref{lemma:pathcoordinates}
and Remark \ref{rem:explicitpathcoordinates},
then to Lemma \ref{lemma:positivespeed}
and Remark \ref{rem:explicitpositivespeed},
and following the steps.
A key difference is that strictly following  Lemma \ref{lemma:submanifold}
gives curves which fail to be smooth precisely at the times $\tilde t_i$;
the curves produced by our construction in this section
are smooth (indeed algebraic) in a neighborhood of $\tilde t_i$,
even though they still violate smoothness elsewhere.
We first present the construction as an algorithm,
then provide examples.

Consider $\sigma \in S_{n+1}$, $\sigma \ne e$, $\rho = \eta\sigma$ and
$d = \dim(\sigma) = \inv(\sigma) - 1$.
Consider $z_0 = q \acute\eta \acute\sigma \in \widetilde B_{n+1}^{+}$, 
$q \in \Quat_{n+1}$,
so that $\chop(z_0) = q \acute\eta$
and $\adv(z_0) = q \acute\eta \hat\sigma$.
Let $Q_0 = \Pi(z_0) \in B_{n+1}^{+}$.
We first construct an explicit transversal section
$\psi: \RR^{d+1} \to \SO_{n+1}$
to the Bruhat cell $\Bru_{Q_0} \subset \SO_{n+1}$
passing through $Q_0 = \psi(0)$
(compare with  
Remarks \ref{rem:explicitpositivespeed} and 
\ref{rem:explicitpathcoordinates}).
First we define a matrix
$\tilde M \in (\RR[x_1,\ldots,x_{d+1}])^{(n+1)\times (n+1)}$
where $x_l$, $1 \le l \le d+1$, are new variables.
For $i \in \nmaisum$,
set $(\tilde M)_{i,i^\rho} = (Q_0)_{i,i^\rho} = \pm 1$.
There are $d+1$ zero entries in $Q_0$ which are simultaneously
below a nonzero entry and to the left of a nonzero entry:
these are the pairs $(i,j)$ for which
$j < i^\rho$ and $j^{\rho^{-1}} < i$.
Number the positions from $1$ to $d+1$ in the same order
you would read or write them on a page
(top to bottom and left to right).
For each such position $(i,j)$,
set $(\tilde M)_{i,j} = (Q_0)_{i,i^\rho} x_l$.
The other entries of $\tilde M$ are set to $0$:
this defines the desired matrix
$\tilde M \in (\RR[x_1,\ldots,x_{d+1}])^{(n+1)\times (n+1)}$
or, equivalently, a function $\psi_L: \RR^{d+1} \to \GL^{+}_{n+1}$
where $\psi_L(\bfx)$ is obtained
by evaluating $\tilde M$ at $\bfx \in \RR^{d+1}$.
As an example, the two matrices below correspond to
$n = 2$ and $\sigma_0 =  [321]$ (so that $d = 2$) and
$n = 3$ and $\sigma_1 =  [3142]$ (so that $d = 2$):
\[ \tilde M_0 = \begin{pmatrix}
1 & 0 & 0 \\
x_1 & 1 & 0 \\
x_2 & x_3 & 1 \end{pmatrix}; \qquad
\tilde M_1 = \begin{pmatrix}
0 & -1 & 0 & 0 \\
0 & x_1 & 0 & 1 \\
1 & 0 & 0 & 0 \\
x_2 & x_3 & 1 & 0 \end{pmatrix}. \]
Notice that the map $\psi_L$ is a homeomorphism from $\RR^{d+1}$ to
$Q_0 \Lo_{\sigma^{-1}} \subset \GL^{+}_{n+1}$.
The function $\psi_A = \bQ \circ \psi_L: \RR^{d+1} \to \SO_{n+1}$
is the desired transversal section to the Bruhat cell $\Bru_{Q_0}$.
In order to define $\psi: \RR^{d+1} \to \Spin_{n+1}$,
$\psi_A = \Pi \circ \psi$,
lift the map $\psi_A$ starting at $\psi(0) = z_0$.

Consider $\RR^d \subset \RR^{d+1}$ defined by $x_{d+1} = 0$.
Let $\fn = \sum_i \fl_i$ be the lower triangular nilpotent matrix
whose only nonzero entries are $\fn_{j+1,j} = 1$
(see Example \ref{example:fh}).
For each $\bfx \in \RR^{d}$ define a curve
$\phi_L(\bfx;\cdot): \RR \to Q_0 \Lo_{n+1}^{1} \subset \GL^{+}_{n+1}$
by the IVP
\[ \frac{d}{dt}\phi_L(\bfx;t) = \phi_L(\bfx;t) \fn, \quad
\phi_L(\bfx;0) = \psi_L(\bfx), \]
so that $\phi_L(\bfx;t) = \psi_L(\bfx) \exp(t\fn)$.
Since entries of $\phi_L(\bfx;t)$ are polynomials in $\bfx$ and $t$,
we may equivalently consider the matrix
$M \in (\RR[\bfx;t])^{(n+1)\times (n+1)}$,
$M(\bfx,t) = \phi_L(\bfx;t)$,
whose entries are polynomials in $\bfx$ and $t$,
of degree at most $n$ in the variable $t$ and satisfying
\[ (M)_{i,j+1} = \frac{d}{dt} (M)_{i,j}. \]
As an example, the two matrices below again correspond to
$n = 2$,  $\sigma_0 =  [321]$  and $n = 3$,  $\sigma_1 =  [3142]$:
\begin{equation}
\label{eq:M0M1}
M_0 = \begin{pmatrix}
1 & 0 & 0 \\
t+x_1 & 1 & 0 \\
\frac{t^2}{2}+x_2 & t & 1
\end{pmatrix}; \qquad
M_1 = \begin{pmatrix}
-t & -1 & 0 & 0 \\
\frac{t^3}{6} + x_1 t & \frac{t^2}{2} + x_1 & t & 1 \\
1 & 0 & 0 & 0 \\
\frac{t^2}{2} + x_2 & t & 1 & 0
\end{pmatrix}.
\end{equation}
Notice that, given $\bfx \in \RR^d$,
the curve $Q_0^{-1} \phi_L(\bfx;\cdot): \RR \to \Lo_{n+1}^{1}$
is locally convex.
Let $\Gamma_{\bfx}: \RR \to \Spin_{n+1}$
be the locally convex curve defined by
$\Gamma_{\bfx}(t) = \bQ(\phi_L(\bfx,t))$, $\Gamma_{\bfx}(0) = \psi(\bfx)$.
Clearly, $\Gamma_0(0) = z_0$,
$\Gamma_0(t) \in \Bru_{\chop(z_0)}$ for $t < 0$ and
$\Gamma_0(t) \in \Bru_{\adv(z_0)}$ for $t > 0$.

We now construct the desired transversal surface
$\phi: \DD^d \to \cL_n$.
Choose $z_0$ above such that $\chop(z_0) = \acute\eta$
and $\adv(z_0) = \acute\eta\hat\sigma$;
let $q_1 = \acute\eta\hat\sigma\acute\eta \in \Quat_{n+1}$.
For sufficiently small $r \in (0,\frac{\pi}{4})$,
there exists a convex arc contained in $\Bru_{\chop(z_0)}$
going from $\exp(r \fh)$ to $\Gamma_0(-r)$.
Similarly,
for sufficiently small $r \in (0,\frac{\pi}{4})$,
there exists a convex arc contained in $\Bru_{\adv(z_0)}$
going from $\Gamma_0(r)$ to $q_1 \exp(-r \fh)$.
Fix such a small $r \in (0,\frac{\pi}{4})$.
By continuity, there exists a small $\tilde s > 0$ such that,
if $|\bfx| \le \tilde s$ then there exists a convex arc
contained in $\Bru_{\chop(z_0)}$
going from $\exp(r \fh)$ to $\Gamma_{\bfx}(-r)$.
Similarly, for sufficiently small $\tilde s > 0$
if $|\bfx| \le \tilde s$ then there exists a convex arc
in contained in $\Bru_{\adv(z_0)}$
going from $\Gamma_{\bfx}(r)$ to $q_1 \exp(-r \fh)$.
Fix such a small $\tilde s > 0$.
Use Lemma \ref{lemma:convex1}
to define such convex arcs
$\tilde \phi(\bfx)|_{[\frac18,\frac38]}$
going from $\exp(r \fh)$ to $\Gamma_{\tilde s\bfx}(-r)$ and
$\tilde \phi(\bfx)|_{[\frac58,\frac78]}$
going from $\Gamma_{\tilde s\bfx}(r)$ to $q_1 \exp(-r \fh)$.
For $t \in [0,\frac18]$, set $\tilde\phi(\bfx)(t) = \exp(8rt\fh)$;
for $t \in [\frac78,1]$, set $\tilde\phi(\bfx)(t) = q_1 \exp(8r(t-1)\fh)$;
for $t \in [\frac38,\frac58]$,
set $\tilde \phi(\bfx)(t) = \Gamma_{\tilde s\bfx}(8r(t-\frac12))$.
Consider now $s \in (0,\tilde s]$ sufficiently small so that,
for all $\bfx \in \DD^d$ with $|\bfx| \le \frac{s}{\tilde s}$,
we have $\tilde\phi(\bfx) \in \hat \cA_\sigma$
(where $\hat \cA_\sigma$ is the open neighborhood of $\cL_n[\sigma]$
constructed in Lemma \ref{lemma:submanifold}).
Define $\phi: \DD^d \to \hat\cA_\sigma \subset \cL_n$ by
$\phi(\bfx) = \tilde\phi(\frac{s}{\tilde s} \bfx)$.

\begin{lemma}
\label{lemma:transsection}
Consider $\sigma \in S_{n+1}$, $\sigma \ne e$, $\dim(\sigma) = d$
and construct the map $\phi: \DD^d \to \cL_n$ as above.
This map is topologically transversal to $\cL_n[\sigma]$,
with a unique intersection at $\bfx = 0 \in \DD^d$.
\end{lemma}

\begin{proof}
Uniqueness of intersection follows from Lemma \ref{lemma:pathcoordinates}.
Topological transversality follows from taking the composition
$\hat F \circ \phi$, where $\hat F: \hat \cA_\sigma \to \DD^d \subset \RR^d$
is constructed in Lemma \ref{lemma:submanifold}.
The map $\hat F \circ \phi: \DD^d \to \DD^d$
is a positive multiple of the identity.
\end{proof}

Notice that the maps $\hat F: \hat \cA_\sigma \to \RR^d$ and
$\phi: \DD^d \to \hat\cA_\sigma \subset \cL_n$ consistently provide us
with a transversal orientation to $\cL_n[\sigma]$.

This completes the process of contructing a transversal section
to $\cL_n(\sigma_1)$ at the path $(z_1)$,
$z_1 = q \acute\sigma_1$, $q \in \Quat_{n+1}$.
By applying affine transformations in the interval
and projective transformations in the group $\Spin_{n+1}$,
this defines a map $\phi_1$ from $\bfx \in \DD^{d_1}$
($d_1 = \dim(\sigma_1)$) to convex arcs
$\Gamma_{\bfx}: [t_1 - \epsilon, t_1 + \epsilon] \to \Spin_{n+1}$
with $\sing(\Gamma_{\bfx}) \ne \emptyset$,
$\sing(\Gamma_{\bfx}) \subset
(t_1 - \frac{\epsilon}{2}, t_1 + \frac{\epsilon}{2})$
and $\iti(\Gamma_{\bfx}) = (\sigma_1)$ if and only if $\bfx = 0$.
We may furthermore assume that 
$\Gamma_{\bfx}(t_1 \pm \epsilon) = z_1 \exp(\pm\epsilon\fh)$
for all $\bfx \in \DD^{d_1}$
and that
$\Gamma_{\bfx}(t) = z_1 \exp((t-t_1)\fh)$ for $\bfx = 0 \in \DD^{d_1}$.

More generally, for any $w = \sigma_1\cdots\sigma_k = 
(\sigma_1, \ldots, \sigma_k) \in \Word_n$,
for any path $(z_1, \ldots, z_k) \in \Pathiti(w)$
and for any set $\{t_1 < \cdots < t_k \} \subset (0,1)$
we show how to construct
a smooth map $\phi: \DD^d \to \cL_n$, $d = \dim(w)$,
transversal to $\cL_n[w]$ at $\phi(0) \in \cL_n[w]$,
$\pathiti(\phi(0)) = (z_1, \ldots, z_k)$,
$\sing(\phi(0)) = \{t_1 < \cdots < t_k \}$.
Make the convention $t_0 = 0$, $z_0 = 1$, $t_{k+1} = 1$
and $z_{k+1} = \acute\eta\hat\sigma_1 \cdots \hat\sigma_k\acute\eta$.
Define $q_i \in \Quat_{n+1}$ such that 
$q_i \acute\eta = \adv(z_{i-1}) = \chop(z_i)$.
First, choose $\epsilon > 0$ such that for all $i$, $0 < i \le k+1$,
$t_{i-1} + \epsilon < t_i - \epsilon$,
$z_{i-1} \exp(\epsilon\fh) \in \Bru_{q_i \acute\eta}$ and
$z_{i} \exp(-\epsilon\fh) \in \Bru_{q_i \acute\eta}$.
Define $L_{i,-}, L_{i,+} \in \Lo_{n+1}^1$ by
$z_{i-1} \exp(\epsilon\fh) = q_i \acute\eta \bQ(L_{i,-})$
$z_{i} \exp(-\epsilon\fh) = q_i \acute\eta \bQ(L_{i,+})$:
by taking $\epsilon$ sufficiently small we may assume that
$L_{i,-} \ll L_{i,+}$.
Choose fixed convex arcs
$\Gamma_{i-\frac12}:
[t_{i-1}+\epsilon, t_i-\epsilon] \to \Bru_{q_i\acute\eta}$,
$\Gamma_{i-\frac12}(t_{i-1}+\epsilon) = z_{i-1} \exp(\epsilon\fh)$,
$\Gamma_{i-\frac12}(t_{i}-\epsilon) = z_{i} \exp(-\epsilon\fh)$.
In each interval $[t_i - \epsilon, t_i + \epsilon]$,
define as above a map $\phi_i$ associating to each $\bfx_i \in \DD^{d_i}$
a convex arc
$\Gamma_{i,\bfx_i}: [t_i - \epsilon, t_i + \epsilon] \to \Spin_{n+1}$
with $\Gamma_{i,\bfx_i}(t_i-\epsilon) = z_i \exp(-\epsilon\fh)$,
$\Gamma_{i,\bfx_i}(t_i+\epsilon) = z_i \exp(\epsilon\fh)$.
Set $\Gamma_0: [0,\epsilon] \to \Spin_{n+1}$,
$\Gamma_0(t) = \exp(t \fh)$ and
$\Gamma_{k+1}: [1-\epsilon,1] \to \Spin_{n+1}$,
$\Gamma_{k+1}(t) = z_{k+1} \exp((t-1)\fh)$.
Finally, for $\bfx = (\bfx_1, \ldots, \bfx_k)$,
concatenate these arcs to define $\phi(\bfx) = \Gamma_{\bfx} \in \cL_n$:
the map $\phi$ is the desired transversal section.

In the next section we continue the description of these examples.
We first discuss the concept of multiplicity.

\section{Multiplicities}
\label{sect:multitineraries}

Recall that the multiplicity of $\sigma \in S_{n+1}$ is the vector
\[ \mult(\sigma) = (\mult_1(\sigma), \cdots, \mult_n(\sigma)),
\quad
\mult_j(\sigma) = (1^\sigma + \cdots + j^\sigma) - (1+\cdots + j). \]
Consider a matrix $Q \in \SO_{n+1}$ and $j \le n$.
Let $\swminor(Q,j) \in \RR^{j \times j}$ be the southwest $j\times j$ minor.
For a locally convex curve $\Gamma \in \cL_n$,
let $m_j = m_{\Gamma;j}: [0,1] \to \RR$, $1 \le j \le n$,
be the function defined by
the determinant of the southwest $k\times k$ minor of $\Gamma$:
\begin{equation}
\label{eq:mj}
m_{\Gamma;j}(t) = \det(\swminor(\Gamma(t),j)).
\end{equation}
Write $\mult_j(\Gamma;t_\bullet) = \mu$
if $t_\bullet$ is a zero of multiplicity $\mu$
of the function $m_{\Gamma;j}$, that is, if
$(t-t_\bullet)^{(-\mu)} m_{\Gamma;j}(t)$
is continuous and non-zero at $t = t_\bullet$.
Notice that for a general locally convex curve $\Gamma$,
$\mult_j(\Gamma;t_\bullet)$ as above is not always well defined;
if $\Gamma$ is smooth near $t_\bullet$, however,
$\mult_j(\Gamma;t_\bullet) = \mu \in \NN$ is well defined.
Let the \emph{multiplicity vector} be
\[ \mult(\Gamma;t_\bullet) = \left(
\mult_1(\Gamma;t_\bullet), \mult_2(\Gamma;t_\bullet), \ldots,
\mult_n(\Gamma;t_\bullet)
\right). \]
Recall that $\Gamma(t_\bullet) \in \Bru_\eta$ if and only if
there exist upper triangular matrices $U_1$ and $U_2$ such that
$\Gamma(t_\bullet) = U_1 \acute\eta U_2$.
It is a basic fact of linear algebra that this happens if and only if
$m_{\Gamma;j}(t_\bullet) \ne 0$ for all $j$.
In other words, for $t_\bullet \in (0,1)$,
$\mult(\Gamma;t_\bullet) = 0$ if and only if $t_\bullet \notin \sing(\Gamma)$;
roots of the functions $m_j$ indicate times $t$
for which $\Gamma(t) \notin \Bru_{\eta}$.
The next lemma generalizes this remark
and shows how to define the multiplicity vector in general.
It also justifies the notation $\mult(\sigma)$.

\begin{lemma}
\label{lemma:deg}
For a locally convex curve $\Gamma: J \to \Spin_{n+1}$
and $t_\bullet \in J$, if $\Gamma$ is smooth in an open interval
containing $t_\bullet$ and $\Gamma(t_\bullet) \in \Bru_{\eta\sigma}$
then
\[ \mult(\Gamma;t_\bullet) = \mult(\sigma). \]
\end{lemma}

We use this formula to define $\mult(\Gamma;t_\bullet)$
even if $\Gamma$ is not smooth.
First, however, an easy result in linear algebra.

\begin{lemma}
\label{lemma:vandert}
Let $k_1, k_2, \cdots, k_n$ be non-negative integers.
Let $M$ be the $n\times n$ matrix with entries
\[ M_{i,1} = t^{k_i}, \qquad
M_{i,j+1} = \frac{d}{dt} M_{i,j}. \]
Then
\[ \det(M) = Ct^\mu; \qquad
C = \prod_{i_0 < i_1} (k_{i_1} - k_{i_0}); \qquad 
\mu = -\frac{n(n-1)}{2} + \sum_i k_i. \]
If $\tilde M$ is obtained from $M$
by substituting $1$ for $t$ then $\det(\tilde M) = C \ne 0$.
\end{lemma}

\begin{proof}
We have $M_{i,j} = \tilde M_{i,j} t^{(k_i + 1 - j)}$.
All monomials in the expansion of $\det(M)$ have therefore degree $\mu$.
The first column of $\tilde M$ consists of ones;
the second column has $i$-th entry equal to $k_i$.
The third column has $i$-th entry equal to $k_i(k_i-1) = k_i^2 - k_i$:
an operation on colums leaves the determinant unchanged
but now makes the third column have entries $k_i^2$.
Perform similar operations on columns to obtain a Vandermonde matrix,
implying $\det(\tilde M) = C$, as desired.
\end{proof}

\begin{proof}[Proof of Lemma \ref{lemma:deg}]
Assume without loss of generality that $t_0 = 0$, $J = (-\epsilon,\epsilon)$.
Notice that projective transformations
have the effect of multiplying the functions $m_j$
by a positive multiple and therefore do not affect 
the multiplicity vector.
We may therefore assume that
$\Gamma(0) = Q_0 \in B_{n+1}^{+} \cap \Bru_{\eta\sigma_0}$.
We thus have $(Q_0)_{i,i^{\eta\sigma_0}} = \varepsilon_i \in \{ \pm 1\}$
and $(Q_0)_{i,j} = 0$ otherwise.
As in Section \ref{sect:transversal},
we use generalized triangular coordinates:
$\Gamma_L: (-\epsilon,\epsilon) \to Q_0 \Lo_{n+1}^{1}$,
$\Gamma_L(t) = Q_0 \bL(Q_0^{-1} \Gamma(t))$,
$\Gamma(t) = \bQ(\Gamma_L(t))$.
Notice that $\det(\swminor(\Gamma_L(t),k))$ is a positive multiple
of $\det(\swminor(\Gamma(t),k))$, so that we may work with $\Gamma_L$.
Let $\Lambda_0 = (\Gamma_L(0))^{-1} \Gamma_L'(0) = \sum_i c_i \fl_i$,
$c_i > 0$;
let $k_i = \prod_{j < i} c_j$.

For given $i_0 \in \nmaisum$,
set $j_0 = (n+2-i_0)^{\sigma_0} = i_0^{\eta\sigma_0}$.
For $j > j_0$ we have $(\Gamma_L(t))_{i_0,j} = 0$;
also, $(\Gamma_L(t))_{i_0,j_0} = \varepsilon_{i_0} = \pm 1$.
For $j = j_0 - 1$, we have that
the derivative of the function $(\Gamma_L(t))_{i_0,j}$
is a smooth positive multiple of $(\Gamma_L(t))_{i_0,j+1}$;
we thus have $(\Gamma_L(t))_{i_0,j} = t\;c_j \varepsilon_{i_0} u_{i_0,j}(t)$
where $u_{i_0,j}$ is smooth and $u_{i_0,j}(0) = 1$.
Similarly, for $j = j_0 - \mu$, $\mu \ge 0$,
we have
$$ 
(\Gamma_L(t))_{i_0,j} = \frac{t^\mu}{\mu!}\;\frac{k_{j_0}}{k_j}\;
\varepsilon_{i_0} u_{i_0,j}(t), \qquad
u_{i_0,j}(0) = 1 $$
or, equivalently,
$$ 
(\Gamma_L(t))_{i,j} =
\frac{1}{(i^{\eta\sigma_0} - j)!} \;
\frac{\varepsilon_i k_{i^{\eta\sigma_0}}
t^{i^{\eta\sigma_0}}}{k_{j} t^{j}} \;
u_{i,j}(t), $$
where we follow the convention that $\frac{1}{\mu!} = 0$ for $\mu < 0$.

Consider now $\det(\swminor(\Gamma_L(t),k))$ as a function of $t$.
Write the entries as above.
The powers of $t$ can be taken out of the determinant,
yielding a factor $t^{\mult_k(\sigma_0)}$.
The terms $\varepsilon_\ast$ and $k_\ast$ can be taken out,
giving us a nonzero constant multiplicative factor.
Multiply the $i$-th row by $(i^{\eta\sigma_0} - 1)! \ne 0$:
the remaining matrix $M(t)$ has entries 
\[ M_{i,j}(t) =
\frac{(i^{\eta\sigma_0}-1)!}{(i^{\eta\sigma_0} - j)!} \;
u_{i,j}(t). \]
The matrix $\swminor(M(0),k)$ is just like the matrix $\tilde M$
in Lemma \ref{lemma:vandert},
and therefore, $\det(\swminor(M(0),k)) \ne 0$.
By continuity, $\det(\swminor(M(t),k))$ is nonzero near $t = 0$.
\end{proof}

For $w = \sigma_1\cdots\sigma_k = (\sigma_1,\ldots,\sigma_k) \in \Word_n$,
set its \emph{multiplicity} to be
\[ \mult(w) = \sum_j \mult(\sigma_j) \in \NN^n. \]





\begin{example}
\label{example:transversalsection}
We shall write $a = a_1$, $b = a_2$, $c = a_3$, \dots
As in Section \ref{sect:transversal},
let $\Gamma_{\bfx}$ be a family of convex arcs
and let $M$ be a matrix with polynomial entries.
Let $m_j(t)$ be the determinant of the southwest
$j\times j$ minor of $M$,
so that $m_j$ is an explicit polynomial in rational coefficients
in $t$ and $\bfx$ (or $x_i$ for $1 \le i \le d$).

In our first example ($n = 2$, $\sigma = [321] = [aba]$;
see the matrices in Equation \ref{eq:M0M1}), we have
\[ m_1(t) = \frac{t^2}{2} + x_2, \quad m_2(t) = \frac{t^2}{2} + x_1t - x_2. \]
Thus, $m_1$ has two real roots $t = \pm\sqrt{-2x_2}$ if $x_2 < 0$ and
$m_2$ has two real roots $t = -x_1\pm\sqrt{x_1^2+2x_2}$ if
$x_2 > -\frac{x_1^2}{2}$.
Thus, if $x_2 > 0$ the itinerary is $bb = (a_2,a_2)$ and
if $x_2 <  -\frac{x_1^2}{2}$ the itinerary is $aa = (a_1,a_1)$.
If $x_1 < 0$ (resp. $x_1 > 0$) and $ -\frac{x_1^2}{2} < x_2 < 0$ 
the itinerary is $abab = (a_1,a_2,a_1,a_2)$ (resp. $baba = (a_2,a_1,a_2,a_1)$);
the reader should compare these results
with Figures \ref{fig:aba} and \ref{fig:ababcb}.




\begin{figure}[ht]
\def\svgwidth{65mm}
\centerline{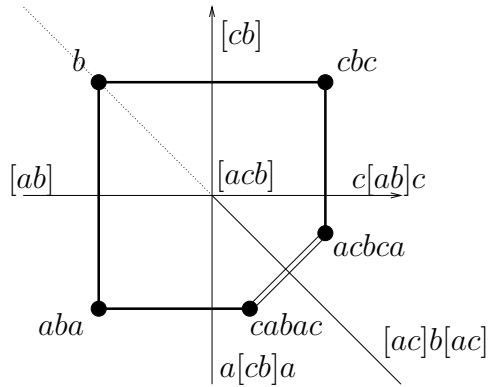}
\caption{A transversal section to $\cL_3[[acb]]$.}
\label{fig:planeacb}
\end{figure}

In our second example ($n = 3$, $\sigma = [3142] = [acb]$),
we have
\[ m_1(t) = \frac{t^2}{2} + x_2, \quad
m_2(t) = t, \quad
m_3(t) = \frac{t^2}{2} - x_1. \]
Thus, $m_2$ has a simple root at $t = 0$.
If $x_2 > 0$, $m_1$ has no real roots;
if $x_2 < 0$, $m_1$ has roots $t = \pm\sqrt{-x_2}$.
Similarly, for $x_1 < 0$, $m_3$ has no real roots
and for $x_1 > 0$, $m_3$ has roots $t = \pm\sqrt{x_1}$.
It is now easy to verify the itineraries in Figure \ref{fig:planeacb}.
Notice that this section is transversal to $\cL_3[[3142]] = \cL_3[[acb]]$
(as promised)
but is not transversal to $\cL_3[([a_1a_3],a_2,[a_1a_3])] = \cL_3[[ac]b[ac]]$
(which was never promised).
\end{example}


\section{The poset $\Word_n$}
\label{sect:poset}

Let $\sigma_0 \in S_{n+1}$, $\sigma_0 \ne e$;
let $\sigma_1 = \eta\sigma_0$.
Let $z_1 =q \acute\sigma_1 \in \widetilde B_{n+1}^{+}$, $q \in \Quat_{n+1}$.
Recall from Example \ref{example:fh}
that $\bL(\exp(\theta\fh)) = \exp(\tan(\theta) \fh_L)$
for $\theta \in (-\frac{\pi}{2}, \frac{\pi}{2})$.
Let $\theta_0 \in (0,\frac{\pi}{2})$.
The curve $\Gamma_0 \in \cL_n(z_1\exp(-\theta_0\fh); z_1\exp(\theta_0\fh))$,
$\Gamma_0: [-\theta_0,+\theta_0] \to \Spin_{n+1}$,
$\Gamma_0(\theta) = z_1 \exp(\theta\fh)$,
is locally convex with image contained in $\cU_{z_1}$.
Moreover, it can be expressed in triangular coordinates
and is therefore convex
(see Appendix \ref{appendix:convex}):
$\Gamma_0(\theta) = z_1\bQ(\Gamma_{0;L}(\theta))$
where $\Gamma_{0;L}(\theta): [-\theta_0,\theta_0] \to \Lo_{n+1}^{1}$,
$\Gamma_{0;L}(\theta) = \exp(\tan(\theta) \fh_L)$.
The itinerary of $\Gamma_0$ is $\iti(\Gamma_0) = (\sigma_0)$.

For $w \in \Word_n$,
define $w \tok \sigma_0$ (or, more precisely, $w \tok (\sigma_0)$)
if there exists a convex curve
$\Gamma_1 \in \cL_{n,\conv}(z_1\exp(-\theta_0\fh); z_1\exp(\theta_0\fh))$
with $\iti(\Gamma_1) = w$.
Our first remark is that this condition does not depend
on the choice of $\theta_0 \in (0,\frac{\pi}{2})$.

\begin{lemma}
\label{lemma:tok0}
Consider $z_1 \in \widetilde B_{n+1}^{+}$,
$\theta_1, \theta_2 \in (0,\frac{\pi}{2})$.
Consider $w \in \Word_n$.
Then there exists 
$\Gamma_1 \in \cL_{n,\conv}(z_1\exp(-\theta_1\fh); z_1\exp(\theta_1\fh))$
with $\iti(\Gamma_1) = w$
if and only if there exists
$\Gamma_2 \in \cL_{n,\conv}(z_1\exp(-\theta_2\fh); z_1\exp(\theta_2\fh))$
with $\iti(\Gamma_2) = w$.

Furthermore, if there exists
$\Gamma_1 \in
\cL_{n,\conv}(z_1\exp(-\theta_1\fh); z_1\exp(\theta_1\fh))$
such that $\iti(\Gamma_1) = w$
then there exists a homotopy 
\begin{gather*}
H: [0,1] \to 
\cL_{n,\conv}(z_1\exp(-\theta_1\fh); z_1\exp(\theta_1\fh)), \quad 
H(0) = \Gamma_0, \quad H(1) = \Gamma_1, \\
\Gamma_0(\theta) = z_1 \exp(\theta\fh), \quad
H(s)|_{[-\theta_1,-s\theta_1] \sqcup [s\theta_1,\theta_1]} =
\Gamma_0|_{[-\theta_1,-s\theta_1] \sqcup [s\theta_1,\theta_1]}
\end{gather*}
such that $\iti(H(s)) = w$ for all $s \in (0,1]$.
\end{lemma}

\begin{proof}
We start with the first claim.
Assume without loss of generality that $\theta_1 < \theta_2$.
Given $\Gamma_1$ as above, $\Gamma_2$ can be constructed by attaching arcs:
set 
\[ \Gamma_2(\theta) = \begin{cases}
\Gamma_1(\theta), & \theta \in [-\theta_1,\theta_1], \\
z_1 \exp(\theta\fh), &
\theta \in [-\theta_2,-\theta_1] \sqcup [\theta_1,\theta_2].
\end{cases} \]
Conversely, given $\Gamma_2$ we apply a projective transformation
to obtain $\Gamma_1$.
More precisely, set
$\Gamma_{2;L}: [-\theta_2,\theta_2] \to \Lo_{n+1}^{1}$,
$\Gamma_{2;L}(\theta) = \bL(z_1^{-1} \Gamma_2(\theta))$.
Notice that a diagonal projective transformation
takes $\exp(\pm\tan(\theta_2) \fh_L)$
to $\exp(\pm\tan(\theta_1) \fh_L)$:
\[ \exp(\pm\tan(\theta_1) \fh_L) =
\diag\left(1,\lambda,\ldots,\lambda^n\right)
\exp(\pm\tan(\theta_2) \fh_L) 
\diag\left(1,\lambda^{-1},\ldots,\lambda^{-n}\right) \]
for $\lambda = \tan(\theta_1)/\tan(\theta_2)$;
apply this projective transformation and reparametrize the domain
to obtain $\Gamma_{1;L}$ and therefore 
$\Gamma_1 \in \cL_{n,\conv}(z_1\exp(-\theta_1\fh); z_1\exp(\theta_1\fh))$
with $\iti(\Gamma_1) = \iti(\Gamma_2)$.

For the second claim, given $\Gamma_1$,
apply a projective transformation as above to define $H(s)$
satisfying the conditions in the statement
(compare with the construction of the homotopy
in the proof of Lemma \ref{lemma:convex2}).
\end{proof}

\begin{lemma}
\label{lemma:tok1}
Consider $w \in \Word_n$, $\sigma_0 \in S_{n+1}$, $\sigma_0 \ne e$,
and $\Gamma_0 \in \cL_n[(\sigma_0)]$.
If there exists a sequence $(\Gamma_k)_{k \in \NN^\ast}$
of curves $\Gamma_k \in \cL_n$
with $\iti(\Gamma_k) = w$ for all $k \in \NN^\ast$
and $\lim_{k \to \infty} \Gamma_k = \Gamma_0$ then $w \tok \sigma_0$.
\end{lemma}

Notice that the converse is already known:
if $w \tok \sigma_0$ we constructed in Lemma \ref{lemma:tok0}
a path $H$ of curves of itinerary $w$ tending to $\Gamma_0$.

\begin{proof}
By reparametrizing the domain
and applying a projective transformation, we may assume
$\sing(\Gamma_0) = \{\frac12\}$ and
$\Gamma_0(\frac12) = z_1 = q \acute\sigma_1 \in \widetilde B_{n+1}^{+}$
where $q \in \Quat_{n+1}$ and $\sigma_1 = \eta\sigma_0$.
Consider $\epsilon_0 > 0$ such that $|t-\frac12| \le \epsilon_0$
implies $\Gamma_0(t) \in \cU_{z_1}$.
For $t \in [\frac12 - \epsilon_0, \frac12 + \epsilon_0]$,
define $\Gamma_{0;L}(t) = \cL(z_1^{-1} \Gamma_0(t))$.
Notice that 
$\Gamma_{0;L}(\frac12 - \epsilon_0) \ll
I \ll \Gamma_{0;L}(\frac12 + \epsilon_0)$.
Take $\epsilon_1 \in (0,\frac{\epsilon_0}{2})$ such that
$\Gamma_{0;L}(\frac12 - \epsilon_0) \ll \exp(-\arctan(\epsilon_1) \fh_L)$ and
$\exp(\arctan(\epsilon_1) \fh_L) \ll \Gamma_{0;L}(\frac12 + \epsilon_0)$.
Set $L_{1;-} = \exp(-\arctan(\epsilon_1) \fh_L)$ and
$L_{1;+} = \exp(\arctan(\epsilon_1) \fh_L)$.
Take $\epsilon_2 \in (0,\frac{\arctan(\epsilon_1)}{2})$ such that
$L_{1;-} \ll \Gamma_{0;L}(\frac12 - \epsilon_2) \ll I$ and
$I \ll \Gamma_{0;L}(\frac12 + \epsilon_2) \ll L_{1;+}$.

Take open neighborhoods $A_{0;-}$, $A_{2;-}$, $A_{2;+}$ and
$A_{0;+} \subset \Lo_{n+1}^{1}$ of
$\Gamma_{0;L}(\frac12 - \epsilon_0)$,
$\Gamma_{0;L}(\frac12 - \epsilon_2)$,
$\Gamma_{0;L}(\frac12 + \epsilon_2)$ and
$\Gamma_{0;L}(\frac12 + \epsilon_0)$, respectively, such that,
for all $L_{i;\pm} \in A_{i;\pm}$, $i \in \{0,2\}$,
we have
$L_{0;-} \ll L_{1;-} \ll L_{2;-} \ll I \ll
L_{2;+} \ll L_{1;+} \ll L_{0;+}$.
Let $B_{i;\pm} = z_1 \bQ[A_{i;\pm}] \subset \cU_{z_1}$,
$i \in \{0,2\}$;
notice that
$\Gamma_0(\frac12 \pm \epsilon_i) \in B_{i;\pm}$, $i \in \{0,2\}$.

For sufficiently large $k$ we have
$\Gamma_k(\frac12 \pm \epsilon_i) \in B_{i;\pm}$, $i \in \{0,2\}$.
By Corollary \ref{coro:hausdorff1},
for sufficiently large $k$ we also have
$\sing(\Gamma_k) \subset (\frac12 - \epsilon_2, \frac12 + \epsilon_2)$.
For such large $k$, define a locally convex curve $\tilde\Gamma_k$
which coincides with $\Gamma_k$ except in the intervals
$[\frac12 - \epsilon_0, \frac12 - \epsilon_2]$ and
$[\frac12 + \epsilon_2, \frac12 + \epsilon_0]$.
In these arcs, $\tilde\Gamma_k$ is defined so that
$\tilde\Gamma_k(\frac12 - \epsilon_1) =
z_1 \bQ(L_{1;-}) = z_1 \exp(-\epsilon_1 \fh)$ and
$\tilde\Gamma_k(\frac12 + \epsilon_1) =
z_1 \bQ(L_{1;+}) = z_1 \exp(\epsilon_1 \fh)$:
the above conditions guarantee that this is possible.
The restriction of any such curve $\tilde\Gamma_k$ 
to the interval $[\frac12-\epsilon_1,\frac12+\epsilon_1]$ yields,
by definition, $w \tok \sigma_0$.
\end{proof}

We are ready to define a poset structure in $\Word_n$.
Recall that the Bruhat order in the symmetric group $S_{n+1}$
can be defined as follows:
$\sigma_0 \le \sigma_1$ if and only if
$\Bru_{\sigma_0} \subseteq \overline{\Bru_{\sigma_1}}$.
For $w_0, w_1 \in \Word_n$, set 
\[ w_0 \tok w_1 \qquad\iff\qquad
\cL_n[w_1] \subseteq \overline{\cL_n[w_0]}; \]
notice the reversion.
From the previous results, this coincides with our first definition
for $w_1 = \sigma \in S_{n+1} \smallsetminus \{e\}$.




\begin{lemma}
\label{lemma:preclosure}
For $\sigma_0 \in S_{n+1} \smallsetminus \{e\}$ and $w \in \Word_n$,
the following conditions are equivalent:
\begin{enumerate}[label=(\roman*)]
\item{$w \tok \sigma_0$;}
\item{given $\Gamma_1 \in \cL_n[(\sigma_0)]$
and an open neighborhood  $U \subset \cL_n$ of $\Gamma_1$
there exists $\Gamma \in U \cap \cL_n[w]$;}
\item{given $\Gamma_1 \in \cL_n[(\sigma_0)]$, $\epsilon > 0$,
$\sing(\Gamma_1) = \{t_1\}$
and an open neighborhood $U \subset \cL_n$ of $\Gamma_1$ there exists
$\Gamma \in U \cap \cL_n[w]$
with $\Gamma$ and $\Gamma_1$
coinciding outside $(t_1-\epsilon,t_1+\epsilon)$.}
\end{enumerate}
\end{lemma}

\begin{proof}
Condition (iii) clearly  implies (ii);
(ii) is a rewording of the second definition of (i);
the old definition of (i) implies (iii).
\end{proof}

The set $\NN[X]$ of polynomials with coefficients
in $\NN = \{n \in \ZZ \;|\; n \ge 0\}$ is well ordered by 
\[ p_0 < p_1 \quad\iff\quad \exists C \in \RR, \; \forall x \in \RR, \;
((x > C) \to (p_0(x) < p_1(x))). \]
For $w = (\sigma_1, \ldots, \sigma_\ell) \in \Word_n$, define
\[ \Xdim(w) = \sum_{1 \le j \le \ell} X^{\dim(\sigma_j)} \in \NN[X]. \]
In particular, $w \in \Word_n$ implies $\Xdim(w) < X^{m}$,
$m = {\frac{n(n+1)}{2}}$.

Clearly, given $p \in \NN[X]$ there exist only finitely many words
$w \in \Word_n$ with $p = \Xdim(w)$.
As usual in set theory, identify the ordinal $\omega^{m}$
with the set of ordinals smaller than $\omega^{m}$,
i.e.,
\[ \omega^{m} = \{ \omega^{m-1} c_{m-1} + \cdots + \omega c_1 + c_0; \;
c_0, c_1, \ldots, c_{m-1} \in \NN \}. \]
There exists therefore a bijection $\qord: \Word_n \to \omega^{m}$
such that
if $w_0, w_1 \in \Word_n$ and $\Xdim(w_0) < \Xdim(w_1)$
then $\qord(w_0) < \qord(w_1)$.

\begin{lemma}
\label{lemma:closure}
Consider $\sigma_0 \in S_{n+1} \smallsetminus \{e\}$,
$w = (\sigma_1,\sigma_2,\ldots,\sigma_\ell) \in \Word_n$,
$w \tok (\sigma_0)$, $w \ne (\sigma_0)$.
Then $\ell > 0$
and $\sigma_i < \sigma_0$ (in the Bruhat order),
$\dim(\sigma_i) < \dim(\sigma_0)$ and
$\mult(\sigma_i) < \mult(\sigma_0)$ for $1 \le i \le \ell$.
Furthermore, $\Xdim(w) < \Xdim(\sigma_0)$, $\qord(w) < \qord(\sigma_0)$.
\end{lemma}

\begin{proof}
Notice first of all that the empty word is known to be isolated
in the poset $\Word_n$:
this follows either from Lemma \ref{lemma:novanishingletter} above
or from the known fact that
convex curves form a connected component of $\cL_n$
(as in Lemma \ref{lemma:convex2}).
This implies $k > 0$.
By definition of the Bruhat order, $\sigma_i \le \sigma_0$.
Also, from Lemma \ref{lemma:singleletter} and $w \ne (\sigma_0)$
we have $\sigma_i < \sigma_0$.
It then follows that $\dim(\sigma_i) < \dim(\sigma_0)$
and $\mult(\sigma_i) < \mult(\sigma_0)$; the rest is easy.
\end{proof}

\begin{conj}
\label{conj:multtok}
If $\sigma \in S_{n+1}$,
$w \in \Word_n$ and $w \tok (\sigma)$
then $\mult(w) \le \mult(\sigma)$.
\end{conj}

\begin{rem}
\label{rem:multtok}
Conjecture \ref{conj:multtok} above
implies Conjecture 2.4 in \cite{Shapiro-Shapiro3}.
The cases $n = 2$ and $\inv(\sigma) \le 1$ are easy.
The case $n = 3$ and $\sigma = \eta$ is essentially answered
in \cite{Shapiro-Shapiro4} (with different notation).
Notice that for the words which appear
in the transversal section constructed
in Sections \ref{sect:transversal} and \ref{sect:multitineraries}
the claim is correct.
Conjecture \ref{conj:multtok} would simplify some arguments below
but it will be neither proved nor used.

We might also ask whether $w \tok (\sigma)$ implies 
$\dim(w) \le \dim(\sigma)$.
We do not have a counterexample,
but lean towards disbelieving the implication.
The example $[a_1a_3]a_2[a_1a_3] \tok [a_1a_3a_2]$
shows that $w \tok (\sigma)$ and $w \ne (\sigma)$
do not imply $\dim(w) < \dim(\sigma)$.
\end{rem}




We sum up some of our conclusions.

\begin{lemma}
\label{lemma:xclosure}
The set $\Word_n$ is a poset.
Consider $w_0, w_1 \in \Word_n$,
$w_1 = (\sigma_1,\ldots,\sigma_{\ell_1})$.
\begin{enumerate}[label=(\roman*)]
\item{$w_0 \tok w_1$ if and only if there exist nonempty words
$\tilde w_1, \ldots, \tilde w_{\ell_1} \in \Word_n$ such that
$w_0$ is the concatenation of
$\tilde w_1,\tilde w_2,\ldots,\tilde w_{\ell_1}$ and, for all $j$,
$\tilde w_j \tok \sigma_j$;}
\item{if $w_0 \tok w_1$ then $\hat w_0 = \hat w_1$, 
$\ell(w_0) \ge \ell(w_1)$ and $\qord(w_0) \le \qord(w_1)$; }
\item{$\Word_n$ is well-founded;}
\item{$\Word_n$ is converse well-founded, and given $w_0$
there are only finitely many $w \in \Word_n$ such that $w_0 \tok w$.}
\end{enumerate}
\end{lemma}

\begin{proof}
By definition, the relation $\tok$ is reflexive and transitive.
Item (i) follows from Lemmas \ref{lemma:preclosure} and \ref{lemma:closure}.
It now follows that $w_0 \tok w_1 \tok w_0$ implies
$w_0 = w_1$, completing the proof the $\Word_n$ is a poset.
Item (ii) follows from  Lemma \ref{lemma:closure}.
Item (iii) is a corollary of (ii);
(iv) follows from (ii) and the finiteness of $S_{n+1}$.
\end{proof}

\begin{example}
\label{example:posetsection}
Consider $\sigma \in S_{n+1} \smallsetminus \{e\}$, $\dim(\sigma) = k$,
and the map $\phi: \DD^k \to \cL_n$ transversal to $\cL_n[\sigma]$
constructed in Lemma \ref{lemma:transsection}.
By construction of $\phi$,
if $x \in \DD^k \smallsetminus \{0\}$ and $\iti(\phi(x)) = w$
there exists a continuous map $h: [0,1] \to \DD^k$ such that
$h(0) = 0$, $h(1) = x$ and $\iti(\phi(h(s))) = w$ for all \( s \in (0,1] \).
In particular, $w \tok \sigma$.
As we shall see, the reciprocal does not hold.

Thus, for instance, again with $a = a_1$, $b = a_2$ and $c = a_3$,
it follows from the transversal sections constructed
in Example \ref{example:transversalsection} that
$$ a[ba], [ba]a, b[ab], [ab]b \tok [aba]; \qquad
[ab], [cb], a[cb]a, c[ab]c, [ac]b[ac] \tok [acb]. $$
Notice that $\dim([ac]b[ac]) = \dim([acb]) = 2$.
Also, by transitivity we have
$$ acbac, cabca, [ac]bac, [ac]bca, acb[ac], cab[ac] \tok [acb] $$
but none of the itineraries on the left hand side appear
in the transversal section constructed in
Example \ref{example:transversalsection}.

It is easy, on the other hand, to perturb the map $\phi$
to obtain other maps transversal to $\cL_n[[ac]b[ac]]$.
Take
$$ M = \begin{pmatrix}
-t & -1 & 0 & 0 \\
\frac{t^3}{6} + x t & \frac{t^2}{2} + x & t & 0 \\
ut + 1 & u & 0 & 0 \\
\frac{t^2}{2} + y & t & 1 & 0
\end{pmatrix} $$
where $u$ is to be thought of as a fixed real number
of small absolute value: say, $|u| < 1/4$.
The previous construction corresponds of course to $u = 0$.
Figure \ref{fig:newacb} shows the resulting sections.

\begin{figure}[ht]
\def\svgwidth{125mm}
\centerline{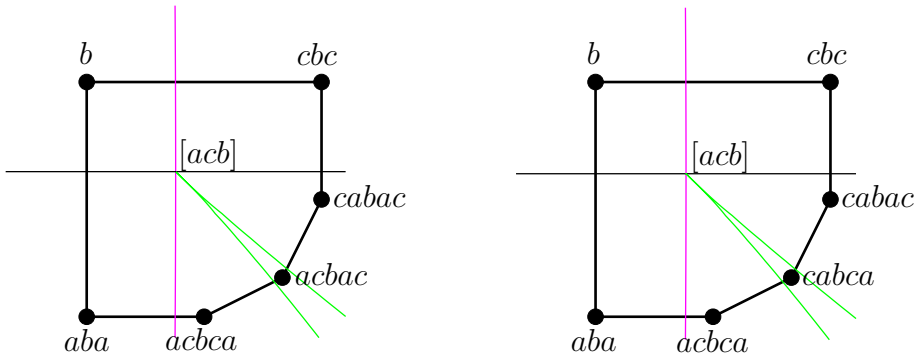}
\caption{Two other transversal sections to $\cL_3[[acb]]$
($u < 0$ and $u > 0$).}
\label{fig:newacb}
\end{figure}
\bigskip

Notice that the two diagrams differ combinatorially.
For $u < 0$, the itinerary $acbac$ appears and $cabca$ does not;
for $u > 0$ it is the other way round.
\end{example}

\section{Lower and upper sets}
\label{sect:lowersets}

A subset $\Ideal \subseteq \Word_n$ is a \emph{lower set} if,
for all $w_1 \in \Ideal$ and $w_0 \in \Word_n$,
$w_0 \tok w_1$ implies $w_0 \in \Ideal$.
In particular, $\emptyset$ and $\Word_n$ are lower sets.
If $\Ideal$ is a lower set then
$$ \cL_n[\Ideal] = \bigsqcup_{w \in \Ideal} \cL_n[w] \subseteq \cL_n $$
is an open subset.
Given $w_0 \in \Word_n$,
let $\Ideal(w_0) = \{w \in \Word_n, w \tok w_0\}$ and
$\Ideal^\ast(w_0) = \Ideal(w_0) \smallsetminus \{w_0\}$:
both are lower sets.
The map $\phi: \DD^k \to \cL_n$ transversal section to $\cL_n[\sigma]$
constructed in Lemma \ref{lemma:transsection}
is of the form  $\phi: \DD^k \to \cL_n[\Ideal(\sigma)] \subset \cL_n$.
Figure \ref{fig:subaba} is a diagram of 
$\Ideal([aba]) = \{ aa, abab, bb, baba, [ba]a, a[ba], [ab]b, b[ab], [aba] \}$;
compare with Figure \ref{fig:aba}.

\begin{figure}[ht]
\def\svgwidth{10cm}
\centerline{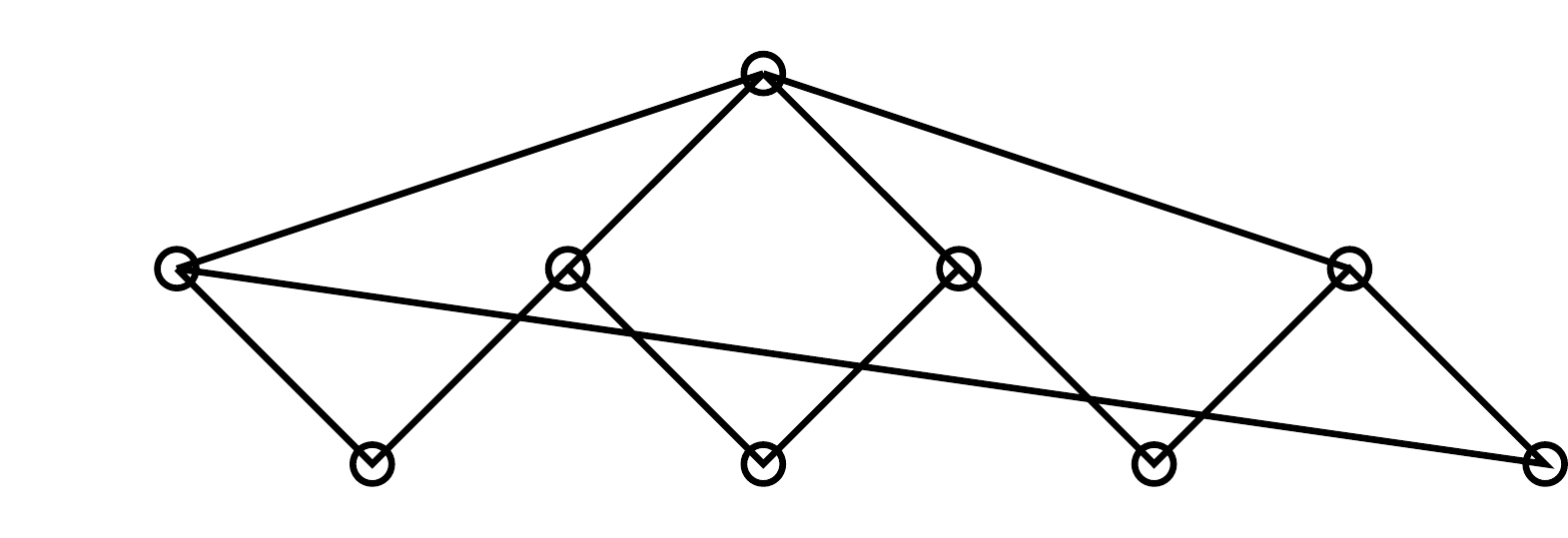}
\caption{The lower set $\Ideal([aba])$.}
\label{fig:subaba}
\end{figure}

For $\alpha \in \omega^{m}$, let
\[ \Ideal_{(\alpha)} = \{ w \in \Word_n \;|\; \qord(w) < \alpha \}, \]
a lower set.
For instance, $\Ideal_{(\omega)}$ is the set of words of degree $0$,
i.e., words whose letters are generators $a_k$.
The set $\cL_n[\Ideal_{(\omega)}]$
is a disjoint union of contractible open sets $\cL_n[w]$,
$w \in \Ideal_{(\omega)}$.

Similarly, $\Filter \subseteq \Word_n$ is an \emph{upper set} if,
for all $w_0 \in \Filter$ and $w_1 \in \Word_n$,
$w_0 \tok w_1$ implies $w_1 \in \Filter$.
If $\Filter$ is an upper set then
$$ \cL_n[\Filter] = \bigsqcup_{w \in \Filter} \cL_n[w] \subseteq \cL_n $$
is a closed subset.

A few examples are in order.

\begin{example}
\label{example:uppersets}
The following are examples of upper sets:
\begin{gather*}
\Filter_{2,2} = \{[aba]\} \subset \Word_2, \qquad
\Filter_{2,3} = \{[aba], [bacb], [bcb] \} \subset \Word_3, \\
\{ [abc], [ac], [cba] \} \subset \Word_3, \qquad
\Filter_{x,3} = \{[cba]\} \subset \Word_3, \\
\Filter_{2,4} = \{[aba], [bacb], [bcb], [cbdc], [cdc] \} \subset \Word_4, \\
\Filter_{4,4} = \{[34521], [32541], [52341], [52143], [54123] \}
\subset \Word_4.
\end{gather*}
The example $\Filter_{2,2}$ follows directly from Lemma \ref{lemma:xclosure}:
there is no letter in $S_3$
of greater dimension than $[aba]$.
For $\Filter_{2,3}$,
verify that there are no other letters
$\sigma \in S_4$ with $\hat\sigma = \longhat([aba]) = -1$
(notice also that $[aba] \tok [bacb]$, $[bcb] \tok [bacb]$).
Also, the only letters $\sigma \in S_4$ with
$\hat\sigma = \hat a\hat c$ are $[abc]$, $[ac]$ and $[cba]$
(notice that $[ac] \tok [abc]$, $[ac] \tok [cba]$;
see also Lemma \ref{lemma:smalln}).

The only letters $\sigma \in S_5$ with $\hat\sigma = -1$
are the elements of $\Filter_{2,4}$ and $\eta$.
We do not have $\sigma \tok \eta$ for $\sigma \ne \eta$, however.
Indeed, the only point in $\overline{\Pos_\eta} \cap \overline{\Neg_\eta}$
is the identity.
Consider a locally convex curve $\Gamma_L: [-1,1] \to \Lo_{5}^{1}$
with $\Gamma_L(-1) \in \Neg_\eta$ and $\Gamma_L(1) \in \Pos_\eta$.
Set $\Gamma(t) = q \bQ(\Gamma_L(t))$, $q \in \Quat_{n+1}$.
If $\Gamma_L(t) = I$ for some $t$ then $\iti(\Gamma) = \eta$.
Otherwise, $\Gamma_L$ must cross $\partial\Neg_\eta$ and $\partial\Pos_\eta$
at two distinct times $t_{-} < t_{+}$, $t_{-}, t_{+} \in \sing(\Gamma)$:
thus $\iti(\Gamma) \ne \sigma$.
The only letters $\sigma \in S_5 \smallsetminus \{e\}$ with $\hat\sigma = 1$
are the elements of $\Filter_{4,4}$.
\end{example}

\section{Valid complexes $\cD_n[\Ideal]$}
\label{sect:valid}

Here $\DD^k$ is the closed disk of dimension $k$.
Let $\Ideal \subseteq \Word_n$ be a lower set.
A \emph{valid (CW) complex} $\cD_n[\Ideal]$ (for $\Ideal$)
has the following ingredients and properties:
\begin{enumerate}[label=(\roman*)]
\item{A CW complex $\cD_n[\Ideal]$
with one cell of dimension $\dim(w)$
for each $w \in \Ideal$.
The continuous gluing maps
$g_w: \partial\DD^{\dim(w)} \to \cD_n[\Ideal]$
have image contained in $\cD_n[\Ideal_w^\ast]$.
Let $i_w: \DD^{\dim(w)} \to \cD_n[\Ideal]$ be the inclusion
(with quotient at the boundary).}
\item{For each $w \in \Ideal$,
we have a continuous map $c_w: \DD^{\dim(w)} \to \cL_n[\Ideal_w]$
(which will be called the $w$ cell).}
\item{The maps $c_w$ are compatible:
if $i_{w_0}(p_0) = i_{w_1}(p_1)$ then
$c_{w_0}(p_0) = c_{w_1}(p_1)$.
In particular, if $w_1 \tok w_0$,
$p_0 \in \partial\DD^{\dim(w_0)}$,  $p_1 \in \DD^{\dim(w_1)}$ and
$g_{w_0}(p_0) = i_{w_1}(p_1)$
then $c_{w_0}(p_0) = c_{w_1}(p_1)$.
The union of the cells $c_w$
defines a continuous map $c_\Ideal: \cD_n[\Ideal] \to \cL_n[\Ideal]$.}
\item{For any $w \in \Ideal$,
the cell $c_w$ is a topological embedding in the interior of $\DD^{\dim(w)}$,
smooth in the ball of radius $\frac12$,
and intersects some tubular neighborhood $\hat\cA_w \supset \cL_n[w]$
transversally around $c_w(0) \in \cL_n[w]$.
More precisely, for $\hat F_w: \hat\cA_w \to \BB^{\dim(w)}$
as in the characterization of such tubular neighborhoods,
the map $\hat F_w \circ c_w$ is a homeomorphism from
an open neighborhood of $0$ to an open neighborhood of $0$.  }
\item{If $s \in \DD^{\dim(w)}$, $s \ne 0$,
then $c_w(s) \in \cL_n[\Ideal^\ast(w)]$
(and therefore $c_w(s) \notin \cL_n[w]$).}
\end{enumerate}

We abuse notation and identify a valid complex
with the corresponding family of cells and write
$\cD_n[\Ideal] = (c_w)_{w \in \Ideal}$
(notice that this family yields all the necessary information).
A valid complex is \emph{good} if for every lower set
$\tilde\Ideal \subseteq \Ideal$
the map
$c_{\tilde\Ideal}: \cD_n[\tilde\Ideal] \to \cL_n[\tilde\Ideal]$
is a weak homotopy equivalence.



\begin{rem}
\label{rem:validcomplex}
Given a valid complex $\cD_n[\Ideal] = (c_w)_{w \in \Ideal}$
and a lower set $\tilde\Ideal \subseteq \Ideal$,
the subcomplex $(c_w)_{w \in \tilde\Ideal}$
is also valid (for $\tilde\Ideal$).
We call this subcomplex $\cD_n[\tilde\Ideal]$
the \emph{restriction} of $\cD_n[\Ideal]$ to $\tilde\Ideal$.
A restriction of a good complex is also good.

Let $J$ be a totally ordered set.
Let $(\Ideal_j)_{j \in J}$ be a family of lower sets such that
$j_0 < j_1$ implies $\Ideal_{j_0} \subseteq \Ideal_{j_1}$.
Let $\Ideal = \bigcup_j \Ideal_j$,
which is therefore also a lower set.
For each $j \in J$, let $\cD_n[\Ideal_j]$ be a valid complex.
Assume that if $j_0 < j_1$ then
$\cD_n[\Ideal_{j_0}]$ is the restriction
of $\cD_n[\Ideal_{j_1}]$ to $\Ideal_{j_0}$.
Then the union $\cD_n[\Ideal]$ of all the valid complexes 
$\cD_n[\Ideal_j]$ (for all $j \in J$) is a valid complex for $\Ideal$.
A family $(\cD_n[\Ideal_j])_{j \in J}$ of valid complexes is a \emph{chain}
if it satisfies the conditions above.
\end{rem}

\begin{lemma}
\label{lemma:goodchain}
Let $(\cD[\Ideal_j])_{j \in J}$ be a chain of good complexes.
Let $\Ideal = \bigcup_j \Ideal_j$ and
$\cD_n[\Ideal]$ be the union of the complexes $\cD_n[\Ideal_j]$.
Then $\cD_n[\Ideal]$ is a good complex.
\end{lemma}

\begin{proof}
Let $\alpha: \Ss^k \to \cD_n[\Ideal]$ be a continuous map
which is homotopically trivial in $\cL_n[\Ideal]$,
i.e., there exists a map
$A_0: \DD^{k+1} \to \cL_n[\Ideal]$,
$A_0|_{\Ss^k} = c_{\Ideal} \circ \alpha$.
By compactness, there exists $j \in J$ such that the image of $A_0$
is contained in $\cL_n[\Ideal_j]$ so that we may write
$A_0: \DD^{k+1} \to \cL_n[\Ideal_j]$.
Since $\cD_n[\Ideal_j]$ is good,
there exists 
$A_1: \DD^{k+1} \to \cD_n[\Ideal_j]$
such that $A_1|_{\Ss^k} = \alpha$.

Let $\alpha_0: \Ss^k \to \cL_n[\Ideal]$ be a continuous map.
By compactness, there exists $j \in J$ such that the image
of $\alpha_0$ is contained in $\cL_n[\Ideal_j]$.
Since $\cD_n[\Ideal_j]$ is good,
there exists $\alpha_1: \Ss^k \to \cD_n[\Ideal_j] \subset \cD_n[\Ideal]$
such that $c_{\Ideal_j} \circ \alpha_1$
is homotopic to $\alpha_0$ in $\cL_n[\Ideal_j]$
(and therefore also in $\cL_n[\Ideal]$).

This completes the proof that $c_{\Ideal}: \cD_n[\Ideal] \to \cL_n[\Ideal]$
is a weak homotopy equivalence.
The proof for $\tilde\Ideal \subseteq \Ideal$ is similar.
\end{proof}

\begin{lemma}
\label{lemma:goodstep}
Let $\tilde\Ideal \subset \Ideal \subseteq \Word_n$ be lower sets
with $\Ideal \smallsetminus \tilde\Ideal = \{w_0\}$.
Let $\cD_n[\tilde\Ideal]$ be a good complex (for $\tilde\Ideal$).
Then this complex can be extended to a valid complex
$\cD_n[\Ideal]$ (for $\Ideal$).
Moreover, all such valid complexes are good.
\end{lemma}

Before proving Lemma \ref{lemma:goodstep} we present 
the main conclusion for this section.

\begin{lemma}
For any lower set $\Ideal \subseteq \Word_n$
there exists a valid complex $\cD_n[\Ideal] \subset \cL_n[\Ideal]$.
All such valid complexes are good.
\end{lemma}

\begin{proof}
We prove by transfinite induction on $\alpha \le \omega^{m}$
that there exists a chain of good complexes
$(\cD_n[\Ideal \cap \Ideal_{(\beta)}])_{\beta < \alpha}$.
For $\alpha = 0$ there is nothing to do.
For $\alpha = \tilde\alpha + 1$ apply Lemma \ref{lemma:goodstep}.
For $\alpha$ a limit ordinal apply Lemma \ref{lemma:goodchain}.

The fact that all valid complexes are good is likewise
proved by transfinite induction.
We again use Lemma  \ref{lemma:goodstep} for successor ordinals
and Lemma \ref{lemma:goodchain} for limit ordinals.
\end{proof}




\begin{proof}[Proof of Lemma \ref{lemma:goodstep}]
Assume $\cD_n[\tilde\Ideal]$ given (and good):
we construct the cell $c_{w_0}$ so that $\cD_n[\Ideal]$ is also good.
Let $k = \dim(w_0)$.
For a small ball $B \subset \RR^{k}$ around the origin,
construct as in Section \ref{sect:transversal}
a smooth map
$\tilde c: B \to \cL_n[\Ideal] \subseteq \cL_n$
with $\tilde c(0) = \Gamma_0 \in \cL_n[w_0]$
and topologically transversal to $\cL_n[w_0]$ at this point.
By taking $r > 0$ sufficiently small,
the image under $\tilde c$ of a ball $B(2r) \subset B$
of radius $2r$ around the origin satisfies the following condition:
$s \in B(2r)$ and $s \ne 0$ imply
$\tilde c(s) \in \cL_n[\Ideal^\ast_{w_0}] \subseteq \cL_n[\tilde\Ideal]$.

Let $S(r) \subset B(2r)$ be the sphere of radius $r$ around the origin
so that $\tilde c|_{S(r)}: S(r) \to \cL_n[\Ideal^\ast_{w_0}]$.
Since $\cD_n[\tilde\Ideal]$ is good,
the map $c_{\Ideal^\ast_{w_0}}:
\cD_n[\Ideal^\ast_{w_0}] \to \cL_n[\Ideal^\ast_{w_0}]$
is a weak homotopy equivalence.
There exists therefore a map
$g_{w_0}: \Ss^{k-1} \to \cD_n[\Ideal^\ast_{w_0}]$
such that $c_{\Ideal^\ast_{w_0}} \circ g$
is homotopic (in $\cL_n[\Ideal^\ast_{w_0}]$) to $\tilde c|_{S(r)}$.
In other words, there exists $c_{w_0}: \DD^k \to \cL_n[\Ideal_{w_0}]$
coinciding with $\tilde c$ in $B(r)$,
assuming values in $\cL_n[\Ideal^\ast_{w_0}]$ outside $0$
and with boundary $c_{\Ideal^\ast_{w_0}} \circ g_{w_0}$.
This completes the construction of a valid complex $\cD_n[\Ideal]$
and the proof of the first claim.


We now prove the second claim.
More precisely, let $c_{w_0}$ be such that extending
the good complex $\cD_n[\tilde\Ideal]$ by $c_{w_0}$
is a valid complex (for $\Ideal$):
we prove that it is a good complex.
In other words, we prove 
that the map $c_{\hat\Ideal}: \cD_n[\hat\Ideal] \to \cL_n[\hat\Ideal]$
is a weak homotopy equivalence for any lower set
$\hat\Ideal \subseteq \Ideal$.
Again, we present the proof for $\hat\Ideal = \Ideal$;
the general case is similar.
The construction is similar to that of certain classical results
involving CW complexes;
a minor difference here is that $\cL_n[w_0] \subset \cL_n[\Ideal]$
is not a smooth submanifold.
Recall that $\cL_n[w_0] \subseteq \cL_n[\Ideal]$
is a contractible closed subset (of $\cL_n[\Ideal]$) and
a (globally) collared topological submanifold
of codimension $d = \dim(w_0)$.
Indeed, in Lemma \ref{lemma:submanifold}
we constructed a tubular neighborhood of $\hat\cA_{w_0} \supset \cL_n[w_0]$,
a projection $\Pi: \hat\cA_{w_0} \to \cL_n[w_0]$
and a map $\hat F = \hat F_{w_0}:  \hat\cA_{w_0} \to \BB^d$
(where $\BB^d$ is the unit open ball)
such that $(\Pi,\hat F): \hat\cA_{w_0} \to (\cL_n[w_0], \BB^d)$
is a homeomorphism (see Remark \ref{rem:Ftransversal}).
By construction,
the map $c_{w_0}$ intersects $\hat\cA_{w_0} \supset \cL_n[w_0]$
transversally so that
$c_{w_0}^{-1}[\hat\cA_{w_0}]$ is an open neighborhood of $0 \in \BB^d$.
The map $\hat F \circ c_{w_0}$ (where defined) is a homeomorphism
from a neighborhood of $0$ to a neighborhood of $0$;
indeed, in the example constructed in the first part of the proof,
it is the multiplication by a positive constant.
As we shall see, the tubular neighborhood (and not smoothness)
is the essential ingredient in the proof.

Consider a compact manifold $M_0$ of dimension $k$
and continuous map $\alpha_0: M_0 \to \cL_n[\Ideal]$.
We prove that there exists
$\alpha_1: M_0 \to \cD_n[\Ideal]$
such that
$\alpha_0$ is homotopic to $c_{\Ideal} \circ \alpha_1$. 
First deform $\alpha_0$ to obtain a map
$\alpha_{\frac{1}{3}}$ which intersects $\hat\cA_{w_0}$ transversally;
this is similar to the more familiar construction for smooth maps.
By using the contractibility of $\cL_n[w_0]$,
we may deform $\alpha_{\frac{1}{3}}$ to obtain a map 
$\alpha_{\frac{2}{3}}$ which intersects $\hat\cA_{w_0}$
only along the image of $c_{w_0}$.
More precisely, take $\Gamma_0 = c_{w_0}(0) \in \cL_n[w_0]$ as a base point;
let $H: [0,1] \times \cL_n[w_0] \to \cL_n[w_0]$
be a homotopy such that,
for all $\Gamma \in \cL_n[w_0]$,
$H(0,\Gamma) = \Gamma$ and
$H(1,\Gamma) = \Gamma_0$;
assume without loss of generality that 
$\Pi(c_w(p)) = \Gamma_0$ for $p$ is an open neighborhood
$B_0 \subset \BB^d$, $0 \in B_0$;
let $\bump: \BB^d \to [0,1]$ be a bump function
with support contained in the neighborhood above
and constant equal to $1$ in a smaller open ball
$B_1 \subset B_0$, $0 \in B_1$.
For $s \in [\frac13,\frac23]$, we define $\alpha_s$
to coincide with $\alpha_{\frac13}$ outside $\hat\cA_{w_0}$;
if $\alpha_{\frac13}(p) \in \hat\cA_{w_0}$ define
$\alpha_s(p) \in \hat\cA_{w_0}$ to satisfy
\[
\Pi(\alpha_s(p)) = 
H((3s-1)\bump(\hat F_{w_0}(\alpha_{\frac13}(p))),\Pi(\alpha_{\frac13}(p))),
\quad
\hat F_{w_0}(\alpha_s(p)) = \hat F_{w_0}(\alpha_{\frac13}(p)). \]
Thus, the set of points $p \in M_0$ such that
$\alpha_{\frac23}(p) \in \hat\cA_{w_0}$ and
$\hat F_{w_0}(\alpha_{\frac23}(p)) \in B_1$
consists of open tubes taken to the image of the cell $c_{w_0}$.
By removing these open tubes we have $M_1 \subseteq M_0$,
a manifold with boundary, also of dimension $k$.
Let $\cD_n^\ast[\Ideal] =
\cD_n[\Ideal] \smallsetminus i_{w_0}[B_1]$ so that
both the map $c_{\Ideal}|_{\cD_n^\ast[\Ideal]}:
\cD_n^\ast[\Ideal] \to \cL_n[\tilde\Ideal]$
and the inclusion $\cD_n[\tilde\Ideal] \subset \cD_n^\ast[\Ideal]$
are weak homotopy equivalences
(since $\cD_n[\tilde\Ideal]$ is assumed to be good).
The restriction
$\alpha_{\frac{2}{3}}|_{M_1}: M_1 \to \cL_n[\tilde\Ideal]$
has boundary
$\alpha_{\frac23}|_{\partial M_1} =
c_{\Ideal} \circ \beta_1$
for $\beta_1: \partial M_1 \to \cD_n^\ast[\Ideal]$.
There exists
$\beta: M_1 \to \cD_n^\ast[\Ideal]$
with $\beta_1 = \beta|_{\partial M_1}$
and such that
$c_{\Ideal} \circ \beta: M_1 \to \cL_n[\tilde\Ideal]$
is homotopic to 
$\alpha_{\frac{2}{3}}|_{M_1}: M_1 \to \cL_n[\tilde\Ideal]$
(in $\cL_n[\tilde\Ideal]$).
Define $\alpha_1$ to coincide with $\beta$ in $M_1$
and such that $c_{w_0} \circ \alpha_1$
coincides with $\alpha_{\frac{2}{3}}$
in the tubes $M_0 \smallsetminus M_1$.
This is our desired map.

Conversely, consider a compact manifold with boundary $M$
and its boundary $\partial M = N$.
Consider a maps $\alpha_0: M \to \cL_n[\Ideal]$
and $\beta_0: N \to \cD_n[\Ideal]$
such that $c_{\Ideal} \circ \beta_0 = \alpha_0|_{N}$.
We prove that there exists a map $\alpha_1: M \to \cD_n[\Ideal]$
with $\alpha_1|_{N} = \beta_0$.
Furthermore, $\alpha_0$ and
$c_{\Ideal} \circ \alpha_1$ are homotopic in $\cL_n[\Ideal]$,
with the homotopy constant in the boundary.
Again, we may assume without loss of generality that
$\alpha_0$ is transversal to $\hat\cA_{w_0}$.
As in the previous paragraph,
use the contractibility of $\cL_n[w_0]$ to deform
$\alpha_0$ in $\hat\cA_{w_0}$ towards $c_{w_0}$,
thus defining $\alpha_{\frac12}$;
notice that this keeps the boundary fixed, as required.
Again remove tubes around points taken to $c_{w_0}$,
thus defining $M_1 \subset M$, also a manifold with boundary.
By hypothesys, $\alpha_{\frac12}$ can be deformed in $M_1$
to obtain $\alpha_1: M_1 \to \cD_n[\tilde\Ideal]$;
more precisely, $\alpha_{\frac12}: M_1 \to \cL_n[\tilde\Ideal]$
is homotopic to $c_{\tilde\Ideal}\circ\alpha_1$.
As in the previous paragraph we fill in the tubes to define
$\alpha_1: M \to \cD_n[\Ideal]$, the desired map.
\end{proof}

\begin{figure}[ht]
\def\svgwidth{12cm}
\centerline{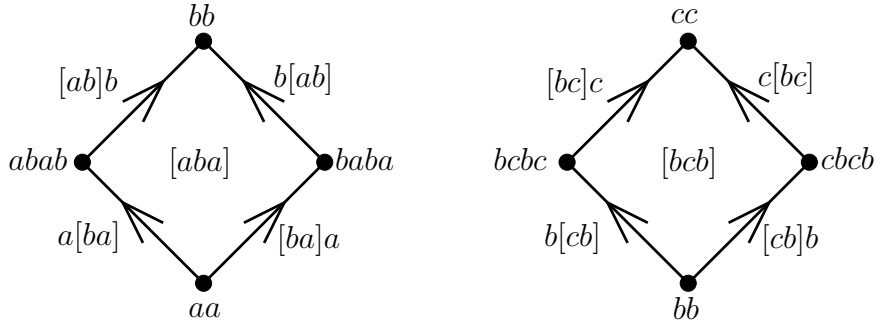}
\caption{The cells $c_{[aba]}$ and $c_{[bcb]}$.}
\label{fig:ababcb}
\end{figure}

\begin{example}
\label{example:goodcomplex}
Following the above construction for $\Ideal_{[aba]}$
we have the cell shown in Figure \ref{fig:ababcb};
the transversal map $\tilde c$ in the proof above
is shown in Figure \ref{fig:aba}.
Notice that $c_{[bcb]}$ is very similar.
The immediate predecessors of $[aba]$ are
$[ba]a$, $b[ab]$, $[ab]b$ and $a[ba]$ (see Figure \ref{fig:subaba})
and we have
\[ \partial[aba] = [ba]a + b[ab] - [ab]b - a[ba], \quad
\partial[bcb] = [cb]b + c[bc] - [bc]c - b[cb]. \]
\end{example}

\begin{rem}
\label{rem:goodcomplex}
As we shall see,
for cells of very low dimension the complex
can be taken to be polyhedric;
it is not clear whether this is true in general.

It seems to be the case that the complex can be
constructed so that cells are injective
and with disjoint images except for the gluing at the boundaries.
We shall not attept to clarify this issue here.
\end{rem}

\section{The complexes
$\cD_n[\Ideal_{(\omega)}] \subset
\cD_n[\Ideal_{(\omega 2)}] \subset \cD_n[\Ideal_{(\omega^2)}]$
}
\label{sect:Dn1}

In this section we consider a few simple examples
of valid complexes.
From now on, when $n \le 4$ we write
$a = a_1$, $b = a_2$, $c = a_3$ and $d = a_4$.
Recall that $\Ideal_{(\omega)} \subset \Word_n$ is
the lower set of words of dimension $0$.
For each $w \in \Ideal_{(\omega)}$,
let $c_w \in \cL_n[w]$ be an arbitrary curve:
we think of $c_w$ as a vertex of the complex $\cD_n$.
In other words, the $0$-skeleton of $\cD_n$ is the infinite countable
set of vertices $\cD_n[\Ideal_{(\omega)}] = \{c_w, w \in \Word^{(0)} \}$;
the inclusion $\cD_n[\Ideal_{(\omega)}] \subset \cL_n[\Ideal_{(\omega)}]$
is a homotopy equivalence.

The subset of $\Ideal_{(\omega 2)} \subset \Word_n$
of words of dimension at most $1$ is also a lower set.
A word of dimension $1$ is of the form $w = w_0 [a_k a_l] w_1$
where $w_0, w_1 \in \Ideal_{(\omega)} \subset \Word_n$
are (possibly empty) words of dimension $0$
and $k \ne l$ so that $[a_ka_l] \in S_{n+1}$ is an element of dimension $1$
(i.e., $\inv(a_ka_l) = 2$).
There are three cases:
case (i) is $l = k-1$,
case (ii) is $l = k+1$ and
case (iii) is $l > k+1$
(if $l < k-1$ we write $a_la_k$ instead;
notice that in this case the permutations $a_k$ and $a_l$ commute).
In each case the stratum $\cL_n[w]$ is a hypersurface
with an open stratum on either side:
let us call them $\cL_n[w^{+}]$ and $\cL_n[w^{-}]$.
The words $w^{+}, w^{-} \in \Word_n$ have dimension $0$;
their values according to case are:
\begin{enumerate}[label=(\roman*)]
\item{$w^{-} = w_0 a_l w_1$, $w^{+} = w_0 a_k a_l a_k w_1$;}
\item{$w^{-} = w_0 a_k a_l a_k w_1$, $w^{+} = w_0 a_l w_1$;}
\item{$w^{-} = w_0 a_k a_l w_1$, $w^{+} = w_0 a_l a_k w_1$.}
\end{enumerate}
From Lemma \ref{lemma:closure} it follows easily that
the only two words $\tilde w \in \Word_n$ such that $\tilde w \tok w$
are $\tilde w = w^{\pm}$.
In order to construct the $1$-skeleton $\cD_n[\Ideal_{(\omega 2)}]$ of $\cD_n$,
we add an oriented edge $c_w$ from $c_{w^{-}}$ to $c_{w^{+}}$.
The edge may be assumed to be contained in
$\cL_n[w^{-}] \cup \cL_n[w] \cup \cL_n[w^{+}]$
and to cross $\cL_n[w]$ (topologically) transversally and precisely once;
edges are also assumed to be simple, and disjoint except at endpoints.
Again, the inclusion
$\cD_n[\Ideal_{(\omega 2)}] \subset \cL_n[\Ideal_{(\omega 2)}]$
is a homotopy equivalence.

The four sides of Figure \ref{fig:aba} above provide examples
of edges of cases (i) and (ii).
Figure \ref{fig:abbaac} shows the edges of $\cD_n$.
Here we omit the initial and final words $w_0$ and $w_1$
(since they are not involved anyway).
We also prefer to present examples
(instead of spelling out conditions as above;
a more formal discussion is given
in Sections \ref{sect:valid} and \ref{sect:transversal}).
Thus, cases (i), (ii) and (iii) are represented by
$[ba]$, $[ab]$ and $[ac]$, respectively.

\begin{figure}[ht]
\def\svgwidth{10cm}
\centerline{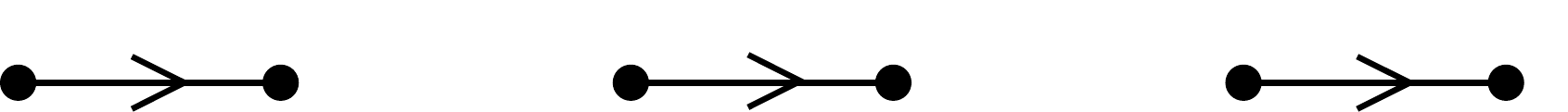}
\caption{Examples of edges of $\cD_n$.}
\label{fig:abbaac}
\end{figure}

In a homological language, we write
\[ \partial[ba] = bab - a, \quad \partial[ab] = b - aba, \quad
\partial[ac] = ca - ac. \]
Notice that in all cases $\hat w = \hat w^{+} = \hat w^{-} \in \Quat_{n+1}$,
consistently with the fact that we are constructing disjoint complexes
$\cD_n(z)$, $z \in \Quat_{n+1}$.

As a simple application of our methods,
we compute the connected components of $\cD_n(z)$
(i.e., we are reproving in our notation the main result of \cite{Shapiro-Shapiro}).
We need only consider words of dimension $0$ (the vertices of $\cD_n$),
i.e., words in the generators $a_k$, $1 \le k \le n$.
Write $w_0 \sim w_1$ if $c_{w_0}$ and $c_{w_1}$
are in the same connected component;
thus, if $w_0 \tok w_1$ then $w_0 \sim w_1$ and
if $w_0 \sim w_1$ then $\hat w_0 = \hat w_1$.

For $w$ the empty word, the vertex $c_w$ is not attached
to any edge and thus forms a contractible connected component
in the complex $\cD_n$.
This of course corresponds to $\cL_n[w] \subset \cL_n$,
the component of convex curves;
notice that $\hat a \hat a \hat a \hat a = \hat w = 1$
but $aaaa \not\sim w$.

\begin{prop}
\label{prop:connected}
Consider two non-empty words $w_0, w_1 \in \cD_n$ of dimension $0$.
Then $w_0 \sim w_1$ if and only if $\hat w_0 = \hat w_1 \in \Quat_{n+1}$.
\end{prop}


\begin{proof}
We have already seen that $w_0 \sim w_1$ implies $\hat w_0 = \hat w_1$.
We prove the other implication.  A \emph{basic word} is either:
\begin{enumerate}[label=(\roman*)]
\item{ a non-empty word $a_{k_1}a_{k_2}\cdots a_{k_l}$
with $k_1 < k_2 < \cdots k_l$; }
\item{ of the form $aaa_{k_1}a_{k_2}\cdots a_{k_l}$
with $k_1 < k_2 < \cdots k_l$ (here the word $aa$
corresponding to $l = 0$ is allowed); }
\item{ $aaaa$.}
\end{enumerate}
In particular, words of length $1$ are basic.
Clearly, for each $z \in \Quat_{n+1}$ there exists a unique basic word $w$
with $z = \hat w$.

Using the edges above, first notice that
\[ aa \sim abab \sim bb \sim bcbc \sim cc \sim \cdots \sim a_ka_k; \]
also 
\[ a_{k+1}a_k \sim a_{k+1}a_{k+1}a_ka_{k+1} \sim aaa_ka_{k+1}; \]
furthermore, $baa \sim babab \sim aab$ and, for $k > 2$,
$a_k aa \sim aa a_k$.
Also,
\[ a \sim bab \sim ababababa \sim aaaaa. \]
Thus $aa$ commutes with all generators $a_k$
and can be brought to the beginning of the word;
other generators either commute ($a_ka_l \sim a_la_k$ if $|k-l| \ne 1$)
or anticommute ($a_{k+1}a_k \sim aaa_ka_{k+1}$).
Thus, for an arbitrary non-empty word $w$,
generators can be arranged in increasing order of index,
at the price of creating copies of $aa$
which are taken to the beginning of the word.
Duplicate generators can also be transformed into further copies of $aa$.
Finally, if there are more than $4$ copies of $a$,
they can be removed $4$ by $4$ thus arriving at a basic word.
\end{proof}



We now construct the complex $\cD_n[\Ideal_{(\omega^2)}]$.
Notice that $w \in \Ideal_{(\omega^2)}$ if and only if
$w$ is of the form $w = w_0 \sigma_1 w_1 \cdots \sigma_l w_l$
where $\dim(w_j) = 0$ and $\dim(\sigma_j) = 1$
(some of the $w_j$ may equal the empty word).
Set $c_w$ to be a \emph{product cell} of dimension $l$,
i.e., the $l$-th dimensional cube
\[ c_w = c_{w_0} \times c_{\sigma_1} \times c_{w_1} \times \cdots
\times c_{\sigma_l} \times c_{w_l}; \]
the gluing instructions are legitimate.
See Figure \ref{fig:baabacbac} for the following examples:
\begin{gather*}
\partial([ba][ab]) = [ba]aba + bab[ab] - [ba]b - a[ab], \\
\partial([ac]b[ac]) = [ac]bac + cab[ac] - [ac]bca - acb[ac].
\end{gather*}

\begin{figure}[ht]
\def\svgwidth{10cm}
\centerline{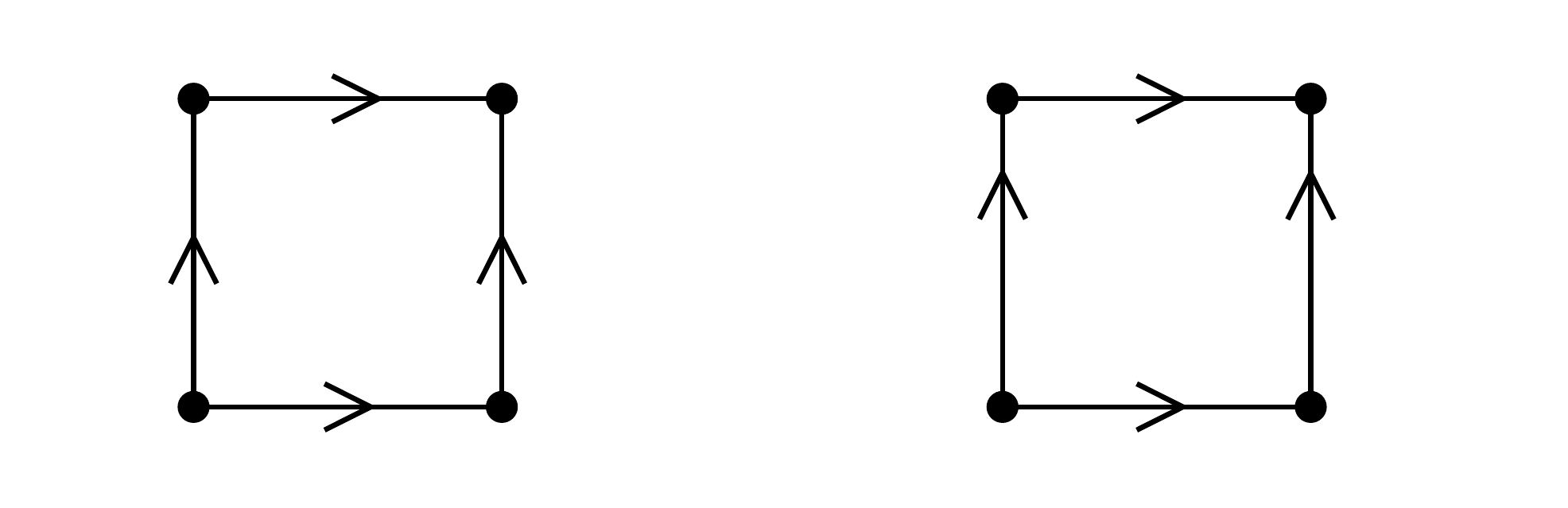}
\caption{The cells $[ba][ab]$ and $[ac]b[ac]$. }
\label{fig:baabacbac}
\end{figure}

The construction of product cells works in greater generality.
If $w \in \Word_n$ is a word of length greater than $1$,
possibly containing more than one letter of positive dimension,
define $c_w$ as a \emph{product cell}.
As before, write $w = w_0 \sigma_1 w_1 \cdots \sigma_l w_l$
where $\dim(w_j) = 0$ and $\dim(\sigma_j) > 0$
(some of the $w_j$ may equal the empty word).
Set 
\[ c_w = c_{w_0} \times c_{\sigma_1} \times c_{w_1} \times \cdots
\times c_{\sigma_l} \times c_{w_l}. \]
Here we assume the cells $c_{\sigma_j}$ to have been previously constructed.

Also, let $\sigma \in S_{n+1}$ be a letter of dimension $k$
such that $1^\sigma  = 1$.
Write $\sigma = [a_{n_1}\cdots a_{n_{k+1}}]$ and $s = -1 +\min n_j > 0$.
Set $\tilde\sigma = [a_{n_1 - s}\cdots a_{n_{k+1}- s}]$:
the cell $c_{\tilde\sigma}$ is assumed to be already constructed.
Define $c_{\sigma}$ from $c_{\tilde\sigma}$
by adding $s$ to the index of every generator of every letter.
Notice that this fits with our construction of the $1$-skeleton;
see also Figure  \ref{fig:ababcb} for $c_{[bcb]}$.

\section{Cells of dimension $2$}
\label{sect:Dn2}

Given $n$ and the simplifications in the previous section,
there is a finite (and short) list of possibilities
of letters of dimension $2$.
In $S_3$, the only letter of dimension $2$ is $[aba]$:
Figure \ref{fig:subaba} shows the elements of $\Word_2$ (or $\Word_n$)
below $[aba]$, i.e., the smallest lower set containing $[aba]$.
See also Figures \ref{fig:aba} and \ref{fig:ababcb} for 
a transversal surface to the submanifold $\cL_2[[aba]]$
and for the cell $c_{[aba]}$.

In $S_4$ we also have $[bcb]$
(which is of course similar to $[aba]$, as in Figure \ref{fig:ababcb})
and $[acb]$, which we already discussed and for which valid cells
are shown in Figure \ref{fig:newacb} (see also Figure \ref{fig:planeacb}).
The three remaining cells are
and $[abc]$, $[bac]$ and $[cba]$,
for which transversal sections and valid cells are shown
in Figure \ref{fig:cD3};
in a homological notation:
\begin{align*}
\partial[abc] &= abc[ab]+a[bc]+[ac]-[bc]a-[ab]cba+ab[ac]ba, \\
\partial[bac] &= [ac]b-[bc]ab-bc[ba]-b[ac]-ba[bc]-[ba]cb, \\
\partial[cba] &= [ac] + c[ba]+cba[cb]+cb[ac]bc-[cb]abc-[ba]c.
\end{align*}

\begin{figure}[ht]
\def\svgwidth{\linewidth}
\centerline{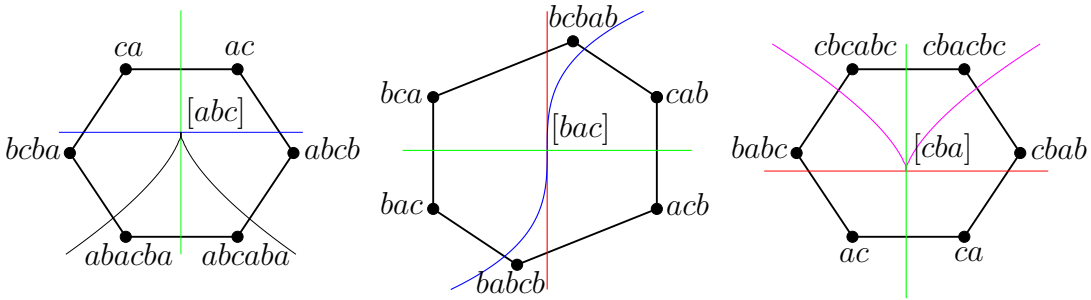}
\caption{The cells $c_{[abc]}$, $c_{[bac]}$ and $c_{[cba]}$.}
\label{fig:cD3}
\end{figure}
\bigskip

In each figure, we consider the family of paths constructed
in Section \ref{sect:transversal}
so that the functions $m_j$ are polynomials
in the real parameters $x_1$ and $x_2$ and in the variable $t$.
The curves in the figure
indicate regions for which the itinerary has positive dimension
and therefore separate open regions
for which the itinerary has dimension $0$.
These curves are therefore semialgebraic.
For instance, for $[abc]$ we have
\[ m_1(t) = \frac{t^3}{6} + x_2 t + x_1, \qquad
m_2(t) = \frac{t^2}{2}+ x_2, \qquad
m_3(t) = -t. \]
The vertical line, corresponding to $[ac]$, has equation
\[ \resultant(m_1,m_3;t) = x_1 = 0. \]
More precisely, the positive semiaxis corresponds to itinerary $[ac]$
and the negative semiaxis corresponds to itinerary $ab[ac]ba$.
The horizontal line, corresponding to $[bc]$, has equation
\[ \discriminant(m_2;t) = -2 x_2 = 0: \]
the positive semiaxis corresponds to itinerary $a[bc]$
and the negative semiaxis to $[bc]a$.
Finally, the cusp-like curve in the figure, corresponding to $[ab]$,
has equation 
\[ \resultant(m_1,m_2;t) = \frac{x_2^3}{9} + \frac{x_1^2}{8} = 0; \]
in the third quadrant the itinerary is $[ab]cba$
and in the fourth quadrant it is $abc[ab]$.

In $\cD_4$ there exist faces (i.e., cells of dimension $2$)
similar to those above
(such as $[bcd]$, which is similar to $[abc]$)
but also a few genuinely new ones:
$[abd]$, $[acd]$, $[adc]$ and $[bad]$,
all shown in Figure \ref{fig:cD4}; more generally, we have
\begin{align*}
\partial[aba_k] &=
ab[aa_k] + a[ba_k]a + [aa_k]ba + a_k[ab] - [ba_k] - [ab]a_k, \quad
k \ge 4; \\
\partial[aa_ka_{k+1}] &=
a[a_ka_{k+1}] + [aa_{k+1}] - [a_ka_{k+1}]a \\
& \qquad - a_ka_{k+1}[aa_k] - a_k[aa_{k+1}]a_k - [aa_k]a_{k+1}a_k, \quad
k \ge 3; \\
\partial[aa_{k+1}a_k] &=
a[a_{k+1}a_k] + [aa_{k+1}]a_ka_{k+1} + a_{k+1}[aa_k]a_{k+1} \\
& \qquad + a_{k+1}a_k[aa_{k+1}] - [a_{k+1}a_k]a - [aa_k], \quad
k \ge 3; \\
\partial[baa_k] &=
[aa_k] + a_k[ba] - [ba_k]ab - b[aa_k]b - ba[ba_k] - [ba]a_k, \quad
k \ge 4.
\end{align*}
The only genuinely new face in $\cD_5$ is $[a_1a_3a_5]$,
shown in Figure \ref{fig:cD5}; we have
\[ \partial[aa_ka_l] =
a[a_ka_l] + [aa_l]a_k + a_l[aa_k] - [a_ka_l]a - a_k[aa_l] - [aa_k]a_l, \quad
3 \le k < l-1. \]

\begin{figure}[ht]
\def\svgwidth{10cm}
\centerline{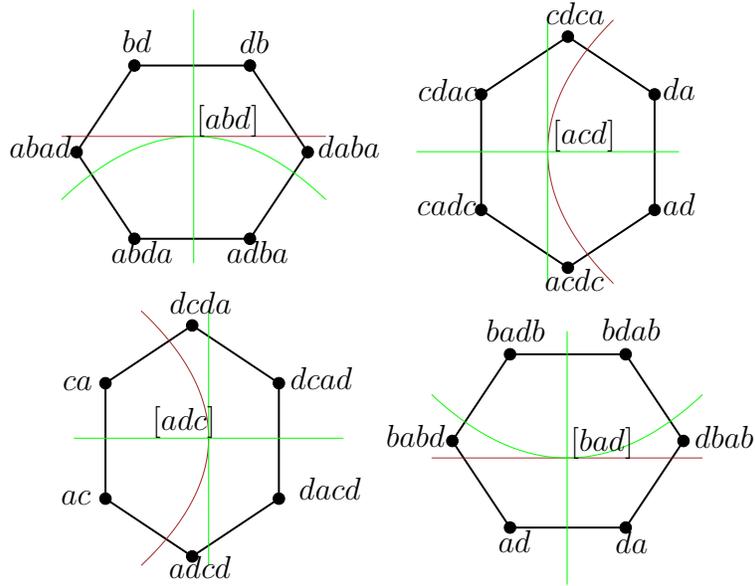}
\caption{The cells $c_{[abd]}$, $c_{[acd]}$, $c_{[adc]}$ and $c_{[bad]}$}
\label{fig:cD4}
\end{figure}

\begin{figure}[ht]
\def\svgwidth{6cm}
\centerline{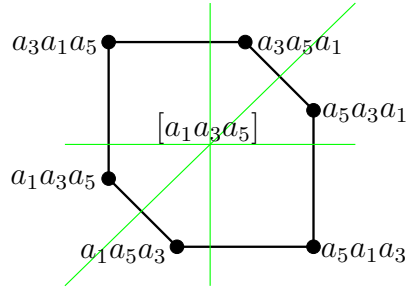}
\caption{The cell $c_{[a_1a_3a_5]}$.}
\label{fig:cD5}
\end{figure}

\section{Final Remarks}
\label{sect:finalremarks}

The present paper casts the foundations 
of a combinatorial approach.
We expect to cover in a forthcoming paper 
\cite{Alves-Goulart-Saldanha-Shapiro} some of the results 
stated in this section as conjectures. 

One important construction in \cite{Shapiro-Shapiro} 
and \cite{Saldanha3} is the add-loop procedure, 
which, in certain cases, is used to loosen up 
compact families of nondegenerate curves
through a homotopy in $\cL_n$. 
The resulting families of curly curves are then maleable:
if a homotopy exists in the space of immersions, another
homotopy exists in the space of locally convex curves. 
In \cite{Saldanha3}, for instance, open dense subsets 
$\mathcal{Y}_{\pm}\subset\cL_2(\pm 1)$ are shown to be 
homotopy equivalent to the space of loops $\Omega\Ss^3$.
Thus, certain questions regarding homotopies can be 
transferred to the space of continuous paths in the 
group $\Spin_{n+1}$.
This approach is reminiscent of classical constructions such as 
Thurston's eversion of the sphere by corrugations 
\cite{Levy-Maxwell-Munzner} 
and the proof of Hirsch-Smale Theorem \cite{Hirsch, Smale}. 
It can be considered as an elementary instance of the h-principle 
of Eliashberg and Gromov \cite{Eliashberg-Mishachev, Gromov}.

We expect to give a combinatorial version of this method which 
generalizes the result above to higher dimensions. 
This would imply the following results.

Recall that the center $Z(\SO_{n+1})$ of $\SO_{n+1}$ equals
$\{I\}$ if $n$ is even and $\{\pm I\}$ if $n$ is odd
(here $I$ is the identity matrix).
We have $Z(\Quat_{n+1}) = \Pi^{-1}[Z(\SO_{n+1})]$.

\begin{conj}
\label{conj:center}
If $q \in \Quat_{n+1}\smallsetminus Z(\Quat_{n+1})$ then
the inclusion $i_q: \cL_n(q) \to \Omega\Spin_{n+1}$
is a weak homotopy equivalence.
\end{conj}

In particular, there are at most two or four homotopy types 
distinct from that of $\Omega\Spin_{n+1}$, depending on 
the parity of $n$.

Recall that $\Quat_3=\operatorname{Q}_8
=\{\pm1, \pm\mathbf{i}, \pm\mathbf{j}, \pm\mathbf{k}\}$ 
and $\Quat_4=\operatorname{Q}_8\times\operatorname{Q}_8$. 
The main result in \cite{Saldanha3} classifies the spaces $\cL_2(q)$, 
$q\in\Quat_3$, into the following three weak homotopy types: 
for $q\neq\pm 1$, 
consistently with Conjecture \ref{conj:center}, we have 
$\cL_2(q)\approx\Omega\Ss^3$; otherwise, 
\begin{equation}
\label{equation:L2}
\cL_2(-1)\approx\Omega\Ss^3 \vee \Ss^0 \vee \Ss^4 \vee \Ss^{8} \vee \cdots, \qquad 
\cL_2(1)\approx\Omega\Ss^3 \vee \Ss^2 \vee \Ss^6 \vee \Ss^{10} \vee \cdots .
\end{equation}
The following conjecture promises a similar result for $n=3$
(see \cite{Alves-Saldanha} for more on these spaces).

\begin{conj}
\label{conj:L3}
We have the following weak homotopy equivalences:
\begin{align*}
\cL_3((+1,+1)) &\approx 
\Omega(\Ss^3 \times \Ss^3) \vee \Ss^4 \vee \Ss^8 \vee \Ss^8
\vee \Ss^{12} \vee \Ss^{12} \vee \Ss^{12} \vee \cdots, \\
\cL_3((-1,-1)) &\approx 
\Omega(\Ss^3 \times \Ss^3) \vee \Ss^2 \vee \Ss^6 \vee \Ss^6
\vee \Ss^{10} \vee \Ss^{10} \vee \Ss^{10} \vee \cdots, \\
\cL_3((+1,-1)) &\approx 
\Omega(\Ss^3 \times \Ss^3) \vee \Ss^0 \vee \Ss^4 \vee \Ss^4
\vee \Ss^{8} \vee \Ss^{8} \vee \Ss^{8} \vee \cdots, \\
\cL_3((-1,+1)) &\approx 
\Omega(\Ss^3 \times \Ss^3) \vee \Ss^2 \vee \Ss^6 \vee \Ss^6
\vee \Ss^{10} \vee \Ss^{10} \vee \Ss^{10} \vee \cdots.
\end{align*}
The above bouquets include one copy of $\Ss^k$,
two copies of $\Ss^{(k+4)}$, \dots, $j+1$ copies of $\Ss^{(k+4j)}$, \dots,
and so on.
\end{conj}

From the $1$ and $2$-skeletons of $\cD_n$, constructed 
in Sections \ref{sect:Dn1} and \ref{sect:Dn2}, we also expect 
to prove the following conjecture.

\begin{conj}
\label{conj:simplyconnected}
Let $z \in \Quat_{n+1}$. 
Every connected component of $\cL_n(z)$ is simply connected.
\end{conj}


Another interesting aspect of the subject is its relation with 
the theory of differential operators, mentioned in the Introduction. 
A linear differential operator can be canonically associated 
to a nondegenerate curve $\gamma:[0,1]\to\Ss^n$ 
(see \cite{Khesin-Ovsienko}).
There is a Poisson structure in the space of the differential operators, given by the Adler-Gelfand-Dickey bracket 
\cite{Adler, Gelfand-Dickey1, Gelfand-Dickey2}. 
The identification above relates the spaces $\cL_n(z)$ with 
symplectic leaves of this strucuture \cite{Khesin-Ovsienko, Khesin-Shapiro1}. 
Notice that the spheres appearing in the bouquets in Equation 
\ref{equation:L2} and Conjecture \ref{conj:L3} are all even-dimensional.
We wonder whether this is fortuitous or a manifestation of this symplectic structure; this question is maybe worth clarification.

\appendix
\section{Convex curves}
\label{appendix:convex}

A smooth parametric curve $\gamma:J\to\Ss^{n}$ 
defined on a compact interval $J\subset\RR$ is said 
to be \emph{strictly convex} if for each nonzero linear functional 
$\omega\in(\RR^{n+1})^{\ast}\smallsetminus\{0\}$ 
the function $\omega\gamma:J\to\RR$ has 
at most $n$ zeroes counted with multiplicities 
(zeroes at endpoints taken into account).  
It is said to be \emph{convex} if its restriction to any proper compact 
subinterval of $J$ is strictly convex.

In other words, a convex curve is one that 
(possibly neglecting one endpoint at a time) 
intersects each $n$-dimensional vector subspace 
$H^{n}\subset\RR^{n+1}$ at most $n$ times 
with multiplicities taken into account. 
Thus, for instance,
a transversal intersection counts as $1$; a tangency counts as $2$; 
an osculation counts as $3$.
Other terms used for the same or closely related 
concepts are \emph{non-oscillatory curves} \cite{Novikov-Yakovenko, Shapiro} 
and \emph{disconjugate curves} \cite{Khesin-Shapiro2, Shapiro-Shapiro, Shapiro-Shapiro3}.

The goal of this appendix 
is to show that a smooth nondegenerate curve 
$\gamma:[0,1]\to\Ss^n$ with initial frame 
$\Frenet{\gamma}(0)=1$ is convex 
if and only if its itinerary is the empty word, \textit{i.e.}, 
that the notion of convexity introduced in 
Section \ref{sect:chopadvance} and given in terms of the
singular set of $\Frenet{\gamma}$ coincides with this
geometric definition.
These results are essentially present in \cite{Shapiro}
but we feel that it may help the reader to have
a mostly self-contained presentation.


Clearly, convexity implies nondegeneracy. 
Conversely, as we shall see in Proposition \ref{prop:convex},
(smooth) nondegeneracy implies local convexity:
this is why the terms \emph{nondegenerate} and 
\emph{locally convex} are used interchangeably. 


Now we present the main result of this Appendix \ref{appendix:convex}. 
For $J$ a compact interval,
we say that a locally convex curve $\Gamma:J\to\Spin_{n+1}$ 
is \emph{short} if there exists $z \in \Spin_{n+1}$ such that
$\Gamma[J]\subset\cU_z$. 
Recall that $\cU_z \subset \Spin_{n+1}$ is the domain
of a triangular system of coordinates (see Section \ref{sect:triangle}).

\begin{prop}
\label{prop:convex}
Let $\gamma:J\to\Ss^n$ be a smooth nondegenerate curve 
defined on a compact interval $J\subset\RR$. 
The following conditions are equivalent:
\begin{enumerate}
\item\label{item:strictlyconvex}{$\gamma$ is strictly convex;}
\item\label{item:short}{$\Frenet{\gamma}$ is short;}
\item\label{item:BruhatAT}{$\forall t_{0},t_{+}\in J 
\left((t_{0}<t_{+})\rightarrow 
(\Frenet{\gamma}(t_{+})\in \Frenet{\gamma}(t_{0}) \Bru_{\acute\eta})\right)$;}
\item\label{item:BruhatA}{$\forall t_{0},t_{-}\in J 
\left((t_{-}<t_{0})\rightarrow 
(\Frenet{\gamma}(t_{-})\in \Frenet{\gamma}(t_{0}) \Bru_{\grave\eta})\right)$.}
\end{enumerate}
\end{prop}

Closely related sufficient conditions for convexity may be 
found in \cite{Novikov-Yakovenko, Shapiro}. 

\begin{example}
\label{example:convex}
Given $z\in\Bru_{\acute\eta}$, consider the curve 
$\Gamma=\Gamma_z:[0,1]\to\Spin_{n+1}$ 
passing through $\Gamma_z(\frac12)=z$ given by 
Lemma \ref{lemma:convex1}. 
It follows immediately from Proposition \ref{prop:convex} 
that $\gamma_z=\Gamma_z e_1$ is a convex curve
(though not strictly convex). 
For an alternative proof, recall that $\Gamma_z$ 
is obtained from $\Gamma_{\acute\eta}(t)=\exp(\pi t\fh)$ 
through a projective transformation and see Lemma 2.2 of 
\cite{Saldanha-Shapiro} for a direct proof of the 
convexity of $\Gamma_{\acute\eta}$.
For $n=2$, 
$\gamma_{\acute\eta}(t)=
\frac{1}{2}(1+\cos(2\pi t),\sqrt{2}\sin(2\pi t),1-\cos(2\pi t))$ 
is the circle of diameter $e_{1}e_{3}$ in $\Ss^2$. 
Notice that $\gamma_{\acute\eta}$ is closed 
if and only if $n$ is even
(as usual, a curve
$\gamma:[0,1]\to\Ss^n$ is closed if
$\Frenet{\gamma}(0)=\Frenet{\gamma}(1)$).
\end{example}

The well known fact that there are no 
closed convex curves in $\Ss^n$ for $n$ odd 
also follows as an easy consequence of 
Proposition \ref{prop:convex}. 
Of course, the projectivization
$[\gamma_{\acute\eta}]: [0,1] \to \mathbb{RP}^{n}$
(where $[\gamma_{\acute\eta}](t) =
\RR\gamma_{\acute\eta}(t)\in\mathbb{RP}^{n}$)
is a closed convex curve for arbitrary $n>1$. 
In \cite{Anisov} it is shown that 
the space of closed convex curves in $\mathbb{RP}^{n}$ 
is a contractible connected component of 
the space of closed, locally convex curves in $\mathbb{RP}^{n}$.
In the same spirit, we have Lemma \ref{lemma:convex2}.

In order to prove proposition \ref{prop:convex} we shall 
need a couple of technical results.

In this section, $\Lo_{n+1}$ (respectivelly, $\Up_{n+1}$) stands for 
the group of invertible real lower (resp., upper) triangular matrices of order $n+1$.
Given a real square matrix $M$ of order $n+1$, 
we call a factorization of the form
$M=LU$, $L\in\Lo_{n+1}$, $U\in\Up_{n+1}$ 
an \emph{$LU$ decomposition} of $M$.
Recall that a necessary and sufficient condition for 
$M=(M_{ij})$ having an $LU$ decomposition is the 
nonvanishing of all of its northwest minor determinants 
\[\Lambda^k(M)_{\{1,2,\cdots,k\},\{1,2,\cdots,k\}}
=\det\left(M_{i, j}\right)_{1 \leq i,j\leq k}.\]

\begin{lemma}
\label{lemma:LU}
Let $J\subseteq\RR$ be an interval and $\phi:J\to\RR^{n+1}$ 
be a smooth  map. If the matrix 
$W_\phi(t)=\left(\phi(t),\phi'(t),\cdots,\phi^{(n)}(t)\right)$ 
admits an $LU$ decomposition for all $t\in J$, 
then, given sequences $\mathbf t$ of real numbers 
$\{t_{0}<t_{1}<\cdots<t_{k}\}\subset J$
and $\mathbf m$ of positive integers $m_{0},m_{1},\cdots,m_{k}$ 
whose sum is $n+1$, the matrix 
\[W_\phi^{\mathbf t, \mathbf m}=
\left(\phi(t_{0}),\phi'(t_{0}),\cdots,\phi^{(m_{0}-1)}(t_{0}),
\cdots,\phi(t_{k}),\phi'(t_{k}),\cdots,\phi^{(m_{k}-1)}(t_{k})\right)\]
admits an $LU$ decomposition as well and, 
in particular, $\det W_\phi^{\mathbf t,\mathbf m}\neq 0$.
\end{lemma}

In what follows, we only need the nonvanishing of 
$\det W_\phi^{\mathbf t, \mathbf m}$ when 
$W_\phi$ has an $LU$ decomposition. 
In fact, a smooth curve $\gamma:J\to\RR^{n+1}$ 
is strictly convex if and only if 
$\det W_\gamma^{\mathbf t, \mathbf m}\neq 0$ 
for all increasing sequences $t_{0}<t_{1}<\cdots<t_{k}$ 
in the compact interval $J$ and all positive integers 
$m_{0},m_{1},\cdots,m_{k}$ whose sum is $n+1$. 
The stronger way the conclusion of Lemma \ref{lemma:LU} is stated 
is convenient for the induction argument. 
Lemma \ref{lemma:LU} can be considered a reformulation of
Theorem \textrm{V} of \cite{Polya};
see also \cite{Shapiro-Shapiro3}, \cite{Shapiro} and
\cite{Novikov-Yakovenko}.

\begin{proof}
Write $\phi(t)=\sum_{0\le j\le n}\phi_{j}(t)e_{j+1}$ 
and $W_\phi(t)=L(t)U(t)$, with $L(t)\in\Lo_{n+1}$ 
and $U(t)\in\Up_{n+1}$. We proceed by induction in $n$. 
The cases $n\in\{0,1\}$ are trivial. 
There is no loss of generality in assuming 
$\phi_{0}$ constant equal to $1$, since $\phi_{0}(t)\neq 0$ 
for all $t$ and we can solve 
$e_{1}^{\transpose}W_\phi(t)V(t)=e_{1}^{\transpose}$ 
explicitly for $V(t)\in\Up_{n+1}$ obtaining a matrix 
whose entries are rational functions of 
$\phi_0(t),\phi'_0(t),\cdots,\phi^{(n)}_0(t)$ 
with powers of $\phi_0(t)$ as denominators. 
Besides, given a sequence of instants 
$t_{0}<t_{1}<\cdots<t_{k}$ and positive integers 
$m_{0},m_{1},\cdots,m_{k}$ as above, 
we only need to prove that 
$\det W_\phi^{\mathbf t, \mathbf m}\neq 0$, 
since the northwest minor determinants of order $k\leq n$ of
$W_\phi^{\mathbf t, \mathbf m}$ are all nonvanishing 
by the induction hypothesis.
Suppose, to the contrary, that there is 
$\omega\in(\RR^{n+1})^{\ast}\smallsetminus\{0\}$ 
such that for all $j\in\{0,1,\cdots,k\}$ one has 
$\omega\cdot\phi(t_{j})=\omega\cdot\phi'(t_{j})
=\cdots=\omega\cdot\phi^{(m_{j}-1)}(t_{j})=0$,
and define $\widehat\phi:\RR\to\RR^{n}$ by 
$\widehat\phi(t)=\sum_{1\le j\le n}\phi'_{j}(t)e_{j}$. 
Note that $W_{\widehat{\phi}}(t)$ is the $n\times n$ 
southeast block of $W_\phi(t)$ and therefore 
admits an $LU$ decomposition as well. 
Consider now the restriction 
$\widehat\omega=\omega|_{\RR^{n}}$, 
where we make the identification 
$\RR^{n}=\{0\}\times\RR^{n}\subset\RR^{n+1}$. 
By Rolle's Theorem, there is an instant 
$t_{j+1/2}\in (t_{j},t_{j+1})$ for each 
$j\in\{0,1,\cdots,k-1\}$ when 
$\widehat\omega\cdot\widehat\phi(t_{j+1/2})
=\omega\cdot\phi'(t_{j+1/2})=0$. 
Besides, for all $j\in\{0,1,\cdots,k\}$, 
$\widehat\omega$ inherits from $\omega$ 
the zeroes $\widehat\omega\cdot\widehat\phi (t_{j})
=\widehat\omega\cdot\widehat\phi'(t_{j})=\cdots
=\widehat\omega\cdot\widehat\phi^{(m_{j}-2)}(t_{j})=0$. 
Now, take $\widehat{k}=2k$ and define 
the refined sequence 
$\widehat{t}_{0}<\widehat{t}_{1}<\cdots<\widehat{t}_{\widehat{k}}$ 
by $\widehat{t}_{j}=t_{j/2}$ with associated multiplicities 
\[\widehat{m}_{j}=\begin{cases} m_{j/2}-1 \text{ , if $j$ is even} \\ 1 \text{ , if $j$ is odd} \end{cases}.\]
Then, $\det W_{\widehat\phi}^{\mathbf{\widehat{t}},\mathbf{\widehat{m}}}=0$, 
what contradicts our induction hypothesis.
\end{proof}

The following result amounts to the intuitive fact that 
the distance between two continuously moving 
compact subsets of a metric space is a 
continuous function of time.

\begin{lemma}
\label{lemma:distance}
Let $(M,d)$ be a metric space and $K,L$ be 
compact topological spaces. 
Let $F:\RR\times K\to M$ and $G:\RR\times L\to M$ 
be continuous maps. 
The function $f:\RR\to\RR$ defined by 
$f(s)=d(F_{s}[K],G_{s}[L])$ is continuous.
\end{lemma}
\begin{proof}
Let $X,Y$ be topological spaces, $Y$ compact, 
and $\Phi:X\times Y\to\RR$ a continuous map. 
It is an easy exercise in point-set topology to prove that
the map $\phi:X\to\RR$ given by 
$\phi(x)=\inf\{\Phi(x,y)\,|\,y\in Y\}$ is well-defined and continuous.
Our result then follows.
\end{proof}




\begin{proof}[Proof of Proposition \ref{prop:convex}]
In this proof we write
$\Frenet{\gamma}(t_0;t_1)
=(\Frenet{\gamma}(t_0))^{-1} \Frenet{\gamma}(t_1)$. 
Also, for each $Q\in\SO_{n+1}$, 
we will denote by $B_{Q}\in B^+_{n+1}$ the unique 
signed permutation matrix such that $Q\in\Bru_{B_{M}}$. 

\smallskip
\noindent 
$(\ref{item:short}\rightarrow\ref{item:strictlyconvex})$ 
Rotations preserve both nondegeneracy and strict convexity, 
\textit{i.e.}, for all $Q\in\SO_{n+1}$, if the smooth curve 
$\gamma:J\to\RR^{n+1}$ is strictly convex 
(respectively, nondegenerate) then so is 
$Q^{\transpose}\gamma$. 
Take $Q\in\SO_{n+1}$ such that 
$\Frenet{\gamma}[J]\subset\cU_Q$ and apply 
Lemma \ref{lemma:LU} to $\phi=Q^{\transpose}\gamma$.

\smallskip
\noindent 
$(\ref{item:strictlyconvex}\rightarrow\ref{item:BruhatAT})$ 
Suppose that $\gamma$ violates condition 
(\ref{item:BruhatAT}) for certain $t_{0}<t_{+}$. 
We know from Lemma \ref{lemma:chopadvance} that 
$\Frenet{\gamma}(t_{0};t)\in\Bru_{\acute\eta}$ 
for sufficiently small $t-t_{0}>0$. 
Assume without loss that we have taken 
$t_{+}>t_{0}$ minimal such that 
$\Frenet{\gamma}(t_{0};t_{+})\notin\Bru_{\acute\eta}$. 
Thus, $\Frenet{\gamma}(t_{0};t_{+})$ will exhibit 
a first non-invertible southwestblock of order $k+1$ 
for some $k\in\{0,1,\cdots, n-1\}$ and the 
corresponding block in the signed permutation matrix
$B(t_{0};t_{+})=B_{\Frenet{\gamma}(t_{0};t_{+})}$ 
will have the form 
\[\begin{pmatrix} & & & & 0 \\ & & & (-1)^{k+1} & \\ & & \iddots & & \\ & -1 & & & \\ 1 & & & & \end{pmatrix}.\]
This means that the $(k+1)^{\text{th}}$ columns of 
$B(t_{0};t_{+})$ is one of the previous vectors of the 
standard basis of $\RR^{n+1}$, \textit{i.e.}, 
some $e_{j}$ with $j\in\{1,2,\cdots,n-k\}$. 
Explicitly, one has $B(t_{0};t_{+}) e_{k+1}=e_{j}$, 
which boils down to the following relation 
between the Wronski matrices of $\gamma$ 
in $t_{0}$ and $t_{+}$: 
\[W_\gamma(t_{+})U_{+} e_{k+1}=W_\gamma(t_{0})U_{0} e_{j}\]
for some $U_{0},U_{+}\in\Up^+_{n+1}$ and $j\in\{1,2,\cdots,n-k\}$. 
That means a linear dependence between the first 
$k+1$ columns of $W_\gamma(t_{+})$ and the first $n-k$ 
columns of $W_\gamma(t_{0})$:
\[\sum_{1\le i_0\le n-k}(U_{0})_{i_0, j}\gamma^{(i_0-1)}(t_{0})
+\sum_{1\le i_+\le k+1}(U_{+})_{i_+,k+1}\gamma^{(i_+ -1)}(t_{+})=0.\]
Notice that the coefficient of $\gamma^{(k)}(t_{+})$ 
is $(U_{+})_{k+1,k+1}>0$. In the notation of 
lemma \ref{lemma:LU}, we have 
$\det W_\gamma^{\mathbf t, \mathbf m}=0$ 
for $\mathbf{t}=(t_{0},t_{+})$ and $\mathbf{m}=(n-k,k+1)$, 
thus ruling out strict convexity for the proper subarc 
$\gamma|_{[t_{0},t_{+}]}$.

\smallskip
\noindent
$(\ref{item:BruhatA}\leftrightarrow\ref{item:BruhatAT})$ 
Notice that for all $t_{a},t_{b}\in J$ and all 
$U_{0},U_{1}\in\Up^+_{n+1}$ we have 
$\Frenet{\gamma}(t_{a};t_{b})=U_{0}\Pi(\acute\eta)U_{1}$ 
iff $\Frenet{\gamma}(t_{b};t_{a})=U^{-1}_{1}\Pi(\acute\eta)^{-1}U^{-1}_{0}=U^{-1}_{1}\Pi(\grave\eta)U^{-1}_{0}$.

\smallskip
\noindent
$(\ref{item:BruhatAT}\rightarrow\ref{item:short})$ 
Take $\Gamma = \Frenet{\gamma}$ and $J = [t_0,t_1]$.
From (\ref{item:BruhatAT}),
$\Gamma[(t_{0},t_1]]
\subset \Gamma(t_0)\Bru_{\acute\eta} = \cU_{\Gamma(t_0)\acute\eta}$. 
Take $\epsilon>0$ and a smooth nondegenerate extension 
$\tilde{\Gamma}:[t_{0}-\epsilon,t_{1}]\to\Ss^n$ of $\Gamma$
(i.e., $\tilde{\Gamma}|_J=\Gamma$).
We claim that
$\Gamma[[t_{0},t_1]] \subset \cU_{\tilde\Gamma(s)\acute\eta}$ 
provided $s \in [t_0-\epsilon,t_0)$ is sufficiently near $t_0$.
By taking $\epsilon$ sufficiently small we may assume that
$\tilde\Gamma|_{[t_{0}-\epsilon,t_{0}+\epsilon]}$ is short. 
By implications $\ref{item:short}
\rightarrow\ref{item:strictlyconvex}
\rightarrow\ref{item:BruhatAT}$, we have 
$\tilde\Gamma[[t_0,t_{0}+\epsilon]] \subset \cU_{\tilde\Gamma(s)\acute\eta}$ 
for all $s\in[t_{0}-\epsilon,t_{0})$. 
Now define $f: [t_0-\epsilon,t_0] \to \RR$ by
$f(s) = d(\Gamma[[t_{0}+\epsilon,t_{1}]], 
\Spin_{n+1} \smallsetminus\, \cU_{\tilde\Gamma(s)\acute\eta})$
(where $d$ is the distance in $\Spin_{n+1}$).
By Lemma \ref{lemma:distance}, $f$ is a continuous function 
and $f(t_0) > 0$.
Thus, $f(s) > 0$ for $s<t_{0}$ sufficiently near $t_{0}$ 
and we have 
$\Gamma[[t_{0}+\epsilon,t_{1}]] \subset  \cU_{\tilde\Gamma(s)\acute\eta}$,
completing the proof.
\end{proof}

\section{Hilbert manifolds of curves}
\label{appendix:Hilbert}

In this appendix we introduce a definition 
of the spaces $\cL_{n}(z_0;z_1)$ that has 
some technical advantages as compared 
to more straightforward ones. 
Among its nice features is a smooth Hilbert 
manifold structure, which has enabled 
some developments reminiscent of 
standard differential topology ones. 
Also, in many arguments we have allowed 
for discontinuities in the derivatives of our 
locally convex curves. More explicitly, we 
have many times concatenated locally 
convex arcs $\Gamma_1$ and $\Gamma_2$ 
on $\Spin_{n+1}$ regardless of the 
differentiability of the resulting path 
$\Gamma_{1}\ast\Gamma_{2}$ at the 
welding point, and considered it nonetheless 
as a locally convex curve, omitting thereby 
a tedious smoothening out argument. 
This appendix therefore plays the the same role
as Section 2 in \cite{Saldanha3},
Section 1 in either \cite{Saldanha-Zuhlke1} or \cite{Saldanha-Zuhlke2}.

Recall from Section \ref{sect:triangle} that 
a map $\Gamma:J\to\Spin_{n+1}$ defined 
on an interval $J\subseteq\RR$ is said to be 
\emph{locally convex} if it is absolutely continuous 
with logarithmic derivative 
$\Lambda_\Gamma(t)=(\Gamma(t))^{-1}\Gamma'(t)$ 
given almost everywhere by a positive linear 
combination of the skew-symmetric matrices  
$\fa_k\in\so_{n+1}$ 
(defined in Equation \ref{equation:fa}): 
\[\Lambda_\Gamma(t)=\sum_{k\in\nmesmo}\kappa_k(t)\fa_k,\quad (\fa_k)_{i,j}=
[k=j=(i-1)]-[k=i=(j-1)], \quad\kappa_k(t)>0.\]
We shall call the functions 
(defined on $J$ up to a set of measure zero) 
$\kappa_1,\cdots,\kappa_n$ the \emph{generalized curvatures} 
of $\Gamma$. 
In fact, the classical Frenet-Serret formulae say that 
a smooth parametric curve 
$\gamma:J\to\Ss^{n}\subset\RR^{n+1}$ 
is nondegenerate (in the sense of the Introduction) 
if and only if the logarithmic derivative 
$\Lambda_{\gamma}=\Lambda_{\Frenet{\gamma}}$ 
of its Frenet lift $\Frenet{\gamma}$ is of the form above 
with smooth functions   
$\kappa_k:J\to(0,+\infty)$ (defined everywhere). 
Notice that in the smooth case we have 
$\kappa_1=v_\gamma=|\gamma'|$, 
the velocity of $\gamma$; 
$\kappa_2=v_\gamma\varkappa_1$, 
where $\varkappa_1$ is the geodesic curvature of $\gamma$; 
$\kappa_3=v_\gamma\varkappa_2$, 
where $\varkappa_2$ is the geodesic torsion of $\gamma$, 
and so on (see \cite{Klingenberg1, Novikov-Yakovenko}). 
A natural relaxation of the notion of local convexity 
is therefore to allow for curves with less regular 
generalized curvatures. 
The definition of Section \ref{sect:triangle} is too wide, though: 
we stick to a compromise that will allow 
enough freedom in the construction of locally 
convex curves, plus a Hilbert manifold structure. This 
manifold structure is the key to showing that 
different spaces are equivalent  
(see Proposition \ref{prop:spaces} below). 
Our approach is reminiscent of the construction 
of a Hilbert manifold $H^1(J,M)$ of absolutely 
continuous curves in a compact Riemannian manifold $M$ (see \cite{Klingenberg2}).

Let $J\subseteq\RR$ be an interval. 
A measurable positive function 
$\kappa:J\to(0,+\infty)$ will be called 
\emph{$L^2$-admissible} if it satisfies 
$\int_{t_{0}}^{t_{1}}(\kappa(t))^{2}dt<\infty$ and 
$\int_{t_{0}}^{t_{1}}(\kappa(t))^{-2}dt<\infty$ 
for all $t_{0}<t_{1}$ in $J$. 
Let us denote by $\mathcal{K}_{J}$ the set of 
admissible functions defined on the interval $J$. 
A locally convex curve $\Gamma:J\to\Spin_{n+1}$ 
(in the sense of Section \ref{sect:triangle}) 
is said to be \emph{$L^2$-admissible} if its 
generalized curvatures $\kappa_1, \ldots, \kappa_n$ 
are all $L^2$-admissible functions 
(i.e., $\kappa_1,\ldots,\kappa_n\in\mathcal{K}_{J}$). 
Also, an absolutely continuous curve  
$\gamma:J\to\Ss^n$ is said to be 
\emph{$L^2$-nondegenerate} if there exists an 
$L^2$-admissible locally convex curve 
$\Gamma:J\to\Spin_{n+1}$ such that $\gamma=\Gamma e_1$. 
In this case, we set $\Frenet{\gamma}=\Gamma$ and 
call the admissible functions 
$\kappa_1,\ldots,\kappa_n\in\mathcal{K}_J$ 
the \emph{generalized curvatures} of $\gamma$.

A theorem of Carath\'eodory's guarantees existence 
and uniqueness for the solution (in an extended sense) 
of the Frenet-Serret IVP 
\begin{equation}
\label{equation:ivp}
\Gamma'(t)=\Gamma(t)\sum_{k\in\nmesmo}\kappa_k(t)\fa_k, \quad \Gamma(t_0)=z_0\in\Spin_{n+1}
\end{equation}
as long as the functions $\kappa_k$
are all Lebesgue integrable on each compact subinterval 
of their common domain (see \cite{Filippov}). 
In this case, the coordinate-functions $(\Gamma(t))_{ij}$ of the 
unique solution $\Gamma$ are all 
absolutely continuous and therefore 
differentiable almost everywhere with 
derivatives integrable on compact intervals. 
Also, given $t\in J$, the element $\Gamma(t)\in\Spin_{n+1}$ 
depends continuously on the functions $\kappa_{k}$ 
with respect to the $L^{1}$-norm 
in any compact interval containing $t_{0}$ and $t$. 
The somewhat more stringent $L^2$ hypotheses 
are meant to produce Hilbert manifolds of curves below.

Notice that $\Frenet{\gamma}$ and 
$\kappa_1,\cdots,\kappa_n$ 
are well-defined for each given  
$L^{2}$-nondegenerate spherical curve $\gamma$.
In fact, for all $t_{0},t\in J$ and all $j\in\nmaisum$, 
an $L^{2}$-admissible curve 
$\Gamma:J\to\Spin_{n+1}$ with logarithmic derivative 
$\Lambda=\sum_{k}\kappa_{k}\fa_k$ satisfies 
\begin{equation}
\label{equation:integral}
\Gamma_{j}(t) = \Gamma_{j}(t_{0}) +\int^{t}_{t_{0}}(\kappa_{j}(s)\Gamma_{j+1}(s)-\kappa_{j-1}(s)\Gamma_{j-1}(s))ds,
\end{equation}
where $\Gamma_{j}=\Gamma e_{j}$ for $j\in\nmaisum$ 
and $\Gamma_{j}\equiv 0$ and $\kappa_{j}\equiv 0$ otherwise. Therefore, if $\Gamma,\widehat{\Gamma}:J\to\Spin_{n+1}$ are both $L^{2}$-admissible satisfying 
$\Gamma_{j}=\widehat{\Gamma}_{j}$, 
$\Gamma_{j-1}=\widehat{\Gamma}_{j-1}$ and 
$\kappa_{j-1}=\widehat{\kappa}_{j-1}$ a.e. for some $j$, 
then $\Gamma'_{j}=\widehat{\Gamma}'_{j}$ a.e. 
tells us that 
$\kappa_{j}\Gamma_{j+1}=\widehat{\kappa}_{j}\widehat{\Gamma}_{j+1}$ a.e., 
and, since $\kappa_{j},\widehat{\kappa}_{j}>0$ a.e. 
and $|\Gamma_{j+1}|=|\widehat{\Gamma}_{j+1}|\equiv 1$, 
we have $\kappa_{j}=\widehat{\kappa}_{j}$ a.e. 
and $\Gamma_{j+1}=\widehat{\Gamma}_{j+1}$. 
By induction we see that if 
$\gamma=\Gamma_{1}=\widehat{\Gamma}_{1}$ with $\Gamma$ and $\widehat{\Gamma}$ $L^{2}$-admissible then $\Gamma=\widehat{\Gamma}$ and $\Lambda=\widehat{\Lambda}$ a.e..
In other words, once we fix the interval $J\subseteq\RR$ 
and the initial conditions  
$\Frenet{\gamma}(t_{0})=\Gamma(t_{0})=z_{0}\in\Spin_{n+1}$, 
there are natural and mutually compatible bijections 
$(\kappa_{1},\ldots,\kappa_{n})\mapsto\Frenet{\gamma}\mapsto\gamma$ between the set $\mathcal{K}_J^n$ 
of $n$-tuples of admissible functions, 
the set of $L^2$-admissible curves and 
the set of $L^2$-nondegenerate spherical curves. 
We make the identifications 
$\gamma\approx\Frenet{\gamma}\approx(\kappa_{1},\ldots,\kappa_{n})$ 
without further clarification.

\begin{rem}
\label{rem:Cr}
Notice that $\gamma$ of class $C^n$ implies 
$\kappa_1\in C^{n-1}$, $\kappa_2\in C^{n-2}$, 
\ldots, $\kappa_n\in C^0$. 
The less obvious converse follows from mixing up Equations \ref{equation:integral} and noticing that $\Gamma=\Frenet{\gamma}$ 
is necessarily of class $C^1$. 
Slightly more explicitly: we already know that $\Gamma_j\in C^1$ 
for all $j\in\nmaisum$. Use Equations \ref{equation:integral} 
to  show recursively that $\Gamma_j\in C^2$ for $j\in\nmesmo$. 
Use again Equations \ref{equation:integral} 
to  show recursively that $\Gamma_j\in C^3$ 
for $j\in\llbracket n-1\rrbracket$; 
repeat the procedure to obtain the desired result.
We spare the reader the rather tedious formalization of this argument.
\end{rem}

Henceforth, we fix $z_0\in\Spin_{n+1}$ and
consider the set of $L^{2}$-nondegenerate 
curves $\gamma:[0,1]\to\Ss^n$ satisfying 
$\Frenet{\gamma}(0)=z_0$. 
We can identify this set with $\mathcal{K}_{[0,1]}^{n}$. 
In order to turn it into a topological Hilbert manifold 
modeled on the separable Hilbert space 
$\mathbf{H}:=(L^{2}([0,1],\RR))^{n}$, we identify 
$(\kappa_{1},\ldots,\kappa_{n})\in\mathcal{K}_{[0,1]}^{n}$ 
with $(\xi_{1},\ldots,\xi_{n})\in\mathbf{H}$ via 
\begin{equation}
\label{equation:xis}
\xi_{j}=\kappa_{j}-\dfrac{1}{\kappa_{j}},\qquad \kappa_{j}=\dfrac{\xi_{j}+\sqrt{\xi_{j}^{2}+4}}{2}
\end{equation}
This choice of topological chart is rather arbitrary, 
but is simple enough and does the job: 
notice that $\int_{0}^{1}\xi^{2}_{j}(t)dt<\infty$ 
iff $\kappa_{j}\in\mathcal{K}_{[0,1]}$. 
Throughout the remaining of this appendix, 
we denote by $\cL_n^{[L^2]}(z_0;\cdot)$ the 
topological Hilbert manifold thus obtained 
(homeomorphic to Hilbert space $\mathbf{H}$ by construction). 
By contrast, we denote by $\cL_n^{[C^n]}(z_0;\cdot)$ 
the space of nondegenerate curves defined in the introduction, 
as a means of explicit reference to both the $C^{n}$ regularity required of its elements and its natural $C^{n}$-metric. 
Under the identification above, we have 
$\xi_1\in C^{n-1}$, $\xi_2\in C^{n-2}$, \ldots, 
$\xi_{n}\in C^{0}$ for 
$\gamma\in\cL_n^{[C^n]}(z_0;\cdot)$ and therefore 
the proper point-set inclusion 
$\cL_n^{[C^n]}(z_0;\cdot)\subset\cL_n^{[L^2]}(z_0;\cdot).$
Alternate metrics for $\cL_n^{[C^n]}(z_0;\cdot)$ are the metric 
induced by this inclusion map, and the product metric 
of the respective $C^r$-metrics in 
$\mathbf{B}=C^{n-1}([0,1],\RR)\times\cdots\times C^{0}([0,1],\RR)$, 
both under the identification $\Gamma\approx(\xi_1,\cdots,\xi_n)$. 
The former shall not be used in what follows, while the latter is 
readily seen to yield the same topology as the natural 
(not complete) $C^n$-metric inherited from $C^n([0,1],\RR^{n+1})$. 
In fact, the continuity of the map 
$\gamma\in\cL_n^{[C^n]}(z_0;\cdot)\mapsto(\xi_1,\cdots,\xi_n)\in\mathbf{B}$ 
follows from $\kappa_j=(\Frenet{\gamma}e_{j+1})\cdot(\Frenet{\gamma}'e_{j})$, 
via Equations \ref{equation:xis}, 
while the existence and continuity of its inverse is due to 
Remark \ref{rem:Cr} and the continuous dependence on parameters
for the IVP in Equation \ref{equation:ivp} 
(alternatively, one could use Equations \ref{equation:integral} 
to obtain the appropriate estimates).
From now on, we identify $\cL_n^{[C^n]}(z_0;\cdot)$ with 
the separable Banach space $\mathbf{B}$, 
just as we have identified $\cL_n^{[L^2]}(z_0;\cdot)$ and 
the separable Hilbert space $\mathbf{H}$. 
With respect to these metrics, the inclusion 
$i:\mathbf{B}\hookrightarrow\mathbf{H}$ 
is clearly continuous with dense image.

For future reference, we now quote the following two general results from 
the homotopy theory of infinite dimensional manifolds.

\begin{fact}[Theorem 2 of \cite{Burghelea-Saldanha-Tomei1}]
\label{fact:BST}
Let $\mathbf{B}_1$ and $\mathbf{B}_2$ be infinite dimensional separable Banach spaces. Suppose $i:\mathbf{B}_1\to \mathbf{B}_2$ is a bounded, injective linear map with dense image and $M_2\subset \mathbf{B}_2$ is a smooth closed Banach submanifold of finite codimension. Then, $M_1=i^{-1}[M_2]$ is a smooth closed Banach submanifold of $\mathbf{B}_1$ and $i:(\mathbf{B}_1,M_1)\to (\mathbf{B}_2,M_2)$ is a homotopy equivalence of pairs.
\end{fact}

\begin{fact}[from Theorem 0.1 of \cite{Burghelea-Henderson} and Corollary 1 of \cite{Henderson}]
\label{fact:BH}
Let $M_1$ and $M_2$ be topological manifolds modeled on infinite dimensional separable Banach spaces. Any homotopy equivalence $i:M_1\to M_2$ is homotopic to a homeomorphism.
\end{fact}

We now endow $\cL_n^{[L^2]}(z_0;\cdot)$ with a 
smooth differentiable structure. 
The previous identification $\cL_n^{[L^2]}(z_0;\cdot)\approx \mathbf H$ 
could be used to define a differentiable structure, but the one 
we are about to introduce is more convenient. 
They may well be equivalent, 
but we shall not try to prove this fact.
In any case, Corollary 2 of \cite{Henderson} implies that the resulting 
two smooth Hilbert manifolds are diffeomorphic.
Around each $\Gamma_0\in\cL_n^{[L^2]}(z_0;\cdot)$ we construct a 
coordinate system 
$(\cU_{(\mathbf t, \mathbf z)}, \phi_{(\mathbf t, \mathbf z)})$ with 
$\phi_{(\mathbf t, \mathbf z)}:\cU_{(\mathbf t, \mathbf z)}
\to\mathbf H^N$, 
$N\in\NN^\ast$, $\mathbf t=(t_1<\cdots<t_N)\in (0,1)^N$, 
$\mathbf z= (z_1,\cdots,z_N)\in\Spin_{n+1}^N$. 

Fix $\Gamma_0\in\cL_n^{[L^2]}(z_0;\cdot)$ and decompose it 
into a finite number of strictly convex arcs 
(equivalently, short arcs; see Proposition \ref{prop:convex}). 
Explicitly, choose times 
$0=t_0<t_1<t_2<\cdots<t_N<t_{N+1}=1$ and 
$\Gamma_0(0)=z_0,z_1, \ldots, z_N, 
z_{N+1}=\Gamma_0(1)\in\Spin_{n+1}$ such that 
$\Gamma_0[[t_{r},t_{r+1}]]\subset\cU_{z_r}$ for each  
$r\in\{0,1,\ldots,N\}$. 
Now, on each open subset $\cU_{z_r}\subset\Spin_{n+1}$ 
we use the triangular system of coordinates of Section \ref{sect:triangle} 
to define  
\[\Gamma_{L,r}(t)=z_r^{-1}\Gamma_0(t)R_r(t)\in\Lo^1_{n+1}, 
\quad R_r(t)\in\Up^{+}_{n+1}, \quad t_{r}\leq t\leq t_{r+1}.\]
The curve $\Gamma_{L,r}:[t_r,t_{r+1}]\to\Lo^1_{n+1}$ is 
locally convex in the sense of Section \ref{sect:triangle}, i.e., 
its logarithmic derivative is almost everywhere of the form 
\[(\Gamma_{L,r}(t))^{-1}\Gamma'_{L,r}(t)=
\sum_{k\in\nmesmo}\beta_{r,k}(t)\fl_k,\quad 
(\fl_k)_{i,j}=[k=j=i-1]\in\lo^1_{n+1}, \quad \beta_{r,k}(t)>0.\]
In fact, we have (see Lemma \ref{lemma:al}) 
\begin{equation}
\label{equation:betas} 
\beta_{r,k}(t)=\dfrac{R(t)_{k,k}}{R(t)_{k+1,k+1}}\kappa_k(t).
\end{equation} 
In particular, the coefficients $\beta_{r,k}$ are all admissible functions. 
Write as before 
\begin{equation}
\label{equation:chis}
\chi_{r,k}=\beta_{r,k}-\dfrac{1}{\beta_{r,k}},\quad
\beta_{r,k}=\dfrac{\chi_{r,k}+\sqrt{\chi_{r,k}^2+4}}{2}.
\end{equation} 
We therefore have $\chi_r=(\chi_{r,1},\ldots,\chi_{r,n})\in\mathbf H$. 
For such $\mathbf t=(t_1,\ldots,t_N)\in (0,1)^N$ and 
$\mathbf z=(z_1,\ldots,z_N)\in\Spin_{n+1}^N$ consider the set 
$\cU_{\mathbf t, \mathbf z}\subset\cL_n^{[L^2]}(z_0;\cdot)$ 
of the locally convex curves $\Gamma$ satisfying 
$\Gamma[[t_r,t_{r+1}]]\subset\cU_{z_r}$
for all $r\in\{0,1,\ldots,N\}$.
Now consider the map 
$\phi_{\mathbf{t},\mathbf{z}}:\cU_{\mathbf{t},\mathbf z}\to\mathbf H^N$ 
defined by $\phi_{\mathbf t, \mathbf{z}}(\Gamma)=
(\chi_1,\ldots,\chi_N)$.

We now establish the compatibility conditions for the atlas 
formed by all the pairs $(\cU_{\mathbf t, \mathbf z},\phi_{\mathbf t, \mathbf z})$, indexed on the set of valid indices $(\mathbf t, \mathbf z)$. 
Let $\Gamma\in
\cU_{\mathbf t, \mathbf z}\cap\cU_{\tilde{\mathbf t}, \tilde{\mathbf z}}$, 
where $\mathbf t=(t_1,\ldots,t_N)$, $\mathbf z=(z_1,\ldots,z_N)$, 
$\tilde{\mathbf t}=(\tilde t_1,\ldots, \tilde t_{\widetilde N})$ and 
$\tilde{\mathbf z}=(\tilde z_1,\ldots, \tilde z_{\widetilde N})$. 
It suffices to show that the $\tilde\chi_r$ depend smoothly 
on the $\chi_s$ with respect to the standard differentiable strucuture
of the separable Hilbert space $\mathbf H$.
Write  $\phi_{\mathbf t, \mathbf z}(\Gamma)=(\chi_1,\ldots, \chi_N)$ and 
$\phi_{\tilde{\mathbf t}, \tilde{\mathbf z}}(\Gamma)=
(\tilde\chi_1,\ldots, \tilde\chi_{\widetilde N})$. 
If $\cU_{z_r}\cap\cU_{\tilde z_s}\neq \emptyset$ 
and $J_{r,s}=[t_r,t_{r+1}]\cap[\tilde t_s, \tilde t_{s+1}]\neq\emptyset$, 
write the $LU$ decompositions of $z_r^{-1}\Gamma(t)$ and 
$\tilde z_s^{-1}\Gamma(t)$ as  
\[\Gamma(t)=z_r\Gamma_L(t)R^{-1}_r(t)
=\tilde z_s\widetilde\Gamma_L(t)\widetilde R^{-1}_s(t),
\quad t\in J_{r,s}, \] 
where $\Gamma_L(t),\widetilde\Gamma_L(t)\in\Lo^1_{n+1}$ and  
$R_r(t), \widetilde R_s(t)\in\Up^{+}_{n+1}$.
Equations \ref{equation:betas} imply
\[\tilde\beta_{s,k}(t)=
\dfrac{\widetilde R(t)_{k,k}}{\widetilde R(t)_{k+1,k+1}}
\dfrac{R(t)_{k+1,k+1}}{R(t)_{k,k}}\beta_{r,k}(t).\]
Now, the entries of both $\widetilde R(t)$ and $R(t)$ 
are rational functions in the entries of $\Gamma(t)$, 
while $\Gamma(t)$ can be obtained from the 
functions $\beta_{r,k}$ using Equations \ref{equation:explicitGamma}.
The desired smoothness of 
$(\tilde\chi_{1},\cdots,\tilde\chi_{\widetilde N})
\in\mathbf H^{\widetilde{N}}$ with respect to 
$(\chi_1,\cdots,\chi_N)\in\mathbf H^N$ now follows from 
these considerations plus Equations \ref{equation:chis}. 

We now consider the the \emph{monodromy map} 
$\mu_{z_0}:\cL_n^{[L^2]}(z_0;\cdot)\to\Spin_{n+1}$ 
given by $\mu_{z_0}(\Gamma)=\Gamma(1)$ 
and, for each element $z_1\in\Spin_{n+1}$, the 
\emph{monodromy subspaces}  
$\cL_n^{[L^2]}(z_0;z_1)=\mu_{z_0}^{-1}[\{z_1\}]$ and 
$\cL_n^{[C^n]}(z_0;z_1)=\cL_n^{[L^2]}(z_0;z_1)\cap\cL_n^{[C^n]}(z_0;\cdot)$. 
The reader might wish to compare $\mu_{z_0}$ with 
the monodromy map studied in \cite{Burghelea-Saldanha-Tomei2}.

Continuous dependence on parameters for 
the IVP in Equation \ref{equation:ivp} implies the continuity 
of $\mu_{z_0}$. Smooth dependence on parameters 
already implies the smoothness of $\mu_{z_0}$ restricted to 
$\mathbf{B}=\cL_n^{[C^n]}(z_0;\cdot)$, but we do not have 
such a general result for IVPs in the Carath\'eodory sense. 
The differentiable structure for $\cL_n^{[L^2]}(z_0;\cdot)$ 
provided by the charts 
$(\cU_{\mathbf t, \mathbf z}, \phi_{\mathbf t, \mathbf z})$ 
was adopted precisely to address the smoothness of $\mu_{z_0}$.

The next result shows that the monodromy subspaces are 
closed embedded submanifolds of codimension $m=n(n+1)/2$.

\begin{lemma}
\label{lemma:submersion}
The monodromy map $\mu_{z_0}:\cL_n^{[L^2]}(z_0;\cdot)\to\Spin_{n+1}$ 
is a surjective smooth submersion.
\end{lemma}

\begin{proof}
The local expression for $\mu_{z_0}$ with respect to 
each chart $(\cU_{\mathbf t, \mathbf z},\phi_{\mathbf t, \mathbf z})$ is a smooth function 
$\mu_{z_0}\circ\phi_{\mathbf t, \mathbf z}^{-1}:\mathbf H^N\to\Spin_{n+1}$ of the coordinates $(\chi_1,\cdots,\chi_N)$ by Equations 
\ref{equation:explicitGamma} and \ref{equation:chis}. 
To see that the derivative $D\mu_{z_0}(\Gamma)$ is surjective, 
assume $\Gamma\in\cU_{\mathbf t, \mathbf z}$. 
By strict convexity and Proposition \ref{prop:convex}, we have 
$\Gamma[(t_N,1]]\subset\Gamma(t_N)\Bru_{\acute\eta}$.
Let $\Gamma_0$ and $\Gamma_N$ be the linear reparametrizations 
on the unit interval $[0,1]$ of the restrictions $\Gamma|_{[0,t_N]}$ 
and $\Gamma|_{[t_N,1]}$, respectively.
Use Remark \ref{rem:projtrans} to produce a map  
$\psi_N:\Gamma(t_N)\Bru_{\acute\eta}\to\cL^{[L^2]}_n(\Gamma(t_N);\cdot)$ 
(given by a projective transformations of $\Gamma_N$) taking $\Gamma(1)$ to
$\Gamma_N$ and $z\in\Gamma(t_N)\Bru_{\acute\eta}$ to 
a convex curve $\psi_N(z)\in\cL_n(\Gamma(t_N);z)$.
A straightforward computation verifies that $\psi_N$ is smooth. 
Now, define the map 
$\psi:\Gamma(t_N)\Bru_{\acute\eta}\to\cL^{[L^2]}_n(z_0;\cdot)$ 
by the concatenation $\psi(z)=\Gamma_0\ast\psi_N(z)$; 
The map $\psi$ is also smooth, by construction.  
Thus, the composition $\mu_{z_0}\circ\psi$ is the identity map 
of $\Gamma(t_N)\Bru_{\acute\eta}\subset\Spin_{n+1}$. 
In particular, $\mu_{z_0}$ is an open map. 
Surjectivity of the map $\mu_{z_0}$ is well known
\cite{Saldanha-Shapiro} but we present a short proof.
It suffices to prove that 
the image of $\mu_{z_0}$ is a closed subset of the 
connected space $\Spin_{n+1}$. 
In fact, let $z\in\Spin_{n+1}$ and assume  
$\Gamma_0(1)\in\cU_z=z\grave\eta\Bru_{\acute\eta}$ 
for some $\Gamma_0\in\cL^{[L^2]}_n(z_0;\cdot)$. 
Use Lemma \ref{lemma:convex1} 
to obtain a convex arc $\Gamma_1\in\cL^{[L^2]}_n(\Gamma_0(1);z\acute\eta)$; 
let $\Gamma_2\in\cL^{[L^2]}_n(z\acute\eta; z)$ 
$\Gamma_2(t)=z\acute\eta\exp\left(t\frac{7\pi}{2}\fh\right)$. 
We have $z=\mu_{z_0}(\Gamma_0\ast\Gamma_1\ast\Gamma_2)$
(notice the similarity between this proof
and the add-loop construction in \cite{Saldanha-Shapiro}).
\end{proof}

\begin{rem}
\label{rem:nofibration}
A natural question would be if $\mu_{z_0}$ 
qualifies as some sort of fibration. 
The reader of course knows that the spaces $\cL_n(z)$ exhibit 
different homotopy types as $z$ ranges over $\Spin_{n+1}$
\cite{Little, Khesin-Shapiro2, Saldanha3, Saldanha-Shapiro,
Shapiro-Shapiro, Shapiro}. 
In fact, $\mu_{1}$ is not even a Serre fibration, 
since it lacks the homotopy lifting property 
for polyhedra (see \cite{Khesin-Shapiro2, Saldanha1}).
\end{rem}

\begin{prop}
\label{prop:spaces}
For all $z_0,z_1\in\Spin_{n+1}$ we have that:
\begin{enumerate}
\item\label{item:prop:spaces:submanifold}{the monodromy subspace 
$\cL^{[L^2]}_n(z_0;z_1)$ is a closed embedded submanifold of 
$\cL^{[L^2]}_n(z_0;\cdot)$ of codimension $m=n(n+1)/2$;}
\item\label{item:prop:spaces:homotopyequivalence}{the inclusion map 
$i:(\cL^{[C^n]}_n(z_0;\cdot),\cL^{[C^n]}_n(z_0;z_1))
\hookrightarrow(\cL^{[L^2]}_n(z_0;\cdot),\cL^{[L^2]}_n(z_0;z_1))$ 
is a homotopy equivalence of pairs.}
\item\label{item:prop:spaces:homeomorphism}{There are homeomorphisms 
$\cL^{[C^n]}_n(z_0;\cdot)\approx \cL^{[L^2]}_n(z_0;\cdot)$ and 
$\cL^{[C^n]}_n(z_0;z_1)\approx \cL^{[L^2]}_n(z_0;z_1)$} 
each homotopic to the corresponding inclusion map. 
\end{enumerate}
\end{prop}

\begin{proof}
Item \ref{item:prop:spaces:submanifold} follows from Lemma \ref{lemma:submersion} 
and the Regular Value Theorem applied to $\mu_{z_0}$.  
Item \ref{item:prop:spaces:homotopyequivalence} follows from 
Item \ref{item:prop:spaces:submanifold} and Fact \ref{fact:BST}. 
Item \ref{item:prop:spaces:homeomorphism} 
follows from Item \ref{item:prop:spaces:homotopyequivalence} and Fact \ref{fact:BH}.
\end{proof}

The previous result warrants us the right to drop the superscripts 
$[L^{2}]$ and $[C^{n}]$ and to adopt either definition 
of $\cL_n(z_0;z_1)$ depending on the purpose at hand. 


\bibliographystyle{apalike}

\end{document}